\documentclass[12pt,a4paper]{amsart}
\usepackage[utf8]{inputenc}
\usepackage[T1]{fontenc}
\usepackage{lmodern}
\usepackage[english]{babel}
\usepackage{amsmath,mathtools,amsthm,amssymb,amsfonts}
\usepackage{color}
\usepackage{geometry}
\usepackage{mathrsfs}
\usepackage{graphicx}
\usepackage{enumitem}
\usepackage{tikz} 
\usepackage{pgfplots}
\usepackage{esvect}
\usepackage{pdfpages}
\usetikzlibrary{shapes,arrows}
\usetikzlibrary{calc}
\usetikzlibrary{shapes.geometric}
\usetikzlibrary{shapes.arrows}
\usetikzlibrary{arrows.meta}
\usetikzlibrary{decorations.markings}
\usetikzlibrary{fit}
\usetikzlibrary{patterns}
\usetikzlibrary{hobby}
\usepgfplotslibrary{patchplots}
\pgfplotsset{compat=1.11}
\usepackage{hyperref}
\usepackage{nicefrac}
\usepackage{cancel}
\usepackage{parskip} 
\usepackage{enumitem}

\theoremstyle{plain}
\newtheorem{theorem}{Theorem}[section]
\newtheorem{corollary}[theorem]{Corollary}

\newtheorem{proposition}[theorem]{Proposition}
\newtheorem{lemma}[theorem]{Lemma}
\theoremstyle{definition}
\newtheorem{definition}[theorem]{Definition}
\newtheorem{remark}[theorem]{Remark}

\newtheorem*{theorem*}{Theorem}
\newtheorem*{definition*}{Definition}
\newtheorem*{corollary*}{Corollary}
\newtheorem*{remark*}{Remark}
\newtheorem*{thm*}{Theorem}
\newtheorem*{conjecture*}{Conjecture}

\newcommand{\definedas}{\mathrel{\raise.095ex\hbox{\rm :}\mkern-5.2mu=}}
\newcommand{\asdefined}{\mathrel{=\mkern-5.2mu\raise.095ex\hbox{\rm :}}}

\newcommand{\diver}{\operatorname{div}}

\newcounter{flabelcounter}
\allowdisplaybreaks

\title[]{Area-Constrained Willmore Surfaces in Lorentzian Manifolds:
Curvature Inequalities and the Hawking Quasi-Local Energy}
\author[]{Alejandro Pe\~nuela Diaz}
\address{University of Rostock, Institute of Mathematics, Ulmenstraße 69, 18057 Rostock, Germany}
\email{alejandro.diaz@uni-rostock.de}

\begin{document}
\begin{abstract}
We develop an area-constrained Willmore theory for closed spacelike surfaces in general four-dimensional Lorentzian manifolds. In this
codimension-two setting, the Euler--Lagrange equation is naturally
normal-bundle-valued, and we introduce a directional formulation in
which criticality is imposed along a prescribed normal direction.
We identify a geometric class of
connection-compatible directions, characterized by the interaction
between the variation of the critical direction and the normal
connection.

Our main result is the sharp curvature inequality $
2\lambda|\Sigma|
+
\int_\Sigma |\vec H|^2\,d\mu
\leq
8\pi\chi(\Sigma) $
for area-constrained Willmore surfaces critical in a
connection-compatible spacelike direction, under the corresponding
Einstein-tensor sign condition. Here $\chi(\Sigma)$ is the Euler characteristic and $\lambda$ is the Lagrange multiplier.   Under the dominant energy condition and $\lambda\geq0$,  this yields $
\int_\Sigma|\vec H|^2\,d\mu\leq16\pi$. Equality is rigid: the surface is a round sphere, and every compact mean-convex spacelike filling has maximal Cauchy development isometric
to a standard causal diamond in Minkowski spacetime. We also prove an analogous sharp inequality for surfaces critical in a single null
direction.

These curvature inequalities yield positivity and rigidity  for
the Hawking quasi-local energy. We further establish infinitesimal
monotonicity properties and characterize unconstrained Hawking-energy stationarity, showing that it corresponds to a  restrictive
class of directionally area-constrained Willmore surfaces. We also give a hypersurface formulation that recovers the previously studied
Hawking-surface problem as criticality in the spatial normal direction.

Finally, we give explicit examples including coordinate spheres in
spherical symmetry and Kerr spacetime, and construct families of
nonspherical Clifford tori in  FLRW spacetimes. These examples
illustrate the distinction between directional and full criticality and
the necessity of the sign assumptions in the rigidity theory.
\end{abstract}
\maketitle
\section{Introduction and Results} 
The Willmore functional is one of the fundamental variational energies in surface geometry. For an immersed surface \(\Sigma\) with mean curvature $H$ in a Riemannian \(3\)-manifold, it is commonly given by
\[ \int_\Sigma H^2\,d\mu, \]
and its critical points under an area constraint are the area-constrained Willmore surfaces. Their analytic and geometric properties have been studied extensively, including the existence of minimizers \cite{bauer2003existence,simon1993existence}, weak formulations and conservation laws \cite{riviere2008analysis}, and the resolution of the Willmore conjecture \cite{marques2014min}. The theory has also expanded in several directions, including surfaces in curved ambient manifolds \cite{lamm2010small,lamm2013minimizers,mondino2010some}, higher-codimension and conformal submanifold geometry \cite{chen1974some}, and more general ambient spaces; see also \cite{lan2026analysis} for a recent review of analytic developments and higher-dimensional generalizations.

By contrast, Lorentzian Willmore theory remains considerably less
developed. The existing literature has focused primarily on the
conformally invariant Willmore functional associated with the
trace-free second fundamental form, $\int_\Sigma |\mathring B|^2 d\mu$, mostly in the setting of
Lorentzian space forms. For instance, Alías and Palmer~\cite{allas1996conformal}
developed the codimension-one theory for spacelike surfaces in
\(3\)-dimensional Lorentzian space forms, while Ma and
Wang~\cite{ma2008spacelike} studied spacelike Willmore surfaces in
\(4\)-dimensional Lorentzian space forms from the viewpoint of
conformal surface geometry. In a space form, the Gauss equation and
Gauss--Bonnet relate this conformal functional to the
mean-curvature-squared functional by an area term and a topological
term. In a general curved Lorentzian manifold, however, these define
genuinely different variational problems. To the best of our
knowledge, a general variational theory for the Willmore functional on spacelike surfaces in arbitrary \(4\)-dimensional Lorentzian
manifolds has not previously been developed.

There is a particularly compelling reason to study this 
variational problem in Lorentzian geometry. Four-dimensional
Lorentzian manifolds are the natural geometric setting of general
relativity, and for a spacelike surface the squared norm of the
mean-curvature vector is determined by the two null expansions:
$$|\vec H|^2=\langle\vec H,\vec H\rangle=-\theta_\ell\theta_k.
$$
These expansions describe the convergence and divergence of the
ingoing and outgoing null congruences orthogonal to the surface.
Moreover, their product is precisely the curvature quantity entering
the Hawking quasi-local energy. Thus the mean-curvature functional,
besides its intrinsic variational interest, has a direct physical
interpretation through the geometry of null directions and quasi-local
gravitational energy.

Motivated by these considerations, we develop the area-constrained variational theory of the mean-curvature functional for closed spacelike surfaces in general $4$-dimensional Lorentzian manifolds. Let $\Sigma$ be a closed spacelike surface with mean curvature vector $\vec H$. We consider the Lorentzian Willmore functional \begin{equation} \label{Willmorefunc} \mathcal W(\Sigma) = \int_\Sigma |\vec H|^2\,d\mu, \end{equation} where $ |\vec H|^2 $ may be positive, zero, or negative.

The passage to Lorentzian geometry is not merely a formal change of signature. A spacelike surface in a $4$-dimensional Lorentzian manifold has a rank-two normal bundle with Lorentzian signature, and hence no distinguished unit normal. Accordingly, the full Euler--Lagrange equation is naturally normal-bundle-valued rather than scalar. Equivalently, after choosing an adapted null frame, full criticality produces a coupled pair of equations in the two null directions. Moreover, the Lorentzian squared norm of the mean-curvature vector may have either sign, so the associated functional is not positive in general.

We study both the full area-constrained variational problem and its directional counterparts. Full criticality requires stationarity under all area-preserving normal variations. More generally, given a prescribed normal direction $V$, we require stationarity only under area-preserving variations of the form $\alpha V$. This directional problem is intrinsic to codimension-two Lorentzian geometry: different normal directions retain different components of the full Euler--Lagrange system, and their causal character leads to distinct variational phenomena. Variations restricted to a fixed spacelike hypersurface arise naturally within this framework by selecting the spatial normal direction.

The Riemannian theory provides an important guide for the Lorentzian problem, both at the level of the Euler--Lagrange equation and through the sharp curvature estimates and rigidity properties of its critical points. For instance, area-constrained Willmore surfaces in Riemannian $3$-manifolds satisfy
\[
0=\lambda H+\Delta_\Sigma H
  +H|\mathring B|^2+H\operatorname{Ric}(\nu,\nu),
\]
and, under suitable ambient curvature assumptions, obey the sharp estimate
\[
\int_\Sigma H^2\,d\mu\leq16\pi.
\]
The equality case is rigid: in the corresponding settings the surface
is round and the enclosed geometry is flat
\cite{willflat,diaz2025rigidity}. Analogous results hold in several
related ambient geometries \cite{diaz2025rigidity}. This raises the
broader question of which features of this variational and rigidity
theory persist for directional criticality in Lorentzian geometry.

Let $\vec H^\star$ denote the conjugate of the mean-curvature vector in the Lorentzian normal bundle. When $\vec H$ is spacelike, every normal direction different from the line generated by $\vec H^\star$ can be written, up to scale, as
\[
V\sim\vec H+\beta\vec H^\star.
\]
The resulting directional curvature identity contains a normal-connection
term coupled to the variation of $\beta$. This leads us to introduce the notion of a \emph{connection-compatible direction}, an intrinsic integral condition that gives precisely the sign needed to recover the sharp Riemannian-type estimate.

We show that this condition is sufficient to recover a sharp Riemannian-type curvature estimate. If $\Sigma$ is area-constrained Willmore in a connection-compatible spacelike direction  and the corresponding Einstein-tensor sign condition holds,  then
\[
2\lambda|\Sigma|
+\int_\Sigma|\vec H|^2\,d\mu
\leq8\pi\chi(\Sigma),
\]
where $\chi(\Sigma)$ denotes the Euler characteristic and $\lambda$ is
the area-constraint multiplier. Under the dominant energy condition and
$\lambda\geq0$, this yields
\[
\int_\Sigma|\vec H|^2\,d\mu\leq16\pi.
\]
The equality case is rigid and, under the dominant energy condition, leads to roundness and Minkowski causal-diamond rigidity. We also obtain a sharp analogue for criticality in a single null direction, while the timelike mean-curvature regime has a mixed sign structure and yields weaker estimates.

A major consequence of this theory concerns the Hawking quasi-local
energy
\begin{equation}\label{Hawkingen}
\mathcal E_H(\Sigma)
=
\sqrt{\frac{|\Sigma|}{16\pi}}
\left(
1-\frac{1}{16\pi}
\int_\Sigma|\vec H|^2\,d\mu
\right).
\end{equation}
Under the dominant energy condition, the preceding estimates yield
positivity and rigidity results for $\mathcal E_H$. We also obtain
characterizations of Hawking stationarity and infinitesimal
monotonicity.

Another natural source of directional critical surfaces comes from
variations restricted to a fixed spacelike hypersurface. The
corresponding critical points, studied in
\cite{Alex,diaz2023local,diaz2025rigidity,penuelafol}, are precisely directional area-constrained Willmore surfaces in the spatial normal direction. Thus the hypersurface-restricted theory appears naturally as a distinguished sector of the Lorentzian problem.

Sharp curvature inequalities and rigidity results for these
hypersurface-restricted critical surfaces were previously obtained under additional auxiliary integral assumptions \cite{diaz2025rigidity,penuelafol}. The present formulation clarifies the geometric origin of these conditions, identifies the  terms responsible for the difficulty and provides connection-compatibility as a geometric mechanism for controlling them. It also shows that the earlier auxiliary assumptions are not necessary for the sharp conclusions and can be unnecessarily restrictive. Moreover, the existence results, including local and asymptotic foliations \cite{diaz2023local,penuelafol}, now provide nontrivial families of
directional area-constrained Willmore surfaces in spacelike directions.

The theory is illustrated by explicit families in spherical symmetry,
Kerr spacetime, and FLRW spacetimes. These examples exhibit substantial
differences between directional and full criticality and show, in
particular, that the sign assumption on the Lagrange multiplier in the
rigidity theorem is essential.

\subsection{Main definitions and results}
Throughout this subsection, $(M,g)$ denotes a time-oriented
$4$-dimensional Lorentzian manifold and
$\Sigma\subset M$ a smooth, closed, connected spacelike surface. We begin with the two variational notions used throughout the paper.

\begin{definition}[Area-constrained Willmore criticality]
\label{defwill}
We say that $\Sigma$ is an \emph{area-constrained Willmore surface} if
it is critical for the spacetime Willmore functional
\eqref{Willmorefunc} under all area-preserving normal variations.

More generally, if $V\in\Gamma(N\Sigma)$ is a nowhere-vanishing normal
vector field, we say that $\Sigma$ is
\emph{area-constrained Willmore in the direction $V$} if it is critical
under all area-preserving normal variations of the form $\alpha V$,
with $\alpha\in C^\infty(\Sigma)$.
\end{definition}

The invariant Euler--Lagrange equation for full criticality and its
directional formulation are derived in Section~\ref{secsetup}. We recall
here the scalar form that will be used to state the main results.

Let $\vec H$ denote the mean curvature vector of $\Sigma$, and let $\vec H^\star$ denote its conjugate vector in the Lorentzian normal bundle. Thus \[ \langle\vec H,\vec H^\star\rangle=0, \qquad \langle\vec H^\star,\vec H^\star\rangle = -\langle\vec H,\vec H\rangle. \]

Directional criticality depends only on the normal line generated by
$V$.  Suppose that $\vec H$ is non-null and that 
$\langle V,\vec H\rangle\neq0$. Then the  line generated by $V$ admits the unique representation
\[
V\sim\vec H+\beta\vec H^\star,
\qquad
\beta=-\frac{\langle V,\vec H^\star\rangle}
{\langle V,\vec H\rangle}.
\]
The directional Euler--Lagrange equation can be expressed in terms of
two canonical scalar quantities, denoted by $W$ and $W^\star$,
associated respectively with the directions $\vec H$ and
$\vec H^\star$;
see Section~\ref{secsetup}. In these terms, directional
area-constrained Willmore criticality is equivalent to
\begin{equation}
\label{eq:main-directional-Willmore}
W+\beta W^\star
=
\lambda|\vec H|^2
\end{equation}
for some constant $\lambda\in\mathbb R$.

The orthogonal direction $\vec H^\star$ is not
represented by a finite value of $\beta$; criticality in this direction
is simply $
W^\star=0$. Consequently, whenever $\vec H$ is non-null, full area-constrained
Willmore criticality is equivalent to
\[
W=\lambda|\vec H|^2,
\qquad
W^\star=0.
\]
When $\Sigma$ lies in a totally geodesic spacelike hypersurface, this
system reduces to the classical area-constrained Willmore equation;
see Section~\ref{Willcla}.

Equation~\eqref{eq:main-directional-Willmore} is the basic
Euler--Lagrange equation underlying the directional theory developed
below. In the spacelike mean-curvature regime, the resulting defect
identity has a favorable sign structure except for a term coupling the
variation of the critical direction to the normal connection. This
motivates the following notion.

\begin{definition}[Connection-compatible direction]
\label{def:connection-compatible0}
Let $|\vec H|^2>0$, and let $
V\sim \vec H+\beta\vec H^\star$,
be a normal direction. We say that $V$ is
\emph{connection-compatible} if
\[
\int_\Sigma
\frac{
\langle\nabla_{\nabla^h\beta}\vec H,\vec H^\star\rangle
}{|\vec H|^2}\,d\mu
\leq0.
\]
\end{definition}
In particular, every direction for which $\beta$ is constant is connection-compatible, with equality. A second important class is provided by arbitrary spacelike directions on the time-flat surfaces introduced by Bray and Jauregui~\cite{bray2015time}; see also Remark~\ref{rem:connection-compatible-Hawking}. Connection compatibility also yields Hawking-energy monotonicity along uniformly area-expanding spacelike flows \cite{bray2007generalized,bray2015time,bray2016time}; see Proposition~\ref{prop:Hawking-monotonicity-connection-compatible}.

The basic analytic tool is the exact directional curvature identity of
Proposition~\ref{prop:exact-directional-defect}. For arbitrary spacelike
directions, Theorem~\ref{thm:general-spacelike-direction} yields a curvature
estimate with an additional correction term depending on $\nabla^h\beta$.
For connection-compatible directions this correction is absent, and one
obtains the following sharp estimate.

\noindent\textbf{Theorem \ref{thm:connection-compatible-spacelike}.}  \emph{Let $(M,g)$ be a time-oriented $4$-dimensional Lorentzian manifold, and let $\Sigma\subset M$ be a closed, connected spacelike surface with $|\vec H|^2>0$. Suppose that $\Sigma$ is area-constrained Willmore in a connection-compatible spacelike direction generated by $ \vec H+\beta\vec H^\star$,  with associated Lagrange multiplier $\lambda$. Assume that $\operatorname{Ein} \left( \vec H^\star,\vec H^\star+\beta\vec H \right) \geq0 $ along $\Sigma$. Then \begin{equation}\label{inequtop} 2\lambda|\Sigma| + \int_\Sigma|\vec H|^2\,d\mu \leq 8\pi\chi(\Sigma). \end{equation} 
If equality holds, then $\operatorname{Ein}(\vec H^\star,\vec H^\star+\beta\vec H)=0$,  $\Sigma$ is totally umbilic, it has constant scalar curvature given by $ \mathrm{Sc}^{\Sigma} = \frac12|\vec H|^2 + \lambda$ and its mean curvature vector is parallel ($(\nabla \vec H)^\perp=0$). } 

The same sharp inequality extends to null critical directions. More
precisely, Theorem~\ref{thm:null-direction-nontimelike} shows that if
$\Sigma$ is area-constrained Willmore in a nowhere-vanishing null
direction $V$, with $|\vec H|^2\geq0$, and the corresponding null
energy conditions hold, then (\ref{inequtop}) holds.
The equality case is weaker: null criticality
controls only the geometry associated with the prescribed null
direction and does not, in general, force the mean curvature vector to
be parallel. In particular, the rigidity conclusion involves the
vanishing of the corresponding null shear and curvature components
rather than both null shears simultaneously.

The timelike mean-curvature regime behaves differently. When
$|\vec H|^2<0$, the sign structure underlying the preceding estimates is less convenient. Proposition~\ref{prop:timelike-null-direction-bounds} still yields two-sided geometric bounds for null critical directions, but the mixed signs prevent the sharp topological and rigidity conclusions available in the non-timelike regime.

Under the dominant energy condition and the sign assumption
$\lambda\geq0$, the spacelike estimate becomes
particularly rigid, leading to our principal Minkowski
rigidity theorem.

\noindent\textbf{Theorem \ref{thm:directional-Minkowski-rigidity}.}  \emph{Let $(M,g)$ be a time-oriented $4$-dimensional Lorentzian manifold satisfying the dominant energy condition, and let $\Sigma\subset M$ be a closed, connected spacelike surface with spacelike mean curvature vector,  $|\vec H|^2>0$.  Suppose that $\Sigma$ is area-constrained Willmore in a
connection-compatible spacelike normal direction 
with associated Lagrange multiplier $\lambda\geq0$. Then   $\Sigma\simeq\mathbb S^2$ or $\Sigma\simeq\mathbb{RP}^2$, and 
    \begin{equation*}
         \int_\Sigma |\vec{H}|^2 \, d\mu \leq 16 \pi.
    \end{equation*}
   If equality holds, then $\Sigma$ is isometric to a round
sphere.  Moreover, for every compact mean-convex spacelike hypersurface \(\Omega\subset M\)
with boundary \(\partial\Omega=\Sigma\), the induced initial data on \(\Omega\)
embeds isometrically into Minkowski spacetime, and the maximal Cauchy development of this initial data is isometric to a standard causal
diamond in Minkowski spacetime.}

Related sharp curvature bounds and Minkowski rigidity for surfaces with spacetime constant mean curvature (STCMC) were obtained in \cite{diaz2026curvature}.  The relation between the two frameworks is discussed further in Remark~\ref{rem:W-STCMC}.

\noindent\textbf{Hawking energy: positivity, stationarity, and monotonicity.}
The preceding curvature inequalities have direct consequences for the
Hawking energy. Under the hypotheses of
Theorem~\ref{thm:directional-Minkowski-rigidity}, they imply
\[
\mathcal E_H(\Sigma)\geq0,
\]
and vanishing of the Hawking energy forces the equality case and hence
the corresponding Minkowski rigidity conclusion; see
Corollary~\ref{cor:directional-Hawking-positivity}. Thus directional
area-constrained Willmore criticality provides a variational mechanism
for Hawking-energy positivity and rigidity in a general Lorentzian
setting.

There is also a direct variational relation between the two functionals.
Since the area is fixed under area-preserving variations, the
area-constrained variational problems for the Willmore functional and
the Hawking energy are equivalent. By contrast, unconstrained
Hawking-energy stationarity is considerably more rigid.
Proposition~\ref{prop:Hawking-stationary-directional} shows that, for any
prescribed normal direction $V$,
\[
\delta_{\alpha V}\mathcal E_H(\Sigma)=0
\qquad\text{for every }\alpha\in C^\infty(\Sigma)
\]
if and only if $\Sigma$ is area-constrained Willmore in the direction
$V$ with the distinguished multiplier $
\lambda
=
\frac{4\mathcal E_H(\Sigma)}{r_\Sigma^3}$.
For spacelike connection-compatible directions this yields the following
rigidity characterization.

\noindent\textbf{Theorem \ref{thm:Hawking-stationary-rigidity}} (Rigidity of Hawking stationarity in spacelike directions). \emph{Let $(M,g)$ be a time-oriented $4$-dimensional Lorentzian manifold, and let $\Sigma\subset M$ be a closed, connected spacelike surface with  $|\vec H|^2>0$.  Let $V=\vec H+\beta\vec H^\star$ be a connection-compatible spacelike normal direction.   Assume that $ \operatorname{Ein} ( \vec H^\star,\vec H^\star+\beta\vec H ) \geq0 $ along $\Sigma$. Then  \[ \delta_{\alpha V}\mathcal E_H(\Sigma)=0 \qquad \text{for every }\alpha\in C^\infty(\Sigma), \]  if and only if $\Sigma$ is totally umbilic, isometric to a round sphere, $ (\nabla \vec H)^\perp=0$ and $\operatorname{Ein} ( \vec H^\star,\vec H^\star+\beta\vec H ) =0$.  In this case $\Sigma$ is area-constrained Willmore in the direction $V$ with $ \lambda =\frac{4\mathcal E_H(\Sigma)}{r_\Sigma^3}$ and $ \operatorname{Sc}^{\Sigma} = \frac12|\vec H|^2+\lambda$. }

 In particular, unconstrained criticality is extremely restrictive. This contrasts with the area-constrained Willmore problem, which admits a substantially richer class of solutions.

\textbf{Infinitesimal monotonicity for the Hawking energy.}  Full area-constrained Willmore criticality yields an infinitesimal
monotonicity principle in the entire normal bundle: the Hawking energy
cannot decrease along any infinitesimally area-nondecreasing normal
variation.

 \noindent\textbf{Corollary \ref{cor:monotofull}} (Infinitesimal monotonicity under full normal criticality). \emph{Let $(M,g)$ satisfy the dominant energy condition, and let
$\Sigma\subset M$ be a closed, connected area-constrained Willmore
surface with $|\vec H|^2\geq0$. Assume that at least one null expansion
is nowhere vanishing. Then, for every normal variation $V$ with $ \delta_V|\Sigma|\geq0 $  \[ \delta_V\mathcal E_H(\Sigma)\geq0. \] If equality holds for a strictly area-increasing variation, $|\vec H|^2>0$ and $\lambda=0$, then $\Sigma$ is a round sphere, and the same Minkowskian rigidity conclusion as in Theorem~\ref{thm:directional-Minkowski-rigidity} holds.}

Direction-specific analogues are proved in
Theorem~\ref{thm:Hawking-monotonicity-directional}: directional Willmore criticality
yields infinitesimal Hawking-energy monotonicity along
area-nondecreasing variations in the same critical direction.

\noindent\textbf{Hypersurface formulation.}
In Section~\ref{sec:hypersurface} we express the directional
Euler--Lagrange equations in a form adapted to a spacelike hypersurface
$(M^3,g,K)$. This gives, in particular, a simple
interpretation of the hypersurface-restricted variational problem
studied in
\cite{Alex,diaz2023local,diaz2025rigidity,penuelafol}:
its critical points, also known as \emph{Hawking surfaces}, are
precisely the surfaces that are area-constrained Willmore in the
spatial normal direction $\nu$ of $\Sigma$ within $M^3$.

In the spacelike mean-curvature regime, the spatial normal direction corresponds to $\beta=\frac{P}{H}$, and connection-compatibility becomes an explicit condition on the initial
data. Under the dominant energy condition, this yields the sharp estimate
of Corollary~\ref{cor:Hawking-surfaces-directional-estimate} and the Hawking-energy monotonicity result of
Corollary~\ref{cor:Hawking-foliation-monotonicity}. We also compare connection-compatibility with the auxiliary integral conditions introduced
in \cite{penuelafol}, thereby clarifying the geometric role of assumptions that originally arose from the hypersurface computation.

\noindent\textbf{Examples and sharpness.}
Section~\ref{examplsec} develops several explicit families illustrating
the distinction between directional and full area-constrained Willmore
criticality.

For symmetry spheres in spherically symmetric Lorentzian manifolds, Proposition~\ref{prop:symmetry-spheres-directional-Willmore} shows that
every constant $\beta$ determines an area-constrained Willmore
direction $ \vec H+\beta\vec H^\star $ with a multiplier that generally depends on the chosen direction. These
spheres are fully area-constrained Willmore precisely when $\operatorname{Ein}(\vec H^\star,\vec H)=0$; see Corollary~\ref{cor:full-Willmore-symmetry-spheres}. In particular,
every symmetry sphere in a vacuum region is fully area-constrained Willmore, including the standard Schwarzschild coordinate spheres.

The Kerr geometry exhibits a different phenomenon. Proposition~
\ref{prop:Kerr-dual-directional} shows that every Boyer--Lindquist
coordinate sphere with $H\neq0$ is time-flat and Willmore critical in the dual mean-curvature direction $\vec H^\star$. For nonzero rotation, these spheres are not area-constrained Willmore in any finite-$\beta$
direction.

Finally, Section~\ref{clifford} provides explicit nonspherical examples. We exhibit families of Clifford tori in spatial slices of  FLRW spacetimes that  are simultaneously STCMC and area-constrained Willmore in every constant-$\beta$ direction; see Proposition~\ref{prop:directional-FLRW-Clifford}.  Among them, the minimal Clifford torus is fully area-constrained Willmore and, under the dominant energy condition, is the unique fully critical member of the family;
see Corollary~\ref{cor:minimal-Clifford-full-criticality}.

For sufficiently thin tori, Proposition~\ref{prop:instability_tori}
shows that the Lagrange multiplier is negative and that the surfaces are unstable as STCMC surfaces. These examples show that the sign condition $\lambda\geq0$ in the
directional rigidity theorem is essential for excluding higher-genus topology. Their STCMC instability also illustrates the role of the stability hypotheses in the related STCMC rigidity theory.

Taken together, these results establish a directional
area-constrained Willmore theory adapted to codimension-two Lorentzian
geometry. The framework contains the previously studied
hypersurface-restricted problem as a distinguished special case,
provides sharp curvature inequalities and rigidity for
connection-compatible directions, and yields positivity,
stationarity, and monotonicity results for the Hawking quasi-local
energy. The existing existence theory, together with the explicit
examples developed here, shows that these directional variational
classes are geometrically natural and substantially richer than the
fully critical problem.

\noindent\textbf{Organization of the paper.}
Section~\ref{secsetup} develops the geometric and variational framework
and the canonical mean-curvature gauge.
Section~\ref{sectionrigi} proves the directional curvature inequalities
and rigidity results, while Section~\ref{Hawkingensec} studies their
consequences for the Hawking energy.
Section~\ref{sec:hypersurface} gives the corresponding formulation on a spacelike hypersurface.
Finally, Section~\ref{examplsec} treats explicit examples in spherical
symmetry, Kerr spacetime, and FLRW spacetimes.

\section{Geometric Setup}\label{secsetup}

We now introduce the Lorentzian framework needed to study directional and
full area-constrained Willmore criticality for spacelike surfaces in a
$4$-dimensional Lorentzian manifold. We fix conventions, recall the basic
null geometry, and derive the null and directional Euler--Lagrange equations
used throughout the paper. The mean-curvature and dual mean-curvature
directions arise naturally as distinguished special cases. 

\subsection{Spacelike surfaces and null geometry}

Let $(M,g)$ be a $4$-dimensional spacetime and let $\Sigma\subset M$
be a smooth closed spacelike surface. We denote by
$ \Phi_\Sigma:\Sigma\hookrightarrow M $
the embedding of $\Sigma$ into $M$, and we will often identify $\Sigma$ with its image. The induced Riemannian metric on $\Sigma$ will be denoted by $h$.
At each point $p\in\Sigma$ the tangent space of the spacetime
decomposes orthogonally as $
T_pM = T_p\Sigma \oplus N_p\Sigma$. Given tangent vector fields $X,Y\in T\Sigma$, the second fundamental
form of $\Sigma$ in $(M,g)$ is defined by
\[
\chi(X,Y) := -(\nabla_X Y)^{\perp},
\]
where $\nabla$ denotes the Levi-Civita connection of $(M,g)$.
Its trace with respect to $h$ defines the mean curvature vector
\[
\vec H := \operatorname{tr}_h \chi \in N\Sigma .
\]
For a normal vector field $n\in N\Sigma$ we define the second
fundamental form in the direction $n$ by
\[
\chi_n(X,Y) = \langle \chi(X,Y),n\rangle ,
\]
and its trace $
\theta_n = \operatorname{tr}_h \chi_n$ is called the \emph{expansion along $n$}. When $n$ is null,
$\theta_n$ is referred to as the \emph{null expansion}.

Locally, we choose a pair of future-directed null normals
\(\{\ell,k\}\), normalized by $
\langle\ell,k\rangle=-2$. This normalization leaves the usual boost freedom $
\ell' = f\,\ell$, 
$k' = f^{-1}k$, for any positive function $f$ on $\Sigma$. The corresponding null expansions are defined by
\[
\theta_\ell := \operatorname{tr}_h\chi_\ell,
\qquad
\theta_k := \operatorname{tr}_h\chi_k .
\]
Associated with the choice of null frame is the normal connection
one-form
\[
s_\ell(X) = -\tfrac12\langle k,\nabla_X\ell\rangle,
\qquad X\in T\Sigma .
\]

With respect to the null frame $\{\ell,k\}$ the mean curvature vector
can be written as
\[
\vec H = -\tfrac12(\theta_k\,\ell + \theta_\ell\,k),
\]
and its squared norm satisfies $          
\langle\vec H,\vec H\rangle = -\theta_\ell\theta_k$.
Orthogonal to \(\vec H\) in the normal bundle is the conjugate mean curvature vector, defined by
\begin{equation}\label{eq:dual-mean-curvature}
    \vec H^\star := \frac12(-\theta_k\,\ell + \theta_\ell\,k).
\end{equation}
A direct computation gives
\[
\langle \vec H^\star,\vec H\rangle = 0, \qquad \langle \vec H^\star,\vec H^\star\rangle = \theta_\ell\theta_k = -\langle\vec H,\vec H\rangle.
\]
Thus, whenever \(\vec H\) is non-null, the vectors \(\vec H\) and \(\vec H^\star\) have opposite causal character. In the standard spacelike mean curvature regime (\(\langle\vec H,\vec H\rangle > 0\)), the conjugate vector \(\vec H^\star\) is strictly timelike. Normalized by its length, it defines the vector field
\begin{equation}\label{vecU}
    U := \frac{\vec H^\star}{|\vec H|}.
\end{equation}
Geometrically, the tangent space $T\Sigma$ together with the spacelike vector $\vec H$ span a three-dimensional spacelike subspace of $TM$, and $U$ is the unique future-directed unit timelike normal vector orthogonal to this subspace. Moreover, $U$ is invariant under boost transformations, so it defines a canonical observer field associated with $\Sigma$.

When $\vec H$ is instead timelike, $\vec H^\star$ becomes spacelike, and the geometric roles of the normal vectors interchange. We will systematically construct the canonical frames for both non-degenerate regimes in Section \ref{subsec:mean-curvature-gauge}.

\subsection{Variation formulas and Euler-Lagrange equations}
We first record the invariant normal-bundle formulation of the Willmore
gradient. Its tensorial form is formally analogous to the corresponding
expression in higher-codimension Riemannian Willmore theory. 
\begin{proposition}[Invariant Willmore gradient]
\label{lem:frame-invariant-willmore}
Let $\Sigma$ be a spacelike surface immersed in a Lorentzian manifold
$(M,g)$. Define
\begin{equation}
\label{eq:Willmore-gradient}
\mathfrak W
:=
\Delta^\perp\vec H
+
\operatorname{tr}_h
\bigl(\operatorname{Rm}^{M}(\vec H,\cdot)\cdot\bigr)^\perp
+
\widetilde\chi(\vec H)
-\frac12|\vec H|^2\vec H,
\end{equation}
where \(\Delta^\perp\) is the normal  Laplacian, \(\mathrm{Rm}^{M}\) is the ambient Riemann curvature tensor, and \(\tilde{\chi}(\vec{H}) = \sum_{i,j} \langle \vec{H}, \chi(e_i, e_j) \rangle \chi(e_i, e_j)\) is the Simons operator associated with the second fundamental form \(\chi\). Then, for every compactly supported normal variation $\xi$,
\begin{equation}
\label{eq:Willmore-first-variation}
\delta_\xi
\int_\Sigma|\vec H|^2\,d\mu
=
-2\int_\Sigma
\langle\mathfrak W,\xi\rangle\,d\mu.
\end{equation}
Consequently, full area-constrained Willmore criticality is equivalent to
\begin{equation}
\label{eq:abstract_willmore}
\mathfrak W+\frac{\lambda}{2}\vec H=0
\end{equation}
for some constant $\lambda\in\mathbb R$.
\end{proposition}
The proof is given in Appendix~\ref{app:invariant-willmore}. More generally,
directional area-constrained criticality along a prescribed normal field
$V$ is governed by the component of $\mathfrak W$ in the direction $V$, 
\begin{equation}
\label{eq:directional-invariant}
2\langle\mathfrak W,V\rangle
=
-\lambda\langle\vec H,V\rangle.
\end{equation}
In particular, if $\langle\vec H,V\rangle\equiv0$, every variation $\alpha V$ is
infinitesimally area-preserving. In this case directional criticality
reduces to $
\langle\mathfrak W,V\rangle=0$,
and the Lagrange multiplier plays no role.

Thus the full Euler–Lagrange equation is intrinsically normal-bundle-valued,
whereas directional criticality selects a scalar component of the
Willmore gradient. For the explicit null components, it is convenient
to use the standard first-variation formulas for the null expansions,
which directly display the normal connection, null second fundamental
forms, and Einstein tensor.

Let $\xi$ be a normal variation vector field along $\Sigma$ and let $\varphi_\tau$ denote the associated flow, with deformed surfaces
$\Sigma_\tau=\varphi_\tau(\Sigma)$. Choosing a corresponding family
of null normals $\ell_\tau$ along $\Sigma_\tau$, we define the first
variation of the null expansion $\theta_\ell$ by $
\delta_\xi\theta_\ell
=
\left.\frac{d}{d\tau}\,
\varphi_\tau^{\ast}(\theta_{\ell_\tau})
\right|_{\tau=0}$. We decompose the variation field along the null frame as $
\xi|_\Sigma=\alpha\,\ell-\tfrac{\psi}{2}\,k$,
where $\alpha,\psi\in C^\infty(\Sigma)$.

The first variation of $\theta_\ell$ is given by
\begin{equation}\label{eq:variation-theta-l}
\begin{aligned}
\delta_\xi\theta_\ell
=& -\Delta_h\psi
+2\,s_\ell(\nabla^h\psi)
+\psi\Big(
\operatorname{div}_h s_\ell-\|s_\ell\|_h^2
+\tfrac12\theta_\ell\theta_k
+\tfrac12\mathrm{Sc}^\Sigma
-\tfrac12\operatorname{Ein}(\ell,k)
\Big)
\\
&-\alpha\Big(\operatorname{Ein}(\ell,\ell)+\|\chi_\ell\|_h^2\Big)
+\kappa_\xi\,\theta_\ell, 
\end{aligned}
\end{equation}
where  $\kappa_\xi := -\tfrac12\,\langle k,\nabla_\xi\ell_{(\xi)}\rangle
$. Interchanging the roles of $\ell$ and $k$, the variation of $\theta_k$
along $\xi=\alpha k-\tfrac{\psi}{2}\ell$ is
\begin{equation}\label{eq:variation-theta-k}
\begin{aligned}
\delta_\xi\theta_k
=& -\Delta_h\psi
-2\,s_\ell(\nabla^h\psi)
+\psi\Big(
-\operatorname{div}_h s_\ell-\|s_\ell\|_h^2
+\tfrac12\theta_\ell\theta_k
+\tfrac12\mathrm{Sc}^\Sigma
-\tfrac12\operatorname{Ein}(\ell,k)
\Big)
\\
&-\alpha\Big(\operatorname{Ein}(k,k)+\|\chi_k\|_h^2\Big)
-\kappa_\xi\,\theta_k .
\end{aligned}
\end{equation}
These variation formulas are standard; in the present conventions they follow from \cite{mars2012stability}. From these formulas we obtain the first variation of the functional
\(\int_\Sigma \theta_\ell\theta_k\,d\mu\) in the relevant directions.

 For a variation in the direction $\alpha\ell$, a direct computation using~\eqref{eq:variation-theta-l} and~\eqref{eq:variation-theta-k} yields
\begin{equation}\label{varia l}
\begin{aligned}
    \delta_{\alpha \ell} \int_\Sigma \theta_\ell \theta_k d\mu =& \int_\Sigma 2 \theta_\ell \Delta_h \alpha +4 \theta_\ell s_\ell (\nabla^h \alpha  ) + 2 \theta_\ell\alpha \Big( \|s_{\ell}\|_{h}^{2} +\diver_{h}s_{\ell}  -\frac{1}{2}\mathrm{Sc}^\Sigma
+\frac{1}{2} \operatorname{Ein}(\ell,k)\Big)  \\
&\quad-\alpha \theta_k\Big(\operatorname{Ein}(\ell,\ell)
+\|\chi_{\ell}\|_{h}^{2}\Big) d\mu\\
=&  \int_\Sigma \Big( 2  \Delta_h\theta_\ell - 2 \theta_\ell \diver_{h}s_{\ell} -4  s_\ell (\nabla^h  \theta_\ell ) +  \theta_\ell \Big(2\|s_{\ell}\|_{h}^{2}  -\mathrm{Sc}^\Sigma
+ \operatorname{Ein}(\ell,k)\Big)\\
&\quad -\theta_k\Big(\operatorname{Ein}(\ell,\ell)
+\|\chi_{\ell}\|_{h}^{2}\Big) \Big)\alpha d \mu
\end{aligned}
\end{equation}
Similarly 
\begin{equation}\label{varia k}
\begin{aligned}
    \delta_{\alpha k} \int_\Sigma \theta_\ell \theta_k d\mu 
=&  \int_\Sigma \Big( 2  \Delta_h\theta_k + 2 \theta_k \diver_{h}s_{\ell} +4  s_\ell (\nabla^h  \theta_k ) +  \theta_k \Big(2\|s_{\ell}\|_{h}^{2}  -\mathrm{Sc}^\Sigma
+ \operatorname{Ein}(\ell,k)\Big)\\
&\quad -\theta_\ell\Big(\operatorname{Ein}(k,k)
+\|\chi_{k}\|_{h}^{2}\Big) \Big)\alpha d\mu
\end{aligned}
\end{equation}
For later use, we denote the integrands appearing in
\eqref{varia l} and \eqref{varia k} by $W_\ell$ and $W_k$, respectively:
\begin{equation}\label{Wl}
\begin{aligned}
W_\ell
:=&
2\Delta_h\theta_\ell
-2\theta_\ell\operatorname{div}_{h}s_\ell
-4s_\ell(\nabla^h\theta_\ell)
+\theta_\ell
\Big(
2\|s_\ell\|_h^2-\mathrm{Sc}^\Sigma+\operatorname{Ein}(\ell,k)
\Big)
-\theta_k
\Big(
\operatorname{Ein}(\ell,\ell)+\|\chi_\ell\|_h^2
\Big),
\end{aligned}
\end{equation}
and
\begin{equation}\label{Wk}
\begin{aligned}
W_k
:=&
2\Delta_h\theta_k
+2\theta_k\operatorname{div}_{h}s_\ell
+4s_\ell(\nabla^h\theta_k)
+\theta_k
\Big(
2\|s_\ell\|_h^2-\mathrm{Sc}^\Sigma+\operatorname{Ein}(\ell,k)
\Big) 
-\theta_\ell
\Big(
\operatorname{Ein}(k,k)+\|\chi_k\|_h^2
\Big).
\end{aligned}
\end{equation}
Since $|\vec H|^2= -\theta_\ell \theta_k$, comparison with \eqref{eq:Willmore-first-variation} gives
\begin{equation}
\label{eq:null-Willmore-components}
W_\ell=2\langle\mathfrak W,\ell\rangle,
\qquad
W_k=2\langle\mathfrak W,k\rangle.
\end{equation}
Thus $W_\ell$ and $W_k$ are precisely the null components of the
invariant Willmore gradient. We will refer to them as the null
Willmore integrands.

Consider an arbitrary normal vector field $V=a\ell+bk$. By
\eqref{eq:null-Willmore-components},
\[
2\langle\mathfrak W,V\rangle
=
aW_\ell+bW_k,
\qquad
\langle\vec H,V\rangle
=
a\theta_\ell+b\theta_k.
\]
Hence the invariant directional equation
\eqref{eq:directional-invariant} becomes
\begin{equation}
\label{eq:general-direction-criticality}
aW_\ell+bW_k
=
-\lambda
\bigl(a\theta_\ell+b\theta_k\bigr).
\end{equation}
In particular, criticality in the null directions $\ell$ and $k$ is
equivalent, respectively, to $
W_\ell=-\lambda\theta_\ell$ and $W_k=-\lambda\theta_k$.

Notice that multiplication of $V$ by a nowhere-vanishing function does not
change the class of variations of the form $\alpha V$. Hence directional
criticality depends only on the normal line field determined by $V$.

\noindent\textbf{The mean-curvature and dual mean-curvature directions.}
The canonical normal vectors $\vec H$ and $\vec H^\star$ determine two
distinguished components of the Willmore gradient. We define
\[
W:=-2\langle\mathfrak W,\vec H\rangle,
\qquad
W^\star:=-2\langle\mathfrak W,\vec H^\star\rangle.
\]
Using $
\vec H=-\frac12(\theta_k\ell+\theta_\ell k)$ and $
\vec H^\star=\frac12(-\theta_k\ell+\theta_\ell k)$,
we obtain
\begin{equation}
\label{eq:W-Wstar-definitions}
W
=
\frac12\left(
\theta_kW_\ell+\theta_\ell W_k
\right),
\qquad
W^\star
=
\frac12\left(
\theta_kW_\ell-\theta_\ell W_k
\right).
\end{equation}
Consequently, criticality in the mean-curvature direction is equivalent to
\begin{equation}
\label{eq:H-direction-criticality}
W=\lambda|\vec H|^2,
\end{equation}
whereas criticality in the dual mean-curvature direction is equivalent to
\begin{equation}
\label{eq:Hstar-direction-criticality}
W^\star=0.
\end{equation}
In particular, variations in the direction $\vec H^\star$ are
infinitesimally area preserving, so constrained and unconstrained
criticality coincide in this direction.
Whenever $|\vec H|^2\neq0$, the pair
$\{\vec H,\vec H^\star\}$ spans the normal bundle. Hence full
area-constrained Willmore criticality is equivalent to the simultaneous conditions $ W=\lambda|\vec H|^2$ and  $W^\star=0$.

The canonical components satisfy $
W+W^\star=\theta_kW_\ell$ and $
W-W^\star=\theta_\ell W_k$.
Thus, wherever $\theta_\ell\theta_k\neq0$,
\[
W^\star=0
\qquad\Longleftrightarrow\qquad
\frac{W_\ell}{\theta_\ell}
=
\frac{W_k}{\theta_k}.
\]
\noindent\textbf{Geometric forms of the canonical components.}
For later use, we record more geometric expressions for $W$ and $W^\star$. Expanding
\eqref{eq:W-Wstar-definitions}, we obtain

\begin{equation}\label{defW}
\begin{aligned}
W:=&
 \theta_k  \Delta_h\theta_\ell +\theta_\ell  \Delta_h\theta_k
 -2 \theta_k s_\ell (\nabla^h  \theta_\ell )
 +2 \theta_\ell s_\ell (\nabla^h  \theta_k ) + 2 \theta_k \theta_\ell
\Big( \|s_{\ell}\|_{h}^{2}
-\frac{1}{2}\mathrm{Sc}^\Sigma
+\frac{1}{2} \operatorname{Ein}(\ell,k)\Big) \\
&\quad
-\frac{\theta_k^2}{2}\Big(\operatorname{Ein}(\ell,\ell)+\|\chi_{\ell}\|_{h}^{2}\Big) 
-\frac{\theta_\ell^2}{2}\Big(\operatorname{Ein}(k,k)+\|\chi_{k}\|_{h}^{2}\Big).
\end{aligned}
\end{equation}
Although the expanded expression \eqref{defW} is useful for direct
computations, it admits a more geometric formulation. First, using $
\vec H^\star
=
\frac12\left(
-\theta_k\ell+\theta_\ell k
\right)$,
the spacetime curvature terms satisfy
\[
-2\operatorname{Ein}(\vec H^\star,\vec H^\star)
=
\theta_k\theta_\ell\operatorname{Ein}(\ell,k)
-\frac{\theta_k^2}{2}\operatorname{Ein}(\ell,\ell)
-\frac{\theta_\ell^2}{2}\operatorname{Ein}(k,k).
\]
To rewrite the differential terms, we use the normal projection of the
ambient covariant derivative. Thus, for a normal vector field
$N\in\Gamma(N\Sigma)$ and $X\in T\Sigma$, we write
\[
(\nabla_XN)^\perp
\]
for the component of $\nabla_XN$ normal to $\Sigma$. For the null frame, this is determined entirely by the normal connection one-form \(s_\ell\):
\[
(\nabla_X \ell)^\perp = s_\ell(X)\ell, \qquad (\nabla_X  k)^\perp = -s_\ell(X)k.
\]
Applying this to the mean curvature vector \(\vec H = -\frac12(\theta_k\ell + \theta_\ell k)\), its covariant derivative in a tangent direction \(X\) expands to
\[
(\nabla_X \vec H)^\perp = -\frac{1}{2} \Big( \big(X(\theta_k) + \theta_k s_\ell(X)\big)\ell + \big(X(\theta_\ell) - \theta_\ell s_\ell(X)\big)k \Big).
\]
Summing over an orthonormal frame for \(T\Sigma\), we obtain
\begin{equation}\label{eq:norm_nabla_H}
\|(\nabla \vec{H})^\perp\|_h^2 = -\langle \nabla^h\theta_k, \nabla^h\theta_\ell\rangle - \theta_k s_\ell(\nabla^h\theta_\ell) + \theta_\ell s_\ell(\nabla^h\theta_k) + \theta_k\theta_\ell\|s_\ell\|_h^2.
\end{equation}
On the other hand, the intrinsic Laplacian of the squared mean curvature is
\begin{equation}\label{eq:laplacian_H_squared}
-\Delta_h|\vec H|^2 = \Delta_h(\theta_\ell\theta_k) = \theta_k\Delta_h\theta_\ell + \theta_\ell\Delta_h\theta_k + 2\langle \nabla^h\theta_\ell, \nabla^h\theta_k\rangle.
\end{equation}
Adding \(2\|(\nabla \vec H)^\perp\|_h^2\) to \(-\Delta_h|\vec H|^2\) exactly cancels the inner product of the gradients and recovers the first five differential terms of \(W\) in \eqref{defW}. Combining these identities with \eqref{defW} gives
\begin{equation}
\label{eq:geometric-W}
W= -\Delta_h|\vec H|^2 + 2\|(\nabla \vec H )^\perp\|_h^2 + |\vec H|^2 \mathrm{Sc}^\Sigma - 2\operatorname{Ein}(\vec H^\star, \vec H^\star) - \frac{\theta_k^2}{2}\|\chi_\ell\|_h^2 - \frac{\theta_\ell^2}{2}\|\chi_k\|_h^2.
\end{equation}
For the dual component $
W^\star=\frac12\left(\theta_kW_\ell-\theta_\ell W_k\right)$,
using the definitions of $W_\ell$ and $W_k$ gives
\begin{align*}
W^\star={}&
\theta_k\Delta_h\theta_\ell
-\theta_\ell\Delta_h\theta_k
-2\theta_\ell\theta_k\operatorname{div}_h s_\ell
-2\theta_k s_\ell(\nabla^h\theta_\ell)
-2\theta_\ell s_\ell(\nabla^h\theta_k)
\\
&-\frac{\theta_k^2}{2}
\left(
\operatorname{Ein}(\ell,\ell)+\|\chi_\ell\|_h^2
\right)
+\frac{\theta_\ell^2}{2}
\left(
\operatorname{Ein}(k,k)+\|\chi_k\|_h^2
\right).
\end{align*}
The differential terms can be written in divergence form. 
\[
\theta_k\Delta_h\theta_\ell-\theta_\ell\Delta_h\theta_k
=
\operatorname{div}_h
\left(
\theta_k\nabla^h\theta_\ell
-\theta_\ell\nabla^h\theta_k
\right),
\]
while
\[
\theta_\ell\theta_k\operatorname{div}_h s_\ell
+\theta_k s_\ell(\nabla^h\theta_\ell)
+\theta_\ell s_\ell(\nabla^h\theta_k)=
\operatorname{div}_h
\left(
\theta_\ell\theta_k s_\ell
\right).
\]
Moreover, from the definitions of $\vec H$ and $\vec H^\star$,
\[
-\frac{\theta_k^2}{2}\operatorname{Ein}(\ell,\ell)
+\frac{\theta_\ell^2}{2}\operatorname{Ein}(k,k)
=
-2\operatorname{Ein}(\vec H,\vec H^\star).
\]
 Combining these identities yields
\begin{equation}
\label{eq:geometric-Wstar}
\begin{aligned}
W^\star={}&
\operatorname{div}_h\!\left(
\theta_k\nabla^h\theta_\ell
-\theta_\ell\nabla^h\theta_k
-2\theta_\ell\theta_k s_\ell
\right)
-2\operatorname{Ein}(\vec H,\vec H^\star)\\
&
-\frac{\theta_k^2}{2}\|\mathring\chi_\ell\|_h^2
+\frac{\theta_\ell^2}{2}\|\mathring\chi_k\|_h^2.
\end{aligned}
\end{equation}

\subsection{The non-null mean curvature regime}
\label{subsec:mean-curvature-gauge}

We now assume that the mean curvature
vector is non-null. Then $
|\vec H|^2\neq0$, $
\langle\vec H,\vec H^\star\rangle=0$ and  $\langle\vec H^\star,\vec H^\star\rangle=-|\vec H|^2$,
so $(\vec H,\vec H^\star)$ is an orthogonal basis of the normal bundle.
Since $\theta_\ell\theta_k\neq0$,  full area-constrained Willmore
criticality is equivalent to
\[
W=\lambda|\vec H|^2,
\qquad
W^\star=0.
\]

The non-null condition also allows us to remove the boost freedom by
choosing a canonical null frame adapted to $\vec H$. Set
\begin{equation} |\vec H| := \sqrt{\left|\langle\vec H,\vec H\rangle\right|} \qquad\text{and}\qquad \varepsilon := \operatorname{sgn}(\langle\vec H,\vec H\rangle). 
\end{equation}
Depending on the causal character of $\vec H$, we define a canonical unit spacelike normal $\nu_{\vec H}$ and a future-directed unit timelike normal $U$.

\textbf{Spacelike regime \((\varepsilon=1)\).} Since $ -\theta_\ell\theta_k = |\vec H|^2 > 0$, the null expansions have opposite signs. We label the null directions so that \(\theta_\ell>0>\theta_k\), and set \begin{equation} \nu_{\vec H} := \frac{\vec H}{|\vec H|} \qquad\text{and}\qquad U := \frac{\vec H^\star}{|\vec H|}. \end{equation} With this labeling, \(U\) is future-directed. 

\textbf{Timelike regime \((\varepsilon=-1)\).} The null expansions have the same sign. Let \(\sigma\in\{-1,1\}\) denote their common sign. Then \(\vec H\) is past-directed when \(\sigma=1\) and future-directed when \(\sigma=-1\). We set \begin{equation} U := -\sigma\frac{\vec H}{|\vec H|} \qquad\text{and}\qquad \nu_{\vec H} := -\sigma\frac{\vec H^\star}{|\vec H|}. \end{equation} Then \(U\) is future-directed and unit timelike, while \(\nu_{\vec H}\) is unit spacelike. Interchanging the two null directions changes the sign of \(\vec H^\star\), and hence of \(\nu_{\vec H}\). 

In both regimes, \(U\) is a future-directed unit timelike normal and \(\nu_{\vec H}\) is a unit spacelike normal orthogonal to it. We define \begin{equation} \label{eq:mean-curvature-null-frame} \ell_{\vec H} := U+\nu_{\vec H}, \qquad k_{\vec H} := U-\nu_{\vec H}. \end{equation} The vectors \(\ell_{\vec H}\) and \(k_{\vec H}\) are future-directed null normals satisfying $ \langle\ell_{\vec H},k_{\vec H}\rangle=-2$.  Let \(\sigma=1\) in the spacelike regime and let \(\sigma\) denote the common sign of the expansions in the timelike regime. Then \begin{equation*}\theta_{\ell_{\vec H}}=\langle\vec H,\ell_{\vec H}\rangle = \sigma|\vec H|, \qquad \theta_{k_{\vec H}} =\langle\vec H,k_{\vec H}\rangle= -\varepsilon\sigma|\vec H|. \end{equation*} We call \((\ell_{\vec H},k_{\vec H})\) a \emph{canonical mean curvature null frame}. In the spacelike regime, the ordered frame is uniquely determined by \(\vec H\) and the chosen time orientation. In the timelike regime, the corresponding unordered pair is canonical, whereas the ordered frame is determined only up to interchange of \(\ell_{\vec H}\) and \(k_{\vec H}\). 

For notational simplicity, throughout the remainder of this subsection we write \begin{equation*} \ell:=\ell_{\vec H}, \qquad k:=k_{\vec H}, \end{equation*} and denote by \(s_\ell\) the corresponding normal connection one-form. 

If \((\widehat\ell,\widehat k)\) is an initial future-directed normalized null frame, labeled so that \(\theta_{\widehat\ell}>0>\theta_{\widehat k}\) in the spacelike regime, then the canonical frame is obtained from the boost \begin{equation*} 
\ell = a\,\widehat\ell, \qquad k = a^{-1}\widehat k, \qquad a = \sqrt{ \left| \frac{\theta_{\widehat k}} {\theta_{\widehat\ell}} \right| }. 
\end{equation*} 
In the timelike regime, this construction produces the canonical frame up to interchange of its two null vectors.

\noindent\textbf{General normal directions.} Every normal vector field $V\in\Gamma(N\Sigma)$ admits the unique decomposition
\begin{equation} \label{eq:V-H-Hstar-decomposition} V = \frac{\langle V,\vec H\rangle}{|\vec H|^2}\vec H-\frac{\langle V,\vec H^\star\rangle}{|\vec H|^2}\vec H^\star.
\end{equation}
Suppose  that $\langle V,\vec H\rangle\neq0 $
everywhere on $\Sigma$. We may then define 
\begin{equation} \label{eq:def-c-V}
\beta := -\frac{\langle V,\vec H^\star\rangle} {\langle V,\vec H\rangle}. 
\end{equation} 
Since directional criticality depends only on the normal line generated by
$V$, this direction may equivalently be represented by
\begin{equation} \label{eq:V-H-cHstar} 
V\sim\vec H+\beta\vec H^\star.
\end{equation} 
The function $\beta$ has a direct geometric interpretation.
\begin{equation} \label{eq:V-causal-character} 
\langle V, V \rangle= \frac{\langle V, \vec H \rangle^2}{|\vec H|^2}(1 - \beta^2). 
\end{equation} 
Consequently, the causal character of $V$ is completely determined by
$|\beta|$ and the causal character of $\vec H$:
\begin{center}
\renewcommand{\arraystretch}{1.2}
\begin{tabular}{|l| l| l|}
\hline
& \quad $\vec H$ is spacelike, $|\vec H|^2 >0$ & \quad $\vec H$ is timelike, $|\vec H|^2 <0$ \\
\hline
$|\beta| < 1$ &\quad  $V$ is spacelike & \quad $V$ is timelike \\
$|\beta| = 1$ &\quad $V$ is null & \quad $V$ is null \\
$|\beta| > 1$ &\quad $V$ is timelike & \quad $V$ is spacelike \\
\hline
\end{tabular}
\end{center}
When $|\beta|<1$, the directions $V$ and $\vec H$ have the same causal character, and $\beta=\tanh\varphi$, where $\varphi$ is their relative hyperbolic angle in the normal plane.

The general directional Euler--Lagrange equation derived above therefore takes the particularly simple form 
\begin{equation} \label{eq:directional-H-Hstar}
\langle V,\vec H\rangle W-\langle V,\vec H^\star\rangle W^\star = \lambda \langle V,\vec H\rangle |\vec H|^2. \end{equation} When $\langle V,\vec H\rangle\neq0$, this is equivalently 
\begin{equation} \label{eq:directional-c-equation}
W+\beta W^\star = \lambda|\vec H|^2. 
\end{equation} 
The Willmore equation in direction $\vec H$ corresponds to $\beta =0$, while the case $\langle V,\vec H\rangle=0$ is precisely the dual mean curvature direction and reduces to $ W^\star=0$.
\begin{proposition}[Directional Willmore equation in the canonical gauge] \label{prop:directional-mean-curvature-gauge} 
Let $\Sigma$ have non-null mean curvature vector and suppose that it is area-constrained Willmore   in the
normal direction generated by $
\vec H+\beta\vec H^\star $
for some function $\beta\in C^\infty(\Sigma)$.  Then, in the canonical mean curvature null frame, 
\begin{equation} \label{eq:directional-canonical-gauge} 
\begin{aligned}
\lambda =& \mathrm{Sc}^{\Sigma} -2\Delta_h\log|\vec H| -2\left\|\nabla^h\log|\vec H|\right\|_h^2 -2\|s_\ell\|_h^2 -\frac{|\vec H|^2}{2}  +2\beta\,\operatorname{div}_h s_\ell +4\beta\,s_\ell\big(\nabla^h\log|\vec H|\big) \\ &-\varepsilon \Bigg[ 2\operatorname{Ein} \left( \frac{\vec H^\star}{|\vec H|}, \frac{\vec H^\star}{|\vec H|}+\beta\frac{\vec H}{|\vec H|} \right)  +\frac{1+\beta}{2}\|\mathring\chi_\ell\|_h^2 +\frac{1-\beta}{2}\|\mathring\chi_k\|_h^2 \Bigg]. 
\end{aligned}
\end{equation} Here $\varepsilon=\operatorname{sgn}(|\vec H|^2)$.
\end{proposition} 
\begin{proof} 
We first recall the contribution of the mean-curvature direction. In the canonical gauge, $ |\theta_\ell| = |\theta_k| = |\vec H|$.  The normal derivative identity\eqref{eq:norm_nabla_H} gives
\[ \|(\nabla \vec H)^\perp\|_h^2 = \varepsilon \|\nabla^h |\vec H|\|_h^2 - |\vec H|^2\|s_\ell\|_h^2 . \] 
Moreover, 
$\Delta_h(|\vec H|^2) = \varepsilon \left( 2|\vec H|\Delta_h|\vec H| + 2|\nabla^h|\vec H||_h^2 \right)$. Substituting these identities into the mean-curvature equation and using $ \frac{\Delta_h|\vec H|}{|\vec H|} = \Delta_h\log|\vec H| + \left\|\nabla^h\log|\vec H|\right\|_h^2 $ gives 
\begin{equation} \label{eq:W-over-H-canonical} \begin{aligned} \frac{W}{|\vec H|^2} =& \mathrm{Sc}^{\Sigma} -2\Delta_h\log|\vec H| -2\left\|\nabla^h\log|\vec H|\right\|_h^2 -2\|s_\ell\|_h^2 -\frac{|\vec H|^2}{2} \\ & -\varepsilon \left[ 2\operatorname{Ein} \left( \frac{\vec H^\star}{|\vec H|}, \frac{\vec H^\star}{|\vec H|} \right) + \frac12\|\mathring\chi_\ell\|_h^2 + \frac12\|\mathring\chi_k\|_h^2 \right]. \end{aligned} \end{equation}
For the dual component, \eqref{eq:geometric-Wstar} simplifies in the
canonical frame because $
\theta_k\nabla^h\theta_\ell
-
\theta_\ell\nabla^h\theta_k
=0$ and $-2\theta_\ell\theta_k s_\ell=2|\vec H|^2s_\ell$. Therefore
\[ \frac{1}{|\vec H|^2}
\operatorname{div}_h\!\left(2|\vec H|^2s_\ell\right)
=
2\operatorname{div}_h s_\ell
+
4s_\ell(\nabla^h\log|\vec H|),
\]
and \eqref{eq:geometric-Wstar} yields
\begin{equation}
\label{eq:Wstar-over-H-canonical}
\begin{aligned}
\frac{W^\star}{|\vec H|^2}
=&
2\operatorname{div}_h s_\ell
+
4s_\ell(\nabla^h\log|\vec H|)
-2\varepsilon
\operatorname{Ein}\left(
\frac{\vec H}{|\vec H|},
\frac{\vec H^\star}{|\vec H|}
\right)
\\
&-\frac{\varepsilon}{2}
\left(
\|\mathring\chi_\ell\|_h^2
-
\|\mathring\chi_k\|_h^2
\right).
\end{aligned}
\end{equation}

Finally, by \eqref{eq:directional-c-equation}, $
W+\beta W^\star=\lambda|\vec H|^2$. Combining \eqref{eq:W-over-H-canonical} and
\eqref{eq:Wstar-over-H-canonical} gives
\eqref{eq:directional-canonical-gauge}.
\end{proof}

The form of \eqref{eq:directional-canonical-gauge} makes the role of the
causal character of the critical direction explicit. In the next section
we exploit this structure to derive curvature inequalities and rigidity
results.

\begin{remark}[Relation with stable STCMC surfaces] \label{rem:W-STCMC}The quantity $W$ is closely related to the stability theory of spacetime constant mean curvature (STCMC) surfaces. Recall that an STCMC surface satisfies $\vert{}\vec H\vert{}^2=\mathrm{constant}$. In the stability theory introduced in \cite[Definition~3.2]{diaz2026curvature}, constant-mode stability is defined by the condition $\delta_{\vec H}^2\vert{}\Sigma\vert{}\geq0$. With the notation used above, $\delta_{\vec H}^2\vert{}\Sigma\vert{} = \int_\Sigma W\,d\mu$, thus constant-mode stability is equivalent to $$\int_\Sigma W\,d\mu\geq0.$$
This integral condition was shown in \cite[Theorem 4.1]{diaz2026curvature} to yield the sharp curvature inequality $|\vec H|^2\leq 16\pi / \vert{}\Sigma\vert{}$, while an additional ambient curvature assumption yields Minkowskian
rigidity in the equality case.

By contrast, if $\Sigma$ is area-constrained Willmore in the direction
$\vec H$, with $|\vec H|^2>0$ and $\lambda\geq0$, then 
\[
W=\lambda|\vec H|^2\geq0.
\]
Thus directional Willmore criticality in the mean-curvature direction
provides a pointwise strengthening of constant-mode stability.  This stronger control underlies the sharp curvature inequalities developed below and allows us to obtain the bound $
\int_\Sigma|\vec H|^2\,d\mu\leq16\pi $
and the corresponding Minkowskian rigidity without assuming that $\Sigma$ is STCMC or imposing the additional ambient curvature hypothesis required in the STCMC setting.
\end{remark}

\section{Curvature inequalities and rigidity} \label{sectionrigi} A principal motivation for the results in this section comes from the rigid geometric properties of area-constrained Willmore surfaces in Riemannian $3$-manifolds. In that setting, these surfaces satisfy the classical Euler--Lagrange equation \[ 0 = \lambda H+\Delta_\Sigma H +H|\mathring B|^2 +H\operatorname{Ric}(\nu,\nu), \] where $\lambda$ is the area-constraint multiplier, $\mathring B$ is the trace-free second fundamental form, and $\nu$ is the unit normal. Under nonnegative scalar curvature and suitable sign assumptions, this equation yields a sharp upper bound for the Willmore energy.

\begin{theorem}[{\cite[Corollary~2.11]{diaz2025rigidity}}]
Let $(M,g)$ be a $3$-dimensional Riemannian manifold with nonnegative scalar curvature, and let $\Omega\subset M$ be a relatively compact domain with smooth connected boundary $\Sigma=\partial\Omega$. Assume that $\Sigma$ is a closed, connected, area-constrained Willmore surface with positive mean curvature $H>0$ and nonnegative Lagrange multiplier $\lambda\geq0$. Then
$$\int_\Sigma H^2\,d\mu\leq16\pi.$$
Moreover, equality holds if and only if $\Sigma$ is isometric to a round sphere and $\Omega$ is isometric to a Euclidean ball in $\mathbb R^3$.
\end{theorem}
The sharp inequality was first established for spherical area-constrained Willmore surfaces by Lamm, Metzger, and Schulze \cite[Theorem~4]{willflat}. The equality case and the resulting Euclidean rigidity were obtained subsequently in \cite{diaz2025rigidity}, along with analogous results for hyperbolic, spherical, and higher-dimensional settings.

The Lorentzian codimension-two setting carries an additional directional
and causal structure. The Euler--Lagrange equation depends on the chosen normal direction, and the resulting curvature estimates reflect both its
variation along the surface and its causal character.

We begin by deriving an exact directional curvature identity in the non-null mean-curvature regime. We first consider spacelike critical
directions, corresponding to $|\beta|<1$ in the parametrization introduced before. For a general function $\beta$, the identity
contains a normal-connection term whose integral has no definite sign. This leads to a general corrected estimate, while a sharp topological inequality follows whenever that term has a favorable sign. Motivated by this sign condition, we introduce the notion of
connection-compatibility below.

Null critical directions are treated separately in a null frame. This formulation also applies when the mean curvature vector itself is null, where the parametrization by $\vec H+\beta\vec H^\star$ is no longer
available.

\subsection{The exact directional curvature identity} The canonical-gauge formulation obtained in Proposition~\ref{prop:directional-mean-curvature-gauge} gives an exact relation between the topology, the Willmore functional, the Lagrange multiplier, the variation of the chosen normal direction, and the geometry of the normal bundle. 
\begin{proposition}[Exact directional curvature identity] \label{prop:exact-directional-defect} Let $(M,g)$ be a time-oriented $4$-dimensional Lorentzian manifold and let $\Sigma\subset M$ be a closed surface with non-null mean curvature
vector.   Suppose that it is area-constrained Willmore in the
normal direction generated by $
\vec H+\beta\vec H^\star $
for some function $\beta\in C^\infty(\Sigma)$ and with associated Lagrange multiplier $\lambda$. Then 
\begin{equation} \label{eq:exact-directional-defect} 
\begin{aligned} 8\pi\chi(\Sigma) -\int_\Sigma|\vec H|^2\,d\mu =& 2\lambda|\Sigma| + \int_\Sigma \Big( 2(1+\beta) \left\| \nabla^h\log|\vec H|-s_\ell \right\|_h^2 \\ & + 2(1-\beta) \left\| \nabla^h\log|\vec H|+s_\ell \right\|_h^2 - 4 \beta \operatorname{div}_h s_\ell \Big)d\mu \\ &+ 4\varepsilon \int_\Sigma \operatorname{Ein} \left( \frac{\vec H^\star}{|\vec H|}, \frac{\vec H^\star+\beta\vec H}{|\vec H|} \right)d\mu \\ &+ \varepsilon \int_\Sigma \left( (1+\beta)\|\mathring\chi_\ell\|_h^2 + (1-\beta)\|\mathring\chi_k\|_h^2 \right)d\mu, \end{aligned} 
\end{equation}
Here \((\ell,k)\) denotes the canonical mean curvature null frame, \(\varepsilon = \operatorname{sgn}(|\vec H|^2)\), and \(\chi(\Sigma)\) is the Euler characteristic of \(\Sigma\). 
\end{proposition} 
\begin{proof} Integrating \eqref{eq:directional-canonical-gauge} over the closed surface $\Sigma$, the Laplacian term vanishes and the Gauss--Bonnet theorem gives $ \int_\Sigma\operatorname{Sc}^{\Sigma}\,d\mu = 4\pi\chi(\Sigma)$. Also
\[ 2\|\nabla^h\log|\vec H|\|_h^2+2\|s_\ell\|_h^2-4\beta\,s_\ell(\nabla^h\log|\vec H|) = (1+\beta)\|\nabla^h\log|\vec H|-s_\ell\|_h^2 + (1-\beta)\|\nabla^h\log|\vec H|+s_\ell\|_h^2. \]
Substituting these identities yields \eqref{eq:exact-directional-defect}. 
\end{proof}

\subsection{Spacelike critical directions}
\label{subsec:spacelike-directional}

We first consider surfaces that are area-constrained Willmore in a strictly spacelike normal direction $V$. The favorable case occurs when the mean curvature vector is also spacelike, which we assume throughout this subsection.

Recall that in the canonical mean-curvature frame, $\nu_{\vec H} = \frac{\vec H}{\vert{}\vec H\vert{}}$ and $U = \frac{\vec H^\star}{\vert{}\vec H\vert{}}$. The linear combination $U+\beta\nu_{\vec H}$ is a future-directed timelike vector, since $$\langle U+\beta\nu_{\vec H}, U+\beta\nu_{\vec H}\rangle = -1+\beta^2<0.$$Notice also that\begin{equation}\label{eq:Ein-directional-null-decomposition}\begin{aligned}4\operatorname{Ein} ( U,U+\beta\nu_{\vec H}) ={}& (1+\beta)\operatorname{Ein}(\ell,\ell) + 2\operatorname{Ein}(\ell,k) + (1-\beta)\operatorname{Ein}(k,k).
\end{aligned}
\end{equation}
Thus the curvature hypothesis used below is automatic under the dominant energy condition, although the estimates themselves require only this particular Einstein-tensor inequality.

\noindent\textbf{General spacelike directions.}
For a general function $\beta:\Sigma\to(-1,1)$, the mixed term involving
$\nabla^h \beta$ in \eqref{eq:exact-directional-defect} has no definite sign.
Nevertheless, it can be controlled by completing squares.

\begin{theorem}[Curvature inequality for general spacelike directions]
\label{thm:general-spacelike-direction}
Let $(M,g)$ be a time-oriented $4$-dimensional Lorentzian manifold, and  let $\Sigma\subset M$ be a closed surface with $|\vec H|^2>0$. Suppose that $\Sigma$ is area-constrained Willmore in a strictly spacelike normal direction generated by $
\vec H+\beta\vec H^\star$, 
with associated Lagrange multiplier $\lambda$. Assume that $
\operatorname{Ein} (\vec H^\star,\vec H^\star+\beta\vec H) \geq0 $
along $\Sigma$. Then
\begin{equation}
\label{eq:general-spacelike-corrected}
2\lambda|\Sigma|
+
\int_\Sigma|\vec H|^2\,d\mu
\leq
8\pi\chi(\Sigma)
+
\int_\Sigma
\frac{\|\nabla^h \beta\|_h^2}{1-\beta^2}\,d\mu .
\end{equation}
Equality holds in \eqref{eq:general-spacelike-corrected} if and only if  $ \nabla^h\log|\vec H| = -\frac{\beta\,\nabla^h \beta}{2(1-\beta^2)}$,  $s_\ell = -\frac{\nabla^h \beta}{2(1-\beta^2)}$, $ \operatorname{Ein} ( \vec H^\star,\vec H^\star+\beta\vec H ) =0$ and  $\mathring\chi_\ell=\mathring\chi_k=0$. 
\end{theorem}
\begin{proof}
Since $
\vec H+\beta\vec H^\star $ is spacelike, \eqref{eq:V-causal-character} gives $|\beta|<1$. For $\varepsilon=1 $ the differential contribution in
\eqref{eq:exact-directional-defect} after integrating by parts satisfies
\begin{align*}
&
2(1+\beta)
\left\|
\nabla^h\log|\vec H|-s_\ell
\right\|_h^2
+
2(1-\beta)
\left\|
\nabla^h\log|\vec H|+s_\ell
\right\|_h^2
+
4\left\langle\nabla^h \beta,s_\ell\right\rangle_h
\\
={}&
2(1+\beta)
\left\|
\nabla^h\log|\vec H|
-s_\ell
-\frac{\nabla^h \beta}{2(1+\beta)}
\right\|_h^2
\\
&+
2(1-\beta)
\left\|
\nabla^h\log|\vec H|
+s_\ell
+\frac{\nabla^h \beta}{2(1-\beta)}
\right\|_h^2
-
\frac{\|\nabla^h \beta\|_h^2}{1-\beta^2}.
\end{align*}
Since $|\beta|<1$, the coefficients $1+\beta$ and $1-\beta$ are strictly positive.
Hence the two square terms are nonnegative. The shear contribution is also
nonnegative, and by assumption $\operatorname{Ein} (\vec H^\star,\vec H^\star+\beta\vec H) \geq0$. Dropping these nonnegative terms from \eqref{eq:exact-directional-defect} gives \eqref{eq:general-spacelike-corrected}.

It remains to characterize equality in \eqref{eq:general-spacelike-corrected}. Equality holds if and only if every term discarded above vanishes. Since $1\pm \beta>0$, the two completed squares must vanish, giving \[ \nabla^h\log|\vec H|-s_\ell = \frac{\nabla^h \beta}{2(1+\beta)}, \qquad \nabla^h\log|\vec H|+s_\ell = -\frac{\nabla^h \beta}{2(1-\beta)}. \] Adding and subtracting these identities yields $ \nabla^h\log|\vec H| = -\frac{\beta\,\nabla^h \beta}{2(1-\beta^2)}$ and $ s_\ell = -\frac{\nabla^h \beta}{2(1-\beta^2)}$.  The remaining discarded terms vanish precisely when $ \operatorname{Ein} ( \vec H^\star,\vec H^\star+\beta\vec H ) =0 $ and  $ \mathring\chi_\ell=\mathring\chi_k=0$.  Conversely, these conditions make every discarded nonnegative term vanish, and hence equality holds in \eqref{eq:general-spacelike-corrected}.
\end{proof}

The general estimate above can be sharpened when the variation of the
critical direction is suitably compatible with the normal connection.
This motivates the following terminology.

\begin{definition}[Connection-compatible direction]
\label{def:connection-compatible}
Let $|\vec H|^2>0$, and let $
V\sim \vec H+\beta\vec H^\star$ 
be a  normal direction. We say that $V$ is
\emph{connection-compatible} if
\begin{equation}
\label{eq:connection-compatible}
\int_\Sigma
\beta\,\operatorname{div}_h s_\ell\,d\mu
\leq0.
\end{equation}
\end{definition}
In the canonical mean-curvature null frame. This also admits an
intrinsic formulation, in this frame. 
\[
s_\ell(X)
=
-\langle
\nabla_X\frac{\vec H}{|\vec H|},
\frac{\vec H^\star}{|\vec H|}
\rangle
=
-\frac{\langle\nabla_X\vec H,\vec H^\star\rangle}{|\vec H|^2}.
\]
Since $\Sigma$ is closed, integration by parts shows that the connection-compatibility is equivalently characterized by
\[
\int_\Sigma
\frac{
\langle\nabla_{\nabla^h\beta}\vec H,\vec H^\star\rangle
}{|\vec H|^2}\,d\mu
\leq0,
\]
and therefore depends only on the critical direction and the geometry
of the normal bundle.

\begin{theorem}
\label{thm:connection-compatible-spacelike}
Let $(M,g)$ be a time-oriented $4$-dimensional Lorentzian manifold, and let
$\Sigma\subset M$ be a closed, connected spacelike surface with $|\vec H|^2>0$.
Suppose that $\Sigma$ is area-constrained Willmore in a connection-compatible spacelike direction generated by $ \vec H+\beta\vec H^\star$, 
with associated Lagrange multiplier $\lambda$. Assume that $
\operatorname{Ein}
\left(
\vec H^\star,\vec H^\star+\beta\vec H
\right)
\geq0 $
along $\Sigma$.
Then 
\begin{equation}
\label{eq:spacelike-sharp}
2\lambda|\Sigma|
+
\int_\Sigma|\vec H|^2\,d\mu
\leq
8\pi\chi(\Sigma).
\end{equation}
If equality holds,  then $\Sigma$ has constant scalar curvature given by $ \mathrm{Sc}^{\Sigma} = \frac12|\vec H|^2 + \lambda$,  its mean curvature vector is parallel ($(\nabla \vec H)^\perp=0$),   $\operatorname{Ein}
(\vec H^\star,\vec H^\star+\beta\vec H)=0$ and $\mathring\chi_\ell= \mathring\chi_k=0$. 
\end{theorem}
\begin{proof}
Since $
\vec H+\beta\vec H^\star $ is spacelike, $|\beta|<1$. Hence the
weighted square terms and the shear terms in
\eqref{eq:exact-directional-defect} are nonnegative. The Einstein-tensor
term is nonnegative by assumption, while connection-compatibility gives $
\int_\Sigma
\beta\,\operatorname{div}_h s_\ell\,d\mu
\leq0$. The integrated defect identity therefore yields
\eqref{eq:spacelike-sharp}.

Suppose equality holds. Then $\int_\Sigma
\beta\,\operatorname{div}_h s_\ell\,d\mu=0,$ and every nonnegative term in the
defect identity must vanish. Since $1\pm\beta>0$, we  obtain $
\nabla^h\log|\vec H|-s_\ell=0$ and $\nabla^h\log|\vec H|+s_\ell=0$.
Adding and subtracting these equations gives $
\nabla^h\log|\vec H|=0$ and $s_\ell=0$.
Moreover, $
\mathring\chi_\ell=\mathring\chi_k=0 $
and $\operatorname{Ein}
\left(
\vec H^\star,\vec H^\star+\beta \vec H
\right)
=0$. Substituting these vanishing conditions into the directional Willmore equation
\eqref{eq:directional-canonical-gauge} gives
\[
\mathrm{Sc}^{\Sigma}
=
\frac12|\vec H|^2 +\lambda.
\]
Since $|\vec H|$ is constant, the right-hand side is a  constant. Finally, in the canonical mean curvature null frame \[
(\nabla_X\vec H)^\perp
=
\frac{X(|\vec H|)}{|\vec H|}\,\vec H
+
s_\ell(X)\vec H^\star.
\] Since \(\nabla^h|\vec H|=0\) and \(s_\ell=0\), it follows that $ (\nabla \vec H)^\perp=0$. 
\end{proof}
\begin{remark}[Connection-compatibility and Hawking-energy monotonicity]\label{rem:connection-compatible-Hawking}
There are two geometrically distinguished mechanisms that produce
connection-compatible directions. First, if $\beta$ is constant, then
$\nabla^h\beta=0$, and hence connection-compatibility holds with equality.
This is precisely the constant-angle case.

A second mechanism comes from the normal geometry of the surface.
Bray and Jauregui~\cite{bray2015time} introduced the notion of a
\emph{time-flat} surface, characterized in the present notation by
\[
\operatorname{div}_h s_\ell=0
\]
on the canonical mean curvature frame. For such a surface, $
\int_\Sigma
\beta\,\operatorname{div}_h s_\ell\,d\mu=0$
for every spacelike normal direction, so every such direction is
connection-compatible with equality. 

Both conditions have a common origin in the study of Hawking energy
monotonicity. Bray, Hayward, Mars, and Simon~\cite{bray2007generalized}
studied generalized inverse mean curvature flows with the goal of
identifying conditions under which the Hawking energy is monotone. For
spacelike uniformly expanding flows, the remaining freedom is described
by a parameter that is constant on each leaf; in the present
parametrization this corresponds, up to conventions, to constant
$\beta$. Bray and Jauregui~\cite{bray2015time} subsequently showed that,
under the dominant energy condition, time flatness also leads to
nonnegativity of the Hawking energy variation along outward-spacelike
uniformly area-expanding directions.

In both cases, the corresponding geometric condition makes the
indefinite normal-connection contribution vanish. For monotonicity,
however, vanishing is stronger than necessary: it is enough for this
term to have the favorable sign. In the present notation, this is
precisely the connection-compatibility condition. Thus constant-angle
directions and time-flat surfaces arise as distinguished equality cases
of a more general mechanism. As we will see in
Section~\ref{Hawkingensec}, the same connection-compatibility condition
reappears naturally in the variation of the Hawking energy.
\end{remark}
The sharp curvature inequality shows that the Lagrange multiplier imposes a strict topological constraint on the surface.
\begin{corollary}\label{cor:lambda-topology-bound}
Under the assumptions of Theorem \ref{thm:connection-compatible-spacelike}
\begin{equation}\label{eq:lambda-topology-bound}
\lambda |\Sigma| < 4\pi\chi(\Sigma).
\end{equation}
In particular, a  surface with torus or higher-genus topology (\(\chi(\Sigma) \leq 0\)) must have a strictly negative Lagrange multiplier.
\end{corollary}

\begin{remark}[Timelike critical directions] Suppose that the critical direction $V$ is timelike. If the mean curvature vector is also timelike, then $|\beta|<1$, so the coefficients $1\pm \beta$ remain positive. However, the sign $\varepsilon=-1$ in \eqref{eq:exact-directional-defect} reverses the Einstein and shear contributions, while $|\vec H|^2<0$. Consequently,  identity no longer has a definite defect structure and does not yield a topological constraint on the Lagrange multiplier or an analogue of the sharp rigidity results obtained for spacelike critical directions. If instead $\vec H$ is spacelike, then $|\beta|>1$ and the coefficients $1+\beta$ and $1-\beta$ have opposite signs, so the differential and shear contributions are themselves indefinite. We therefore do not pursue separate curvature inequalities for timelike critical directions. 
\end{remark}

 \subsection{Minkowskian rigidity}

We now strengthen the preceding sharp curvature inequalities by analyzing the
extremal case
\[
\int_\Sigma |\vec H|^2\,d\mu=16\pi.
\]
We will assume that the spacetime $(M,g)$ satisfies the \emph{dominant energy condition} (DEC), namely
\[
\operatorname{Ein}(X,Y)\ge 0
\]
for all future-directed causal vectors $X,Y$.

The rigidity part of our argument uses the Kijowski-Liu-Yau quasi-local energy.
For a surface $\Sigma$ with positive Gaussian curvature and spacelike mean curvature
vector, this energy is defined by
\begin{equation}\label{liuyaumass}
    \mathcal{E}_{KLY}(\Sigma)
=
\frac{1}{8\pi}
\int_\Sigma
\left(
H_0-|\vec H|
\right)\,d\mu,
\end{equation}
where \(H_0\) is the mean curvature of the isometric embedding of \(\Sigma\) into
\(\mathbb{R}^3\). In the totally geodesic hypersurface case \((K=0)\), this
reduces to the Brown-York mass.  We will use the rigidity theorem of Liu-Yau \cite[Theorem 1]{liu2003positivity,liu2006positivity},
together with the stronger rigidity statement of Miao-Shi-Tam
\cite[Theorem 4.1]{miao2010geometric}. These results are recalled in the
appendix as Theorem~\ref{liuyaurigi} and Theorem~\ref{minkowskirigi},
respectively.

To pass from rigidity of the induced initial data to rigidity of the ambient
spacetime region, we use maximal globally hyperbolic developments.

\subsection*{Maximal globally hyperbolic developments}

Let $\Omega\subset M$ be a spacelike hypersurface. The induced metric
$g|_\Omega$ and second fundamental form $K$ define an initial data set
$(\Omega,g|_\Omega,K)$. A development of this data is a spacetime
$(\mathcal M,\mathbf g)$ satisfying the Einstein field equations, in which $\Omega$ embeds as a spacelike hypersurface
with induced metric $g|_\Omega$ and second fundamental form $K$.
Such a development is called globally hyperbolic if $\Omega$ is a Cauchy
hypersurface in $(\mathcal M,\mathbf g)$.

The Choquet-Bruhat-Geroch theorem asserts that every initial data set
satisfying the vacuum Einstein constraint equations admits a unique maximal globally
hyperbolic development, up to isometry \cite{choquet2009general}. In
particular, if $\Omega$ is contained in a globally hyperbolic region of the
original spacetime $(M,g)$, and $(M,g)$ satisfies the vacuum Einstein equations, then this maximal development coincides with the
domain of dependence $D(\Omega)\subset M$.
\begin{definition}[Minkowski causal-diamond rigidity]
\label{def}
Let $(M,g)$ be a Lorentzian manifold and let $\Sigma\subset M$ be a closed
spacelike surface. We say that $\Sigma$ satisfies \emph{Minkowski
causal-diamond rigidity} if, for every compact mean-convex spacelike
hypersurface $\Omega\subset M$ with boundary $\partial\Omega=\Sigma$, the
induced initial data on $\Omega$ admit an initial-data-preserving embedding
into Minkowski spacetime, and the maximal Cauchy development generated by
these data is isometric to a standard causal diamond in Minkowski spacetime.
\end{definition}
\begin{theorem} \label{thm:directional-Minkowski-rigidity} Let $(M,g)$ be a time-oriented $4$-dimensional Lorentzian manifold satisfying the dominant energy condition, and let $\Sigma\subset M$ be a closed, connected spacelike surface with spacelike mean curvature vector,  $|\vec H|^2>0$.  Suppose that $\Sigma$ is area-constrained Willmore in a
connection-compatible spacelike normal direction 
with associated Lagrange multiplier $\lambda\geq0$.
Then  $\Sigma\simeq\mathbb S^2$ or $\Sigma\simeq\mathbb{RP}^2$, and 
 \begin{equation} \label{eq:directional-Willmore-16pi} \int_\Sigma|\vec H|^2\,d\mu \leq16\pi. 
 \end{equation} 
 If equality holds in \eqref{eq:directional-Willmore-16pi}, then $\Sigma$ is isometric to a round sphere and satisfies Minkowski causal-diamond rigidity. 
 \end{theorem}
\begin{proof} 
Since $\vec H+\beta\vec H^\star$ is spacelike, $|\beta|<1$, both $\frac{\vec H^\star}{|\vec H|} $ and $ \frac{\vec H^\star}{|\vec H|}+\beta\frac{\vec H}{\vert{}\vec H\vert{}} $ are future-directed timelike vectors. Hence the dominant energy condition implies $ \operatorname{Ein} ( \vec H^\star ,\vec H^\star+\beta\vec H ) \geq0$.  We apply Theorem~\ref{thm:connection-compatible-spacelike} and obtain $ 2\lambda|\Sigma| + \int_\Sigma|\vec H|^2\,d\mu \leq 8\pi\chi(\Sigma)$.  Since $\lambda\geq0$, it follows that $0< \int_\Sigma|\vec H|^2\,d\mu \leq 8\pi\chi(\Sigma)$.  As every closed, connected surface satisfies $\chi(\Sigma)\leq2$, we obtain $ \int_\Sigma|\vec H|^2\,d\mu \leq16\pi$ and the topological characterization.  

Suppose now that equality holds, so that $ \int_\Sigma|\vec H|^2\,d\mu=16\pi$.  Then 
\[ 16\pi = \int_\Sigma|\vec H|^2\,d\mu \leq 8\pi\chi(\Sigma) \leq 16\pi. \]
Consequently, $\chi(\Sigma)=2$,  and equality holds in the corresponding sharp topological inequality. Therefore $\Sigma$ is a topological sphere and the equality statement in Theorem~\ref{thm:connection-compatible-spacelike} gives:   $\lambda=0$,  $\operatorname{Ein}( \frac{\vec H^\star}{|\vec H|},\frac{\vec H^\star}{|\vec H|}+\beta\frac{\vec H}{\vert{}\vec H\vert{}})=\mathring\chi_\ell= \mathring\chi_k=0$,  $|\vec H|^2$  is constant and  
\begin{equation}\label{eq:directional-rigidity-scalar} 
\mathrm{Sc}^{\Sigma} = \frac12|\vec H|^2.
\end{equation} 
Thus \(\Sigma\) is a topological sphere with constant positive Gaussian curvature and is therefore isometric to a round sphere.

The remainder of the proof follows the rigidity argument of \cite[Theorem~4.1]{diaz2026curvature}. We recall the main steps. Let $r$ denote the radius of the round sphere $\Sigma$. Since $ \mathrm{Sc}^{\Sigma}=\frac{2}{r^2}$,  equation \eqref{eq:directional-rigidity-scalar} gives $ |\vec H|^2=\frac{4}{r^2}$.  The mean curvature of the Euclidean isometric embedding of $\Sigma$ is therefore \[ H_0=\frac{2}{r}=|\vec H|. \] Consequently, its Kijowski-Liu-Yau quasi-local energy (\ref{liuyaumass}) vanishes, $\mathcal E_{KLY}(\Sigma) =0$.

Let $\Omega \subset M$ be any compact mean-convex spacelike hypersurface region with boundary $\partial\Omega=\Sigma$, and let $(\Omega,g_\Omega,K_\Omega)$ denote the induced initial data. Since $(M,g)$ satisfies the dominant energy condition, so does $(\Omega,g_\Omega,K_\Omega)$. By Theorem~\ref{liuyaurigi}, the vanishing of $\mathcal{E}_{KLY}(\Sigma)$ implies that $(\Omega,g_\Omega,K_\Omega)$ embeds
isometrically into Minkowski spacetime $\mathbb R^{3,1}$ as a compact spacelike
graph $\widetilde\Omega$ over a domain in $\mathbb R^3$, with boundary
$\widetilde\Sigma$.

The remainder of the argument is carried out in this Minkowski ambient space. Since $\tilde{\Sigma}$ spans the compact spacelike hypersurface $\tilde{\Omega}$ in $\mathbb{R}^{3,1}$, possesses positive Gaussian curvature, and has a spacelike mean curvature vector, Theorem~\ref{minkowskirigi} dictates that $\tilde{\Sigma}$ lies entirely on a flat spatial hyperplane $P$. Because $\tilde{\Sigma}$ is intrinsically a round sphere lying in a hyperplane, it bounds a totally geodesic flat Euclidean ball $B=B_r$ within that same hyperplane. Since $\widetilde\Omega$ and $B_r$ share the same boundary and are both strictly spacelike, a standard causal curve tracking argument \cite[Theorem 4.1]{diaz2026curvature} shows their domains of dependence coincide:  $D(\widetilde\Omega)=D(B_r)$.

To apply the uniqueness theorem for maximal globally hyperbolic vacuum developments to the compact initial data on $\Omega$, we first extend these data to a complete initial data set without boundary. Since $\widetilde{\Omega}$ is a compact strictly spacelike hypersurface in Minkowski spacetime with boundary $\widetilde{\Sigma}$, it admits a smooth spacelike extension $\widehat{\Omega}\subset\mathbb{R}^{3,1}$ that agrees with $\widetilde{\Omega}$ on the original compact region and coincides with the spacelike hyperplane $P$ outside a larger compact set. The induced data on $\widehat{\Omega}$ are therefore complete, asymptotically Euclidean vacuum initial data extending the original data on $\Omega$.

By the Choquet-Bruhat-Geroch theorem, the maximal globally hyperbolic vacuum development of the extended data is unique up to isometry. Since $\widehat{\Omega}$ is a complete Cauchy hypersurface of Minkowski spacetime, this development is Minkowski spacetime. Consequently, the maximal Cauchy development generated by the original compact data on $\Omega$ is isometric to the domain of dependence $D(\widetilde{\Omega})$ within this Minkowski development. Since $D(\widetilde{\Omega})=D(B_r)$, this development is isometric to the standard causal diamond determined by $B_r$.
\end{proof}
\begin{remark}
The rigidity conclusion in Theorem~\ref{thm:directional-Minkowski-rigidity} is formulated in terms of the maximal Cauchy development of the abstract initial data induced on $\Omega$, rather than directly in terms of the domain of dependence within the ambient spacetime $M$. This distinction is necessary because, unless $M$ satisfies the Einstein equations, its geometry off the hypersurface is not  determined by the initial data. Thus the Minkowski causal
diamond determined by the equality case need not be realized directly inside
\(M\).

If, however, the domain of dependence \(D(\Omega)\subset M\) is a globally hyperbolic region satisfying the Einstein equations, then this maximal development agrees with \(D(\Omega)\). In that case, the standard Minkowski causal diamond is realized directly inside \(M\).
\end{remark}
\begin{remark}[Equivalence in the equality case] Every full area-constrained Willmore surface is, by linearity of the first variation, area-constrained Willmore in any prescribed normal direction $V$. The converse is false in general. 
 
 Assume, however, that the hypotheses of Theorem~\ref{thm:directional-Minkowski-rigidity} hold, so that $\Sigma$ is area-constrained Willmore in a connection-compatible spacelike normal direction $V$ and $\lambda\geq0$. Suppose moreover that the sharp curvature bound is saturated, \[ \int_\Sigma |\vec H|^2\,d\mu=16\pi. \] 
 The vanishing identities obtained in the equality case imply $W^\star=0$. Since equality also forces $\lambda=0$, the directional Euler--Lagrange equation   $W+\beta W^\star=\lambda |\vec H|^2 $ 
 then yields $W=0$. Thus both Euler--Lagrange equations associated with the basis $(\vec H,\vec{H}^\star)$ are satisfied, and hence $\Sigma$ is fully area-constrained Willmore, with multiplier $\lambda=0$. Therefore, under the geometric rigidity of the equality case, criticality in a single spacelike normal direction upgrades to full normal-bundle criticality.
 \end{remark}

\subsection{Null critical directions} 
We now turn to area-constrained Willmore criticality in a null normal direction. Unlike the preceding non-null analysis, this case is treated directly in a null frame and requires no normalization by $|\vec H|$. Consequently, the null-direction Euler--Lagrange equation remains meaningful even when the mean curvature vector itself becomes null.

The resulting curvature identity is valid independently of the causal character of $\vec H$. Its sign structure, however, depends on the ratio of the two null expansions. When $|\vec H|^2\geq0$, this ratio has the sign needed to recover a sharp topological curvature inequality. If $|\vec H|^2<0$, the sign reverses and the identity loses its definite-sign defect structure.
\begin{proposition}[Exact curvature identity for null critical directions] \label{prop:null-direction-exact} Let $(M,g)$ be a time-oriented $4$-dimensional Lorentzian manifold, and let $\Sigma\subset M$ be a closed, connected spacelike surface. Suppose that $\Sigma$ is area-constrained Willmore in the direction of a nowhere-vanishing null normal field $V$, with associated Lagrange multiplier $\lambda$. Let $k$ be the unique null normal field satisfying $ \langle V,k\rangle=-2$,  and let $s_V$ denote the corresponding normal connection one-form.  Assume that $ \theta_V\neq0 $ along $\Sigma$. Then \begin{align} 2\lambda|\Sigma| + \int_\Sigma|\vec H|^2\,d\mu ={}& 8\pi\chi(\Sigma) - 4\int_\Sigma \bigl\| \nabla^h\log|\theta_V|-s_V \bigr\|_h^2\,d\mu \nonumber\\ &- 2\int_\Sigma \operatorname{Ein}(V,k)\,d\mu + 2\int_\Sigma \frac{\theta_k}{\theta_V} \left( \operatorname{Ein}(V,V) + \|\mathring\chi_V\|_h^2 \right)d\mu . \label{eq:null-direction-exact} 
\end{align} 
\end{proposition} 
\begin{proof} Since $\theta_V$ is nowhere zero, the function $\log|\theta_V|$ is globally defined. Dividing the null-direction Euler--Lagrange equation (\ref{Wl}) by $\theta_V$, we use \[ \frac{\Delta_h\theta_V}{\theta_V} = \Delta_h\log|\theta_V| + \|\nabla^h\log|\theta_V|\|_h^2, \qquad \frac{s_V(\nabla^h\theta_V)}{\theta_V} = s_V(\nabla^h\log|\theta_V|). \] The derivative and normal-connection terms therefore combine as \begin{align*} & 2\frac{\Delta_h\theta_V}{\theta_V} - 4\frac{s_V(\nabla^h\theta_V)}{\theta_V} + 2\|s_V\|_h^2 - 2\operatorname{div}_h s_V \\ ={}& 2\operatorname{div}_h \bigl(\nabla^h\log|\theta_V|-s_V\bigr) + 2\bigl\| \nabla^h\log|\theta_V|-s_V \bigr\|_h^2. \end{align*} Using $ |\vec H|^2=-\theta_V\theta_k$,  the remaining terms in the Euler--Lagrange equation give
\begin{align} \mathrm{Sc}^{\Sigma} =& \frac12|\vec H|^2+\lambda + 2\operatorname{div}_h \bigl(\nabla^h\log|\theta_V|-s_V\bigr) + 2\bigl\| \nabla^h\log|\theta_V|-s_V \bigr\|_h^2 \nonumber\\ &+ \operatorname{Ein}(V,k) - \frac{\theta_k}{\theta_V} \left( \operatorname{Ein}(V,V) + \|\mathring\chi_V\|_h^2 \right). \label{eq:null-direction-pointwise} \end{align} 
Integrating over the closed surface $\Sigma$, the divergence term vanishes, while Gauss--Bonnet gives $ \int_\Sigma\mathrm{Sc}^{\Sigma}\,d\mu = 4\pi\chi(\Sigma)$. Multiplying the resulting identity by $2$ and rearranging yields \eqref{eq:null-direction-exact}.
\end{proof}

\begin{theorem}[Sharp curvature inequality for null critical directions] \label{thm:null-direction-nontimelike} Let $(M,g)$ be a time-oriented $4$-dimensional Lorentzian manifold, and let $\Sigma\subset M$ be a closed, connected spacelike surface with  $|\vec H|^2\geq0$.  Suppose that $\Sigma$ is area-constrained Willmore in the direction of a nowhere-vanishing null normal field $V$, with associated Lagrange multiplier $\lambda$. Let $k$ denote the unique null normal field satisfying $ \langle V,k\rangle=-2$,  and assume that $ \theta_V\neq0$,  $\operatorname{Ein}(V,k)\geq0$ and  $\operatorname{Ein}(V,V)\geq0 $ along $\Sigma$. Then \begin{equation} \label{eq:null-direction-sharp} 2\lambda|\Sigma| + \int_\Sigma|\vec H|^2\,d\mu \leq 8\pi\chi(\Sigma). \end{equation} 
If equality holds, then $ \mathrm{Sc}^{\Sigma} = \frac12|\vec H|^2+\lambda$,  $\operatorname{Ein}(V,k)=0$ and $s_V=\nabla^h\log|\theta_V|$. Moreover, on $\{|\vec H|^2>0\}$, $\operatorname{Ein}(V,V)=0$ and $\mathring\chi_V=0$.
\end{theorem}
\begin{proof}
 Since $\theta_k/\theta_V\leq0$, while $\operatorname{Ein}(V,k)\geq0$, $\operatorname{Ein}(V,V)\geq0$, and $\|\mathring\chi_V\|_h^2\geq0$, all the defect terms on the right-hand side of (\ref{eq:null-direction-exact}) are nonpositive. Therefore 
 \[ 2\lambda|\Sigma| + \int_\Sigma|\vec H|^2\,d\mu \leq 8\pi\chi(\Sigma). \]
 Suppose now that equality holds. Then all the defect terms vanish, so  $ \nabla^h\log|\theta_V|-s_V=0$, $\operatorname{Ein}(V,k)=0$, and $-\frac{\theta_k}{\theta_V} \left( \operatorname{Ein}(V,V)+\|\mathring\chi_V\|_h^2 \right)=0$.
 Substituting these identities into \eqref{Wl} gives $ \mathrm{Sc}^{\Sigma} = \frac12|\vec H|^2+\lambda$.  Finally, on $\{|\vec H|^2>0\}$ one has $ -\frac{\theta_k}{\theta_V} = \frac{|\vec H|^2}{\theta_V^2}>0$, 
 and hence equality implies $ \operatorname{Ein}(V,V)+\|\mathring\chi_V\|_h^2=0$. 
 Both terms are nonnegative, so $\mathring\chi_V=0$ and $\operatorname{Ein}(V,V)=0$ on $\{|\vec H|^2>0\}$. 
\end{proof}
\begin{remark}
The sharp null-direction inequality immediately implies
\[
\lambda|\Sigma|
\leq
4\pi\chi(\Sigma).
\]
Moreover, since $\langle V,k\rangle=-2$, the null normals $V$ and $k$ have the
same time orientation. Hence the dominant energy condition implies $
\operatorname{Ein}(V,k)\geq0$ and  $\operatorname{Ein}(V,V)\geq0$. Thus the curvature assumptions in
Theorem~\ref{thm:null-direction-nontimelike} are automatic under the
dominant energy condition.
\end{remark}
We conclude the null-direction analysis by considering the timelike mean-curvature regime, $
|\vec H|^2<0$.  In this case the null expansions have the same sign, so the ratio $ \frac{\theta_k}{\theta_V}$
is positive. The sign structure of the null-direction identity therefore changes, and the sharp topological estimate above is replaced by two-sided geometric bounds.

\begin{proposition}[Null-direction bounds in the timelike mean-curvature regime] \label{prop:timelike-null-direction-bounds} Let $(M,g)$ be a time-oriented $4$-dimensional Lorentzian manifold, and let $\Sigma\subset M$ be a closed, connected spacelike surface with $ |\vec H|^2<0$.  Suppose that $\Sigma$ is area-constrained Willmore in the direction of a nowhere-vanishing null normal field $V$, with associated Lagrange multiplier $\lambda$. Let $k$ be the unique null normal field satisfying $\langle V,k\rangle=-2$, and assume that $ \operatorname{Ein}(V,V)\geq0 $ along $\Sigma$. Then 
\begin{align} \int_\Sigma |\vec H|^2\,d\mu \geq{}& 8\pi\chi(\Sigma)-2\lambda|\Sigma| - 4\int_\Sigma \| \nabla^h\log|\theta_V|-s_V \|_h^2\,d\mu \nonumber\\ &- 2\int_\Sigma \operatorname{Ein}(V,k)\,d\mu , \label{eq:timelike-null-lower-bound} 
\end{align} and 
\begin{align} \int_\Sigma |\vec H|^2\,d\mu \leq{}& 8\pi\chi(\Sigma)-2\lambda|\Sigma| - 2\int_\Sigma \operatorname{Ein}(V,k)\,d\mu \nonumber\\ &+ 2\int_\Sigma \frac{\theta_k}{\theta_V} \left( \operatorname{Ein}(V,V) + \|\mathring\chi_V\|_h^2 \right)d\mu . \label{eq:timelike-null-upper-bound} 
\end{align} 
Equality in \eqref{eq:timelike-null-lower-bound} holds if and only if $ \operatorname{Ein}(V,V)=0$ and $ \mathring\chi_V=0$,  whereas equality in \eqref{eq:timelike-null-upper-bound} holds if and only if $ s_V=\nabla^h\log|\theta_V|$.  
\end{proposition} \begin{proof} Since $|\vec H|^2=-\theta_V\theta_k<0$,  the null expansions $\theta_V$ and $\theta_k$ have the same sign.
Integrating  (\ref{eq:null-direction-pointwise}) over $\Sigma$, using the vanishing of the divergence term and Gauss--Bonnet, gives 
\begin{align} \int_\Sigma|\vec H|^2\,d\mu =& 8\pi\chi(\Sigma)-2\lambda|\Sigma| - 4\int_\Sigma \| \nabla^h\log|\theta_V|-s_V \|_h^2\,d\mu \nonumber\\ &- 2\int_\Sigma \operatorname{Ein}(V,k)\,d\mu + 2\int_\Sigma \frac{\theta_k}{\theta_V} \left( \operatorname{Ein}(V,V) + \|\mathring\chi_V\|_h^2 \right)d\mu . \label{eq:timelike-null-exact} \end{align} Since $\theta_k/\theta_V>0$,  
 $\operatorname{Ein}(V,V)\geq0$ and $\|\mathring\chi_V\|_h^2\geq0$, the final term in \eqref{eq:timelike-null-exact} is nonnegative. Dropping it gives the lower bound \eqref{eq:timelike-null-lower-bound}. Because the coefficient $\theta_k/\theta_V$ is strictly positive, equality holds precisely when $ \operatorname{Ein}(V,V) + \|\mathring\chi_V\|_h^2 =0$,  which, by nonnegativity of both terms, is equivalent to $ \operatorname{Ein}(V,V)=0$ and   $\mathring\chi_V=0$.  On the other hand, the  term $-4\| \nabla^h\log|\theta_V|-s_V \|_h^2$ is nonpositive. Dropping it gives the upper bound \eqref{eq:timelike-null-upper-bound}, and equality holds precisely when $s_V=\nabla^h\log|\theta_V|$,  as claimed. 
 \end{proof}

\section{Hawking energy: positivity, rigidity, stationarity, and monotonicity}\label{Hawkingensec}

The Hawking energy was introduced by Hawking in the study of gravitational radiation in expanding universes \cite{Hawma}. It is one of the foundational quasi-local energy functionals in general relativity; see \cite{Living} for a survey of quasi-local energy-momentum constructions.

 For a closed spacelike surface \(\Sigma\), the Hawking energy is defined as
\[
\mathcal E_H(\Sigma)
 =
\sqrt{\frac{|\Sigma|}{16\pi}}
\left(
1-\frac{1}{16\pi}\int_\Sigma |\vec H|^2\,d\mu
\right).
\]
Because the area $\vert{}\Sigma\vert{}$ is fixed under area-preserving variations, the first variation of the Hawking energy is simply a negative constant multiple of the first variation of the Willmore functional $\int_\Sigma \vert{}\vec H\vert{}^2\,d\mu$. Consequently, the area-constrained variational problem for the Willmore functional is equivalent to the area-constrained variational problem for the Hawking energy.

This equivalence provides a direct physical motivation for the
directional Willmore framework developed here. Moreover, the sharp
inequality
\[
\int_\Sigma |\vec H|^2\,d\mu\leq16\pi
\]
is equivalent to the nonnegativity of the Hawking energy,
$\mathcal E_H(\Sigma)\geq0$. The curvature inequalities of
Section~\ref{sectionrigi} therefore have immediate consequences for the
Hawking energy. In particular,
Theorem~\ref{thm:directional-Minkowski-rigidity} gives the following
positivity and rigidity statement.
 \begin{corollary}[Hawking energy positivity and rigidity] \label{cor:directional-Hawking-positivity} 
 Let $(M,g)$ be a time-oriented $4$-dimensional Lorentzian manifold satisfying the dominant energy condition, and let $\Sigma\subset M$ be a closed, connected spacelike surface with $|\vec H|^2>0$.  Suppose that $\Sigma$ is area-constrained Willmore in a connection-compatible spacelike normal direction 
with associated Lagrange multiplier $\lambda\geq0$. Then $$ \mathcal E_H(\Sigma)\geq0.$$  Moreover, if $ \mathcal E_H(\Sigma)=0$,  then $\Sigma$ is isometric to a round sphere and satisfies Minkowski causal-diamond rigidity. \end{corollary}

Rigidity is a fundamental feature in quasi-local mass theory, since the vanishing case is expected to characterize flat spacetime geometry. Related nonnegativity and rigidity results for the Hawking energy have been obtained for stable STCMC surfaces under additional stability and
ambient curvature assumptions \cite{diaz2026curvature}.

Another important setting is provided by surfaces contained in a spacelike hypersurface. In the totally geodesic case, the variational problem reduces to the classical area-constrained Willmore problem, and the resulting surfaces already satisfy sharp positivity and rigidity
properties for the Hawking energy
\cite{willflat,diaz2025rigidity}. In a general spacelike hypersurface, restricting the variational problem to the spatial slice leads to a critical surface in the direction of the normal within the hypersurface, for which analogous rigidity statements require additional slice-dependent auxiliary assumptions
\cite{diaz2025rigidity}. We compare these hypersurface-restricted critical points with the directional Lorentzian theory in Section~\ref{hawsur}.

\subsection{Hawking stationarity } \label{subsec:Hawking-stationarity}  We next consider critical points of the Hawking energy with respect to variations in a prescribed normal direction, without imposing an area-preserving constraint. The resulting stationarity condition admits a direct interpretation in terms of directional area-constrained Willmore criticality.

\begin{proposition}[Hawking stationarity and directional Willmore criticality] \label{prop:Hawking-stationary-directional} Let $(M,g)$ be a time-oriented $4$-dimensional Lorentzian manifold, and let $\Sigma\subset M$ be a closed spacelike surface. Let $V$ be a smooth normal direction along $\Sigma$. Then \[ \delta_{\alpha V}\mathcal E_H(\Sigma)=0 \qquad \text{for every }\alpha\in C^\infty(\Sigma) \] if and only if $\Sigma$ is area-constrained Willmore in the direction $V$ with multiplier $ \lambda =  \frac{4\mathcal E_H(\Sigma)}{r_\Sigma^3}$, where  $r_\Sigma:=\sqrt{\frac{|\Sigma|}{4\pi}} $ denotes the area radius of $\Sigma$.
\end{proposition} 
\begin{proof} 
Write $ V=a\ell+bk $ with respect to a null frame satisfying $\langle\ell,k\rangle=-2$. Then \[ \delta_{\alpha V}|\Sigma| = \int_\Sigma \alpha(a\theta_\ell+b\theta_k)\,d\mu = \int_\Sigma \alpha\langle\vec H,V\rangle\,d\mu, \qquad \delta_{\alpha V}\int_\Sigma|\vec H|^2d\mu= -\int_\Sigma \alpha(aW_\ell+bW_k)\,d\mu. \]
Hence 
\begin{align*} 
(16\pi)^{3/2}\delta_{\alpha V}\mathcal E_H(\Sigma) =& \frac{16\pi- \int_\Sigma|\vec H|^2d\mu} {2\sqrt{|\Sigma|}} \int_\Sigma \alpha\langle\vec H,V\rangle\,d\mu + \sqrt{|\Sigma|} \int_\Sigma \alpha(aW_\ell+bW_k)\,d\mu \\
=& \sqrt{|\Sigma|} \int_\Sigma \alpha \left[ aW_\ell+bW_k + \lambda\langle\vec H,V\rangle \right]d\mu, 
\end{align*}where $\lambda=\frac{16\pi- \int_\Sigma|\vec H|^2d\mu} {2|\Sigma|}$. Therefore stationarity for every $\alpha$ is equivalent to $ aW_\ell+bW_k = -\lambda\langle\vec H,V\rangle$,  which is precisely the directional area-constrained Willmore equation. If $ r_\Sigma:=\sqrt{\frac{|\Sigma|}{4\pi}} $ denotes the area radius, then \[ \mathcal E_H(\Sigma) = \frac{r_\Sigma}{2} \left( 1-\frac{ \int_\Sigma|\vec H|^2d\mu}{16\pi} \right). \] Consequently, the distinguished multiplier may equivalently be written as $\lambda = \frac{4\mathcal E_H(\Sigma)}{r_\Sigma^3}$.  
\end{proof}

In the Riemannian time-symmetric setting, local maximizers of the Hawking mass, as well as of several modified or charged Hawking masses, have been shown to satisfy strong local rigidity properties; see, for example, \cite{maximo2012hawking,barros2017hawking,sousa2023charged,baltazar2023local,lee2025modified}. These results typically rely on a second-order stability or local maximizing condition to obtain rigidity of the surrounding $3$-manifold. Here we consider instead a first-order stationarity problem in a prescribed normal direction.

For the rigidity analysis below, we restrict to the spacelike
mean-curvature regime $|\vec H|^2>0$ and consider spacelike normal
directions of the form $
V=\vec H+\beta\vec H^\star$. 
\begin{theorem}[Rigidity of Hawking stationarity in spacelike directions] \label{thm:Hawking-stationary-rigidity} Let $(M,g)$ be a time-oriented $4$-dimensional Lorentzian manifold, and let $\Sigma\subset M$ be a closed, connected spacelike surface with  $|\vec H|^2>0$.  Let $V=\vec H+\beta\vec H^\star$ be a connection-compatible spacelike normal direction.   Assume that $ \operatorname{Ein} ( \vec H^\star,\vec H^\star+\beta\vec H ) \geq0 $ along $\Sigma$. Then  \[ \delta_{\alpha V}\mathcal E_H(\Sigma)=0 \qquad \text{for every }\alpha\in C^\infty(\Sigma), \]

if and only if $\Sigma$ is isometric to a round sphere, $ (\nabla \vec H)^\perp=0$, $\mathring\chi_\ell=\mathring\chi_k=0$ and $\operatorname{Ein} ( \vec H^\star,\vec H^\star+\beta\vec H ) =0$.  In this case $\Sigma$ is area-constrained Willmore in the direction $V$ with $ \lambda =\frac{4\mathcal E_H(\Sigma)}{r_\Sigma^3}$ and $ \operatorname{Sc}^{\Sigma} = \frac12|\vec H|^2+\lambda$. 
\end{theorem} 
\begin{proof} Assume first that $\Sigma$ is Hawking-stationary in the direction $V$. By Proposition~\ref{prop:Hawking-stationary-directional}, $\Sigma$ is area-constrained Willmore in the direction $V$ with multiplier $ \lambda = \frac{4\mathcal E_H(\Sigma)}{r_\Sigma^3}$,  and hence $ 2\lambda|\Sigma|+ \int_\Sigma|\vec H|^2d\mu=16\pi$.  Applying the sharp curvature inequalities from Theorem~\ref{thm:connection-compatible-spacelike}  we have $2\lambda\vert{}\Sigma\vert{} +  \int_\Sigma|\vec H|^2d\mu \le 8\pi\chi(\Sigma)$. Therefore, 
$$16\pi \le 8\pi\chi(\Sigma).$$
 Since every closed, connected surface satisfies $\chi(\Sigma)\le 2$, we must have $\chi(\Sigma)=2$, and the sharp curvature inequality is saturated.

 Because equality holds, the equality statements of Theorem~\ref{thm:connection-compatible-spacelike} apply. We obtain $(\nabla \vec H)^\perp=0$, $\mathring\chi_\ell=\mathring\chi_k=0$, $\operatorname{Ein} ( \vec H^\star,\vec H^\star+\beta\vec H ) =0$, and $\operatorname{Sc}^{\Sigma} = \frac12\vert{}\vec H\vert{}^2+\lambda$.
 
 Since $(\nabla \vec H)^\perp=0$, the quantity $\vert{}\vec H\vert{}^2$ is constant. Consequently, the scalar curvature $\operatorname{Sc}^{\Sigma}$ is constant. Since $\chi(\Sigma)=2$, the Gauss--Bonnet theorem ($\int_\Sigma \operatorname{Sc}^{\Sigma}d\mu = 8\pi$) implies that this constant scalar curvature must be strictly positive. A topological sphere with constant positive Gaussian curvature is isometric to a round sphere.

 Conversely, suppose that $\Sigma$ is isometric to a round sphere and that \[ (\nabla \vec H)^\perp=0, \qquad \mathring\chi_\ell=\mathring\chi_k=0, \qquad \operatorname{Ein} ( \vec H^\star,\vec H^\star+\beta\vec H ) =0. \] Since $(\nabla \vec H)^\perp=0$, we have $ \nabla^h|\vec H|=0$ and $s_\ell=0$.  Substituting these identities, together with the vanishing shear and Einstein terms, into the expressions of $W$ and $W^\star$   gives  \[ W+\beta W^\star = |\vec H|^2 ( \operatorname{Sc}^{\Sigma} -\frac12|\vec H|^2 ). \] Because $\Sigma$ is a round sphere, $ \operatorname{Sc}^{\Sigma} = \frac{8\pi}{|\Sigma|}$,  and since $|\vec H|^2$ is constant, $  \int_\Sigma|\vec H|^2d\mu=|\vec H|^2|\Sigma|$.  Consequently, $ W+\beta W^\star = \frac{4\mathcal E_H(\Sigma)}{r_\Sigma^3} |\vec H|^2 $.  Thus $\Sigma$ is area-constrained Willmore in the direction $V$ with multiplier $\frac{4\mathcal E_H(\Sigma)}{r_\Sigma^3}$. Proposition~\ref{prop:Hawking-stationary-directional} then implies $ \delta_{\alpha V}\mathcal E_H(\Sigma)=0$ for every $\alpha\in C^\infty(\Sigma)$.  
 \end{proof}

The preceding Theorem does not imply that \(\mathcal E_H(\Sigma)=0\). For example, Schwarzschild symmetry spheres are Hawking-stationary in mean curvature direction $\vec H$ and have positive Hawking energy. For a Hawking-stationary surface, the additional condition $\lambda=0 $ is equivalent to $ \mathcal E_H(\Sigma)=0$,  and it is this zero-energy condition that leads to Minkowski rigidity. 
 \begin{corollary}[Positivity of Hawking-stationary surfaces] \label{cor:Hawking-stationary-positivity} Let $\Sigma$ satisfy the hypotheses of Theorem~\ref{thm:Hawking-stationary-rigidity}, and suppose that $\Sigma$ is Hawking-stationary in the spacelike direction $V$. Assume in addition that $\Sigma$ bounds a compact mean-convex spacelike hypersurface $\Omega\subset M$, and that the dominant energy condition holds along $\Omega$. Then
 \begin{equation} \label{eq:Hawking-Liu-Yau-round-relation} \mathcal E_H(\Sigma) = \frac{r_\Sigma^3}{4}\lambda = \frac12 \left( 1+\frac{r_\Sigma|\vec H|}{2} \right) \mathcal E_{KLY}(\Sigma)\geq 0. 
 \end{equation}
where  \(\mathcal{E}_{KLY}(\Sigma)\) denotes the Kijowski-Liu-Yau energy of \(\Sigma\) and $ r_\Sigma:=\sqrt{\frac{|\Sigma|}{4\pi}} $ the area radius. Moreover, if $\mathcal E_H(\Sigma)=0$, then the induced initial data on $\Omega$ embed into Minkowski spacetime, and the corresponding domain of dependence is isometric to a standard causal diamond.
\end{corollary} 
\begin{proof} By Theorem~\ref{thm:Hawking-stationary-rigidity}, $\Sigma$ is a round sphere of area radius $r_\Sigma$, and its mean curvature vector is parallel. In particular, $|\vec H|$ is constant. The mean curvature of the Euclidean isometric embedding of $\Sigma$ is therefore $ H_0=\frac{2}{r_\Sigma}$.  Hence \[ \mathcal E_{KLY}(\Sigma) = \frac{1}{8\pi} \int_\Sigma \left( H_0-|\vec H| \right)d\mu = r_\Sigma \left( 1-\frac{r_\Sigma|\vec H|}{2} \right). \] 
Moreover, 
\[  \mathcal E_H(\Sigma) = \frac{r_\Sigma}{2} \left( 1-\frac{r_\Sigma^2|\vec H|^2}{4} \right) = \frac12 \left( 1+\frac{r_\Sigma|\vec H|}{2} \right)\mathcal{E}_{KLY}(\Sigma).  \] 
Together with $ \lambda = \frac{4\mathcal E_H(\Sigma)}{r_\Sigma^3}$,  this proves \eqref{eq:Hawking-Liu-Yau-round-relation}. By the Liu-Yau positivity Theorem \ref{liuyaurigi}, $ \mathcal{E}_{KLY}(\Sigma)\geq0 $ under the stated filling and dominant energy hypotheses. Since $ 1+\frac{r_\Sigma|\vec H|}{2}>0$, it follows that $ \mathcal E_H(\Sigma)\geq0 $ and hence $\lambda\geq0$.  Finally, if $\mathcal E_H(\Sigma)=0$, then
$\mathcal E_{KLY}(\Sigma)=0$. By the rigidity statement in
Theorem~\ref{liuyaurigi}, the induced initial data on $\Omega$ admit
an initial-data-preserving embedding into Minkowski spacetime. Since
$\Sigma$ is a round sphere, the Minkowski rigidity argument of
Theorem~\ref{thm:directional-Minkowski-rigidity} shows that the
corresponding domain of dependence is isometric to a standard causal
diamond.
\end{proof}

\begin{theorem}[Hawking stationarity in a null direction] \label{thm:Hawking-stationary-null} Let $(M,g)$ be a time-oriented $4$-dimensional Lorentzian manifold, and let $\Sigma\subset M$ be a closed, connected spacelike surface with $ |\vec H|^2\geq0$.  Suppose that $V$ is a nowhere-vanishing null normal field with $ \theta_V\neq0$,  and let $k$ be the unique null normal field satisfying $ \langle V,k\rangle=-2$.  Assume that $ \operatorname{Ein}(V,k)\geq0$ and $\operatorname{Ein}(V,V)\geq0 $ along $\Sigma$. If $\Sigma$ is Hawking-stationary in the direction $V$, that is, \[ \delta_{\alpha V}\mathcal E_H(\Sigma)=0 \qquad \text{for every }\alpha\in C^\infty(\Sigma). \] 
Then \ $\Sigma$ is topologically a sphere and $ \operatorname{Sc}^{\Sigma} = \frac12|\vec H|^2+\frac{4\mathcal E_H(\Sigma)}{r_\Sigma^3}$. Moreover  on the set $\{|\vec H|^2>0\}$  $ \operatorname{Ein}(V,V)=0$ and  $\mathring\chi_V=0$.  
\end{theorem} 
\begin{proof}
By Proposition~\ref{prop:Hawking-stationary-directional}, Hawking stationarity in the direction $V$ implies that $\Sigma$ is area-constrained Willmore in that direction with multiplier $ \lambda = \frac{4\mathcal E_H(\Sigma)}{r_\Sigma^3}$.  Consequently, $ 2\lambda|\Sigma|+ \int_\Sigma|\vec H|^2d\mu=16\pi$.  Theorem~\ref{thm:null-direction-nontimelike} gives $ 2\lambda|\Sigma|+\int_\Sigma|\vec H|^2d\mu \leq 8\pi\chi(\Sigma)$, and hence $ 16\pi\leq8\pi\chi(\Sigma)$.  Since every closed, connected surface satisfies $\chi(\Sigma)\leq2$, we obtain \[ 16\pi \leq 8\pi\chi(\Sigma) \leq 16\pi. \] Thus $\chi(\Sigma)=2$ and equality holds in the sharp null-direction curvature inequality. The asserted identities now follow directly from the equality case of Theorem~\ref{thm:null-direction-nontimelike}. 
\end{proof}
The preceding results show that unconstrained Hawking stationarity is an extremely restrictive condition. Under the dominant energy condition, Hawking stationarity in any connection-compatible spacelike direction forces the surface to be a round, totally umbilic sphere with parallel mean curvature vector. This already includes, in particular, all constant-angle spacelike directions, since constant $\beta$ is automatically connection-compatible. This stark rigidity underscores the advantage of restricting our variational framework for the Hawking energy to area-constrained variations, which allows for a much richer class of critical surfaces (including, for instance, the nonspherical Clifford tori, see Section~\ref{clifford}).

\subsection{Infinitesimal monotonicity} \label{subsec:infinitesimal-monotonicity} A basic feature expected of a quasi-local mass is monotonicity under outward or area-expanding deformations. General variation formulas and monotonicity criteria for the Hawking energy have been developed in the setting of uniformly expanding and uniformly area-expanding flows; see \cite{bray2007generalized,bray2015time,bray2016time}.

In the Riemannian codimension-one setting, Lamm, Metzger, and Schulze
proved that the Hawking energy is infinitesimally nondecreasing at
area-constrained Willmore surfaces under every area-nondecreasing
variation \cite[Theorem~3.2]{willflat}. In the hypersurface-restricted
Lorentzian setting, a related monotonicity result was obtained in
\cite{penuelafol} for critical surfaces satisfying an additional
integral condition. We revisit this latter
result in Section~\ref{sec:hypersurface}, where the hypersurface variational problem is identified with directional area-constrained Willmore criticality and the auxiliary condition is compared with connection-compatibility; in particular,
Corollary~\ref{cor:Hawking-foliation-monotonicity} yields a
Hawking energy monotonicity result under the connection-compatibility
condition.

Here we first establish the ambient spacetime monotonicity theory.
Directional criticality yields infinitesimal Hawking-energy monotonicity
along area-nondecreasing variations in the prescribed critical direction,
with spacelike and null directions leading to related but distinct
rigidity conclusions.

Suppose that $\Sigma$ is area-constrained Willmore in a prescribed normal direction $V$, with associated multiplier $\lambda$. For variations with initial velocity $\alpha V$, the directional Euler--Lagrange equation gives \[ \delta_{\alpha V} \int_\Sigma|\vec H|^2\,d\mu = \lambda\,\delta_{\alpha V}|\Sigma|. \] Consequently, the first variation of the Hawking energy factors as \begin{equation} \label{eq:Hawking-monotonicity-factor} (16\pi)^{3/2} \delta_{\alpha V}\mathcal E_H(\Sigma) = \frac{\delta_{\alpha V}|\Sigma|} {2\sqrt{|\Sigma|}} \left( 16\pi - 2\lambda|\Sigma| - \int_\Sigma|\vec H|^2\,d\mu \right). \end{equation} Thus infinitesimal monotonicity under area-nondecreasing variations follows whenever \[ 2\lambda|\Sigma| + \int_\Sigma|\vec H|^2\,d\mu \leq16\pi. \]

\begin{theorem}[Directional infinitesimal Hawking-energy monotonicity]
\label{thm:Hawking-monotonicity-directional}
Let $(M,g)$ be a time-oriented $4$-dimensional Lorentzian manifold, and let $\Sigma\subset M$ be a closed, connected spacelike surface with $ |\vec H|^2\geq0$.  Suppose that $\Sigma$ is area-constrained Willmore in a prescribed normal direction $V$, with associated multiplier $\lambda$. Assume that one of the following holds:
\begin{enumerate}
\item[(i)]  $|\vec H|^2>0$ and
$V=\vec H+\beta\vec H^\star$ is a spacelike, connection-compatible direction satisfying  $
\operatorname{Ein}(\vec H^\star,\vec H^\star+\beta\vec H)\ge0$ along $\Sigma$.

\item[(ii)]  $V$ is a nowhere-vanishing null normal direction with $\theta_V\neq0$, and $\operatorname{Ein}(V,V)\geq0$ and
$\operatorname{Ein}(V,k)\geq0$ along $\Sigma$, where $k$ is the unique null normal field satisfying $\langle V,k\rangle=-2$.
\end{enumerate}
Then every variation $\alpha V$ satisfying $
\delta_{\alpha V}|\Sigma|\ge0$ also satisfies
\[
\delta_{\alpha V}\mathcal E_H(\Sigma)\ge0.
\]
If equality holds for a strictly area-increasing variation, then  $\Sigma$ is Hawking-stationary in the direction $V$, it is topologically a sphere and $ \operatorname{Sc}^{\Sigma} = \frac12|\vec H|^2+\lambda$. In addition, for the corresponding case, the following rigidity conclusions hold:
\begin{enumerate}
\item[(i)] 
  $\Sigma$ is isometric to a round sphere, $ (\nabla \vec H)^\perp=0$,  $\mathring\chi_\ell=\mathring\chi_k=0$ and $\operatorname{Ein} ( \vec H^\star,\vec H^\star+\beta\vec H ) =0$.  If, in addition, the dominant energy condition holds and $\mathcal E_H(\Sigma)=0$ (equivalently $\lambda=0$), then $\Sigma$ satisfies Minkowski causal-diamond rigidity. 

\item[(ii)] On ${|\vec H|^2>0}$, $ \mathring\chi_V=0$ and $\operatorname{Ein}(V,V)=0$.  
\end{enumerate}
\end{theorem}
\begin{proof}
In case~{\rm(i)}, Theorem~\ref{thm:connection-compatible-spacelike}
gives
$ 2\lambda|\Sigma| + \int_\Sigma|\vec H|^2\,d\mu \leq 8\pi\chi(\Sigma)  \leq16\pi$.
In case~{\rm(ii)}, the same estimate follows from
Theorem~\ref{thm:null-direction-nontimelike}. Hence in either case the
factor in parentheses in
\eqref{eq:Hawking-monotonicity-factor} is nonnegative. Together with $\delta_{\alpha V}|\Sigma|\geq0$, this proves $ \delta_{\alpha V}\mathcal E_H(\Sigma)\geq0$.  

Suppose now that equality holds and $ \delta_{\alpha V}|\Sigma|>0$. Then \eqref{eq:Hawking-monotonicity-factor} gives $ 2\lambda|\Sigma| + \int_\Sigma|\vec H|^2\,d\mu = 16\pi$,  and hence \[ \lambda = \frac{16\pi-\int_\Sigma|\vec H|^2\,d\mu}{2|\Sigma|} = \frac{4\mathcal E_H(\Sigma)}{r_\Sigma^3}. \] Proposition~\ref{prop:Hawking-stationary-directional} therefore shows that $\Sigma$ is Hawking-stationary in the direction $V$. 

In case~{\rm(i)},  the rigidity conclusion  follows from Theorem~\ref{thm:Hawking-stationary-rigidity}. If moreover $\lambda=0$, then $ \int_\Sigma|\vec H|^2\,d\mu=16\pi$.  Under the dominant energy condition, the equality case of Theorem~\ref{thm:directional-Minkowski-rigidity} gives Minkowski causal-diamond rigidity.

In case~{\rm(ii)}, the rigidity conclusion  follows from Theorem~\ref{thm:Hawking-stationary-null}.
\end{proof}

Notice that no  condition on the Lagrange multiplier is needed for the monotonicity. We now pass from criticality in one distinguished direction to full area-constrained criticality, which upgrades directional monotonicity to infinitesimal Hawking-energy monotonicity under arbitrary area-nondecreasing normal variations.
\begin{corollary}[Infinitesimal monotonicity under full normal criticality] \label{cor:monotofull} Let $(M,g)$ satisfy the dominant energy condition, and let
$\Sigma\subset M$ be a closed, connected area-constrained Willmore surface with $|\vec H|^2\geq0$. Assume that at least one null expansion is nowhere vanishing. Then, for every normal variation $V$ with $ \delta_V|\Sigma|\geq0 $  \[ \delta_V\mathcal E_H(\Sigma)\geq0. \] If equality holds for a strictly area-increasing variation, $|\vec H|^2>0$ and $\lambda=0$, then $\Sigma$ is isometric to a round sphere and satisfies Minkowski causal-diamond rigidity. 
\end{corollary} 
\begin{proof}
Let $
V=a\ell+bk$, $a,b\in C^\infty(\Sigma)$,
be the initial velocity of an arbitrary normal variation, where as usual $\ell$ and $k$ are two null normals. Then
$
\delta_V|\Sigma|
=
\int_\Sigma
(a\theta_\ell+b\theta_k)\,d\mu$.
Since $\Sigma$ is fully area-constrained Willmore, the Euler--Lagrange
equations hold in both null directions with the same multiplier $\lambda$.
By linearity, $
\delta_V\int_\Sigma|\vec H|^2\,d\mu
=
\lambda\,\delta_V|\Sigma|$.
Hence
\begin{equation}
\label{eq:Hawking-monotonicity-full}
(16\pi)^{3/2}\delta_V\mathcal E_H(\Sigma)
=
\frac{\delta_V|\Sigma|}
{2\sqrt{|\Sigma|}}
\left(
16\pi
-
2\lambda|\Sigma|
-
\int_\Sigma|\vec H|^2\,d\mu
\right).
\end{equation}

It remains to show that the factor in parentheses is nonnegative. By
assumption, at least one null expansion is nowhere vanishing; without loss
of generality, suppose that $\theta_\ell\neq0$ along $\Sigma$. Since
$\Sigma$ is fully area-constrained Willmore, it is in particular
area-constrained Willmore in the null direction $\ell$, with the same
multiplier $\lambda$. The dominant energy condition implies the curvature
hypotheses of Theorem~\ref{thm:null-direction-nontimelike}, and therefore
\[
16\pi
-
2\lambda|\Sigma|
-
\int_\Sigma|\vec H|^2\,d\mu \geq 8\pi\chi(\Sigma)
-
2\lambda|\Sigma|
-
\int_\Sigma|\vec H|^2\,d\mu
\geq0.
\]
Since the original variation $V$ is assumed to satisfy $
\delta_V|\Sigma|\geq0$, equation~\eqref{eq:Hawking-monotonicity-full} yields $
\delta_V\mathcal E_H(\Sigma)\geq0$.

Suppose now that equality holds for a strictly area-increasing variation,
so that $
\delta_V\mathcal E_H(\Sigma)=0$ and
$\delta_V|\Sigma|>0$. Then \eqref{eq:Hawking-monotonicity-full} implies $
2\lambda|\Sigma|
+
\int_\Sigma|\vec H|^2\,d\mu
=
16\pi$.
If moreover $\lambda=0$, then $
\int_\Sigma|\vec H|^2\,d\mu=16\pi$.
Since $|\vec H|^2>0$, the mean-curvature direction $\vec H$ is strictly
spacelike. Full area-constrained Willmore criticality implies, in particular,
that $\Sigma$ is area-constrained Willmore in the direction $\vec H$ with
multiplier $\lambda=0$. For this direction one has $\beta=0$, which is constant.
The equality case of
Theorem~\ref{thm:directional-Minkowski-rigidity} therefore applies and shows
that $\Sigma$ is isometric to a round sphere and satisfies Minkowski
causal-diamond rigidity.
\end{proof}

\begin{remark}[The $\vec H^\star$-direction and infinitesimal conservation]
The conjugate mean curvature
vector $\vec H^\star$ determines a distinguished normal direction which is infinitesimally area-preserving:
\[
\delta_{\alpha\vec H^\star}|\Sigma|=0
\qquad
\text{for every }\alpha\in C^\infty(\Sigma).
\]
Consequently, the first variation of the Hawking energy in this direction is
determined entirely by the variation of the Willmore term. 
\[
\delta_{\alpha\vec H^\star}\mathcal E_H(\Sigma)
=
-\sqrt{\frac{|\Sigma|}{16\pi}}\frac{1}{16\pi}
\delta_{\alpha\vec H^\star}
\int_\Sigma|\vec H|^2\,d\mu .
\]
Thus directional Willmore criticality in the $\vec H^\star$-direction is
equivalent to infinitesimal conservation of the Hawking energy under all
$\alpha \vec H^\star$-variations.

This gives a complementary picture to the directional monotonicity results
above. Spacelike and null critical directions generally change the area to
first order and, under the corresponding curvature hypotheses, lead to
one-sided monotonicity of the Hawking energy. By contrast, the
$\vec H^\star$-direction   is strictly area-preserving
to first order, which instead leads to an infinitesimal conservation law.  For general
timelike directions, the directional curvature identity has no definite-sign defect structure, and hence no analogous unconditional monotonicity statement follows from the present arguments.
\end{remark}

\subsection{Hawking-energy monotonicity along uniformly area-expanding flows} In the Riemannian setting, Geroch observed that the Hawking energy is monotonically nondecreasing along smooth solutions of inverse mean curvature flow \cite{geroch1973energy}, a property that plays a central role in the proof of the Riemannian Penrose inequality of Huisken and Ilmanen \cite{huisken2001inverse}. 
Spacetime analogues were considered by Frauendiener \cite{frauendiener2001penrose} and, more generally, in the theory of uniformly expanding and uniformly area-expanding flows; see \cite{bray2007generalized,bray2015time,bray2016time}.

Assume that $\vert{}\vec H\vert{}^2>0$. In our conventions, a natural uniformly area-expanding spacelike velocity is\begin{equation} \label{eq:uniformly-area-expanding-direction}\xi_\beta = \frac{\vec H+\beta\vec H^\star}{|\vec H|^2}, \qquad |\beta|<1.\end{equation}Note that $\langle\vec H,\xi_\beta\rangle=1$, and hence $\delta_{\xi_\beta}d\mu=d\mu$ and $\delta_{\xi_\beta}\vert{}\Sigma\vert{}=\vert{}\Sigma\vert{}$. Moreover, $\langle\vec H^\star,\xi_\beta\rangle=-\beta$.

Up to sign conventions for the mean curvature vector and its conjugate, the condition that $\beta$ be constant on each leaf corresponds to the dual inverse mean curvature condition appearing in \cite{bray2007generalized}. Together with uniform area expansion, this generates the uniformly expanding flows considered there. 

As discussed in Remark~\ref{rem:connection-compatible-Hawking},
constant $\beta$ and time flatness are two classical mechanisms for
Hawking-energy monotonicity and correspond to distinguished equality
cases of connection-compatibility.  More generally, connection-compatibility gives the
natural sign condition for Hawking-energy monotonicity along uniformly area-expanding flows.

\begin{proposition}[Hawking-energy monotonicity along connection-compatible flows]
\label{prop:Hawking-monotonicity-connection-compatible}
Let $(M,g)$ satisfy the dominant energy condition, and let $\Sigma$ be a
closed, connected spacelike surface with $|\vec H|^2>0$. Consider the uniformly
area-expanding spacelike velocity
\[
\xi_\beta
=
\frac{\vec H+\beta\vec H^\star}{|\vec H|^2},
\qquad
|\beta|<1.
\]
If the normal direction determined by $\xi_\beta$ is
connection-compatible, then
\[
\delta_{\xi_\beta}\mathcal E_H(\Sigma)\geq0.
\]
Consequently, along any smooth uniformly area-expanding spacelike flow
$\{\Sigma_s\}$ whose flow direction is connection-compatible on every
leaf, $
s\longmapsto \mathcal E_H(\Sigma_s)$
is nondecreasing.
\end{proposition}
\begin{proof}
For the velocity $\xi_\beta$, linearity of the first variation gives$$\delta_{\xi_\beta} \int_\Sigma\vert{}\vec H\vert{}^2\,d\mu = \int_\Sigma \frac{W+\beta W^\star}{\vert{}\vec H\vert{}^2}\,d\mu.$$Since $\delta_{\xi_\beta}\vert{}\Sigma\vert{}=\vert{}\Sigma\vert{}$, the variation of the Hawking energy is
\begin{equation} \label{eq:Hawking-variation-uniformly-area-expanding}\delta_{\xi_\beta}\mathcal E_H(\Sigma) = \sqrt{\frac{|\Sigma|}{16\pi}} \left[ \frac12 - \frac{1}{16\pi} \left( \frac{1}{2}\int_\Sigma |\vec H|^2,d\mu + \int_\Sigma \frac{W+\beta W^\star}{|\vec H|^2}d\mu \right) \right].
\end{equation}
Integrating the pointwise expression for
$(W+\beta W^\star)/|\vec H|^2$ from
Proposition~\ref{prop:directional-mean-curvature-gauge}, as in the proof
of Proposition~\ref{prop:exact-directional-defect}, we obtain
\[
\begin{aligned}
\delta_{\xi_\beta}\mathcal E_H(\Sigma)
=&
\sqrt{\frac{|\Sigma|}{16\pi}}
\Bigg[
\frac{2-\chi(\Sigma)}{4} +\frac{1}{16\pi}
\int_\Sigma (1+\beta)
\|\nabla^h\log|\vec H|-s_\ell\|_h^2 \\&+(1-\beta)\|\nabla^h\log|\vec H|+s_\ell\|_h^2 -\frac{1}{8\pi}\beta\,\operatorname{div}_h s_\ell
+\frac{1}{8\pi}\operatorname{Ein}
\Big(
\frac{\vec H^\star}{|\vec H|},
\frac{\vec H^\star+\beta\vec H}{|\vec H|}
\Big)
\\
&+\frac{1}{32\pi}
\left(
(1+\beta)\|\mathring\chi_\ell\|_h^2
+
(1-\beta)\|\mathring\chi_k\|_h^2
\right)\, d\mu
\Bigg].
\end{aligned}
\]
Since $\Sigma$ is closed and connected, $
\chi(\Sigma)\leq2$.
Moreover, $|\beta|<1$, so $1\pm\beta>0$, and the weighted square and
shear terms are nonnegative. The dominant energy condition gives $ \operatorname{Ein} ( \vec H^\star ,\vec H^\star+\beta\vec H ) \geq0$. Finally, connection-compatibility gives $
\int_\Sigma
\beta\,\operatorname{div}_h s_\ell\,d\mu
\leq0$.  Every term on the right-hand side is thus nonnegative, and consequently $
\delta_{\xi_\beta}\mathcal E_H(\Sigma)\geq0$. If $\xi_\beta$ is connection-compatible on every leaf of a smooth
uniformly area-expanding spacelike flow $\{\Sigma_s\}$, the same
inequality applies on each leaf, and hence $
\frac{d}{ds}\mathcal E_H(\Sigma_s)\geq0$. Thus $s\mapsto\mathcal E_H(\Sigma_s)$ is nondecreasing.
\end{proof}

\section{Hypersurface formulation}\label{sec:hypersurface}

In this section, we express the directional Euler--Lagrange equations in terms of the geometry of a spacelike hypersurface. This hypersurface formulation serves two main purposes. First, it recovers classical Riemannian Willmore theory as the totally geodesic spacelike case. Second, it shows how the Lorentzian framework incorporates variational problems restricted to a fixed spacelike hypersurface. Critical points of the area-constrained Willmore problem under purely spatial variations have been studied in \cite{Alex,diaz2023local,diaz2025rigidity,penuelafol}, where they are
typically referred to as Hawking surfaces. We will show that these surfaces are precisely  directionally area-constrained Willmore in the spatial normal direction of the surface within the hypersurface.

\subsection{The spacelike hypersurface form of the directional equations} \label{subsec:hypersurface-directional-equations}

Let $(M^4,\mathbf g)$ be a Lorentzian manifold and let $(M^3,g,K)$ be a spacelike hypersurface with induced Riemannian metric $g$ and second fundamental form $K$. Let $\Sigma \subset M^3$ be a smooth closed surface with induced metric $h$.

Denote by $n$ the future-pointing unit normal to $M^3$ and by $\nu$ the outer unit normal to $\Sigma$ in $M^3$. The associated null normals are defined by
\begin{equation}
    \ell = n + \nu, \qquad
    k = n - \nu,
\end{equation}
so that $\langle \ell, k \rangle = -2$. In this setting, the null expansions are
\begin{equation}
    \theta_\ell = H + P, \qquad \theta_k = -H + P,
\end{equation}
where $P := \operatorname{tr}_h K = \operatorname{tr}_g K - K(\nu,\nu)$ denotes the trace of $K$ along $\Sigma$. 

For a vector field $X$ tangent to $\Sigma$, the normal connection one-form
associated to $\ell$ is given by
\[
s_\ell(X)
=
-\frac12 \langle k, \nabla_X \ell \rangle
=
\langle \nu, \nabla_X n \rangle
=
K(X,\nu).
\]
Hence, $\|s_{\ell}\|_{h}^{2}= \|K(\cdot, \nu)\|_h^2 $. The  null second fundamental forms are given by
\[
\chi_\ell = K^\top +B, \qquad \chi_k =K^\top - B.
\]
We also denote by $\mathring{B}$ the trace-free second fundamental form of $\Sigma \subset M^3$, and by $\mathring{K}^\top := K^\top-\frac12Ph$ the trace-free part of $K$ restricted to $T\Sigma$. The Einstein constraint equations on $(M^3,g,K)$ take the form
\begin{equation}\label{constrequ}
    \mathrm{Sc}^{M^3} - |K|^2 + (\mathrm{tr}_g K)^2 = 2\,\mu, \qquad \operatorname{div}^M (K - (\mathrm{tr}_g K)g) = J,
\end{equation}
where the energy density $\mu$ and momentum density $J$ are defined along $M$ by$$\mu := \operatorname{Ein}(n,n), \qquad J(X) := -\operatorname{Ein}(n,X) \quad \text{for all } X \in TM^3.$$Here $\operatorname{Ein}$ denotes the Einstein tensor of the ambient spacetime.

Substituting the above geometric identities into 
\eqref{Wl} and \eqref{Wk}, we obtain the corresponding
equations for surfaces that are area-constrained Willmore in the null
directions \(\ell\) and \(k\), respectively:
\begin{equation}\label{critical l hy}
\begin{aligned}
   -\lambda (H+P) =& 2 \Delta_h(H+P)  + 2(H+P)\Big( \|K(\cdot, \nu)\|_h^2  -\frac{1}{2}\mathrm{Sc}^\Sigma - \operatorname{div}_h (K(\nu, \cdot))\Big) + 2H\mu \\
   &- 4K(\nabla^h (H+P), \nu) + (H-P)\Big(\frac{1}{2}(H+P)^2 + \|\mathring{B} +\mathring{K}^\top\|_h^2 - 2J(\nu) \Big)  \\
   &- 2P\,\operatorname{Ein}(\nu,\nu) 
\end{aligned}
\end{equation}
and 
\begin{equation}\label{critical k hy}
\begin{aligned}
   \lambda (H-P) =& 2 \Delta_h(P-H)  + 2(P-H)\Big( \|K(\cdot, \nu)\|_h^2 -\frac{1}{2}\mathrm{Sc}^\Sigma + \operatorname{div}_h (K(\nu, \cdot))\Big) - 2H\mu \\
   &+ 4K(\nabla^h (P-H), \nu) - (H+P)\Big( \frac{1}{2}(P-H)^2 + \|\mathring{B} -\mathring{K}^\top\|_h^2 + 2J(\nu) \Big)  \\
   &- 2P\,\operatorname{Ein}(\nu,\nu).
\end{aligned}
\end{equation}
The spacetime mean-curvature vector and its dual are given in the slice-adapted null frame by
\begin{equation*}
\vec H=H\nu-Pn, \qquad \vec H^\star=Hn-P\nu, 
\end{equation*} 
and hence $ |\vec H|^2=H^2-P^2$.  Equations \eqref{critical l hy} and \eqref{critical k hy} give the
slice-adapted expressions for the null Euler--Lagrange quantities
$W_\ell$ and $W_k$, respectively. They are the two null components from
which the general directional area-constrained Willmore equation is
assembled.  Let $
V=a\ell+bk$ be an arbitrary prescribed normal direction. Since
\[
\langle\vec H,V\rangle
=
a(H+P)+b(P-H),
\]
the directional Euler--Lagrange equation takes the form
\begin{equation}
\label{eq:directional-Willmore-initial-data}
aW_\ell+bW_k
=
-\lambda
\bigl(
a(H+P)+b(P-H)
\bigr),
\end{equation}
where $W_\ell$ and $W_k$ are given explicitly by the right-hand sides of
\eqref{critical l hy} and \eqref{critical k hy}.

Suppose now that $|\vec H|^2=H^2-P^2\neq0$ and
$\langle V,\vec H\rangle\neq0$. Then the normal direction determined by
$V$ is generated by $
\vec H+\beta\vec H^\star$,  $\beta=-\frac{\langle V,\vec H^\star\rangle}
       {\langle V,\vec H\rangle}$.
In terms of the hypersurface geometry, 
\begin{equation} \label{eq:direction-beta-initial-data} \vec H+\beta\vec H^\star = (H-\beta P)\nu+(\beta H-P)n. 
\end{equation}
Equivalently, in the slice-adapted null frame, 
\[ \vec H+\beta\vec H^\star = \frac{1+\beta}{2}(H-P)\ell - \frac{1-\beta}{2}(H+P)k. \]
Substituting this expression into
\eqref{eq:directional-Willmore-initial-data} gives
\begin{equation}
\label{eq:beta-directional-Willmore-initial-data}
\lambda(H^2-P^2)
=
\frac12
\left[
(1-\beta)(H+P)W_k
-
(1+\beta)(H-P)W_\ell
\right].
\end{equation}
 The case $\beta=0$ is the spacetime mean-curvature direction $V=\vec H$, while $\beta=\pm1$ yields the two null directions after rescaling.

 The cases $\beta=0$ and $V=\vec H^\star$ are distinguished throughout
the paper. Since the corresponding canonical quantities $W$ and $W^\star$ play a central role in the Lorentzian theory, we record here their slice-adapted expressions, even though the general directional
equation \eqref{eq:beta-directional-Willmore-initial-data} already contains this information. Since  $\nabla_X^\perp n=K(X,\nu)\nu$ and $\nabla_X^\perp\nu=K(X,\nu)n$,  we have \begin{equation*} \nabla_X^\perp\vec H = \left( X(H)-P K(X,\nu) \right)\nu - \left( X(P)-H K(X,\nu) \right)n. \end{equation*}
Consequently, 
\begin{equation} \label{eq:normal-gradient-H-initial-data}  
\|(\nabla \vec H)^\perp\|_h^2 = \left\| \nabla^hH-PK(\cdot,\nu) \right\|_h^2  - \left\| \nabla^hP-HK(\cdot,\nu) \right\|_h^2. 
\end{equation} The curvature term becomes 
\begin{equation*} 
\operatorname{Ein}(\vec H^\star,\vec H^\star) = H^2\mu + 2HP\,J(\nu) + P^2\operatorname{Ein}(\nu,\nu), \end{equation*}
The hypersurface representation of the distinguished specialization $\beta=0$ of the general directional equation \eqref{eq:beta-directional-Willmore-initial-data}  leads us to: 
\begin{equation} \label{critical mean hy}
\begin{aligned} 
W=& -\Delta_h(H^2-P^2) + 2\left\| \nabla^hH-PK(\cdot,\nu) \right\|_h^2 - 2\left\| \nabla^hP-HK(\cdot,\nu) \right\|_h^2  - 2H^2\mu - 4HP\,J(\nu)  \\ &  - 2P^2\operatorname{Ein}(\nu,\nu) + (H^2-P^2)\operatorname{Sc}^{\Sigma}  - \frac{(H-P)^2}{2} \left\|K^\top+B\right\|_h^2 - \frac{(H+P)^2}{2} \left\|K^\top-B\right\|_h^2. \end{aligned} 
\end{equation} 
Similarly, the dual canonical component $
W^\star
=
\frac12\bigl(\theta_kW_\ell-\theta_\ell W_k\bigr) $ admits the following
slice-adapted expression:
\begin{equation}
\label{eq:Wstar-initial-data}
\begin{aligned}
W^\star={}&
2P\Delta_hH-2H\Delta_hP
+2(H^2-P^2)\operatorname{div}_hK(\nu,\cdot)
+4H K(\nabla^hH,\nu)
-4P K(\nabla^hP,\nu)
\\
&+2HP\left(
\mu+\operatorname{Ein}(\nu,\nu)
+\|\mathring B\|_h^2
+\|\mathring K^\top\|_h^2
\right)+2(H^2+P^2)
\left(
J(\nu)
-\langle\mathring B,\mathring K^\top\rangle_h
\right).
\end{aligned}
\end{equation}
\begin{remark}
Note that the equations above involve the spacetime curvature component
\(\operatorname{Ein}(\nu,\nu)\), which cannot be expressed solely in terms of the
initial data \((M^3,g,K)\). Consequently, they are genuinely Lorentzian equations
and depend on the embedding of the initial data set into the ambient spacetime.
\end{remark}
\begin{lemma}[Connection-compatibility in initial-data variables]
\label{lem:connection-compatible-initial-data}
Assume that $
|\vec H|^2=H^2-P^2>0$.
In the canonical mean-curvature frame,
\begin{equation}\label{conechyp}
    s_{\ell_{\vec H}}
=
K(\cdot,\nu)
+
\frac{P\,dH-H\,dP}{H^2-P^2}.
\end{equation}
Consequently, the spatial normal direction $\nu$  is connection-compatible if and only if
\begin{equation}
\label{eq:connection-compatible-spatial-normal}
\int_\Sigma
\left(
K(\nabla^h\beta,\nu)
-
\frac{\|\nabla^h\beta\|_h^2}{1-\beta^2}
\right)
\,d\mu
\geq0,
\qquad
\beta=\frac{P}{H}.
\end{equation}
\end{lemma}
\begin{proof}
The canonical frame is related to the slice-adapted frame by the null boost $
\ell_{\vec H}=a\ell$, $
a=
\sqrt{\left|
\frac{H-P}{H+P}
\right|}$. Under a positive boost it holds $s_{\ell_{\vec H}}
=
s_\ell+d\log a $,   since $ s_\ell=K(\cdot,\nu)$ we obtain
\[
s_{\ell_{\vec H}}
=
K(\cdot,\nu)
+
\frac12
d\log\left|
\frac{H-P}{H+P}
\right| =K(\cdot,\nu)
+
\frac{P\,dH-H\,dP}{H^2-P^2}.
\]
By definition, the direction is connection-compatible if and only if $
\int_\Sigma
\beta\,\operatorname{div}_h s_{\ell_{\vec H}}\,d\mu
\leq0$.  Substituting the explicit expression for $s_{\ell_{\vec H}}$ obtained above and using integration by parts  yields 
\begin{equation}
\label{eq:connection-compatible-initial-data-gradient}
\int_\Sigma
K(\nabla^h\beta,\nu)  
+\frac{P \langle
\nabla^h\beta,\nabla^h H\rangle-H\,\langle
\nabla^h\beta,\nabla^hP\rangle}{H^2-P^2}
 d\mu
\geq0.
\end{equation}
For the spatial normal direction, $
\nu= \frac{H}{H^2-P^2}(
\vec H+\frac{P}{H}\vec H^\star)
$,
so $\beta=P/H$. Since $
P\nabla^hH-H\nabla^hP
=
-H^2\nabla^h\beta$
and $
H^2-P^2=H^2(1-\beta^2)$,
equation
\eqref{eq:connection-compatible-initial-data-gradient}
reduces to
\eqref{eq:connection-compatible-spatial-normal}.
\end{proof}
In particular, using \eqref{conechyp} and \eqref{eq:Wstar-initial-data}  $W^\star$ may be written more
geometrically as
\[
\begin{aligned}
W^\star={}&
2\operatorname{div}_h
\left((H^2-P^2)s_{\ell_{\vec H}}\right)+
2HP\left(
\mu+\operatorname{Ein}(\nu,\nu)
+\|\mathring B\|_h^2
+\|\mathring K^\top\|_h^2
\right)
\\
&+
2(H^2+P^2)
\left(
J(\nu)
-\langle\mathring B,\mathring K^\top\rangle_h
\right).
\end{aligned}
\]
\subsection{Relation with classical Willmore surfaces}\label{Willcla}

Assume  that \(\Sigma\) lies in a totally geodesic spacelike hypersurface, so
that \(K=0\). Then $P=0$, $\mathring{K}^\top=0$, $J=0$, and the constraint equation reduces to  $2\mu = \mathrm{Sc}^{M^3}$. Moreover, the Gauss equation for \(\Sigma\subset M^3\) gives   $2\mu = \mathrm{Sc}^\Sigma + 2\mathrm{Ric}^{M^3}(\nu,\nu) - \frac{1}{2}H^2 + \|\mathring{B}\|_h^2$.

In this case the two null components of the directional Euler--Lagrange operator simplify considerably. Indeed, substituting these identities into \eqref{critical l hy} and \eqref{critical k hy} gives 
\begin{equation} \label{eq:null-components-time-symmetric} 
 \frac{W_\ell}{2}=\Delta_hH + H\|\mathring B\|_h^2 + H\,\operatorname{Ric}^{M^3}(\nu,\nu)=- \frac{W_k}{2}, 
\end{equation}
Let $V=a\ell+bk$ be a prescribed normal direction. Since
$\vec H=H\nu$, $\vec H^\star=Hn$, and
$\langle\vec H,V\rangle=H(a-b)$, the general directional equation
\eqref{eq:directional-Willmore-initial-data} becomes
\[
(a-b)\bigl(\frac{\lambda}{2}H
+\Delta_hH
+H\|\mathring B\|_h^2
+H\,\operatorname{Ric}^{M^3}(\nu,\nu)\bigr)=0.
\]
Thus, for every direction with $a-b\neq0$, directional
area-constrained Willmore criticality is equivalent to
 the classical area-constrained Willmore criticality in
$(M^3,g)$. When $a-b=0$, the directional equation is identically
satisfied; for $H\neq0$, this is precisely the dual mean-curvature
direction $\vec H^\star$.

Consequently, in a totally geodesic spacelike hypersurface the
Lorentzian directional system contains only one nontrivial equation,
namely the classical area-constrained Willmore equation.

\subsection{Hypersurface-restricted Willmore criticality}\label{hawsur}
We next consider area-constrained critical points of the functional 
\[
\int_\Sigma (H^2-P^2)\,d\mu= \int_\Sigma |\vec H|^2\,d\mu
\]
under variations restricted to a fixed spacelike hypersurface $M^3$. Following the terminology of \cite{diaz2025rigidity}, we refer to these
hypersurface-restricted critical points as \emph{Hawking surfaces}. From the present point of view, they admit a simple interpretation: they are precisely directional area-constrained Willmore surfaces in the spatial normal direction $\nu$ of $\Sigma$ within $M^3$.

Restricting the variational problem to the hypersurface $M^3$ means that $\Sigma$ is varied only in the normal direction $\nu$. Since $ \nu=\frac12(\ell-k)$,  the general directional equation \eqref{eq:directional-Willmore-initial-data} becomes \begin{equation} \label{eq:Hawking-surface-directional} \frac12\left(W_\ell-W_k\right) = -\lambda H. \end{equation} After substituting the initial-data expressions \eqref{critical l hy}--\eqref{critical k hy} and using the constraint equations, this is equivalent to 
\begin{equation}\label{eulag}
\begin{split}
0=& \frac{\lambda}{2} H+\Delta_h H+H\|\mathring{B}\|_h^2+H\,\mathrm{Ric}^{M^3}(\nu,\nu)
+P\bigl(\nabla_\nu \operatorname{tr}_gK-\nabla_\nu K(\nu,\nu)\bigr) \\
&
-2P\,\operatorname{div}_h(K(\cdot,\nu))
+\frac12HP^2
-2K(\nabla^hP,\nu).
    \end{split}
\end{equation}
 Thus, equation \eqref{eulag} is the usual Hawking-surface equation. This identifies the hypersurface-restricted problem as a distinguished
special case of the Lorentzian directional theory: the choice of the
spacelike hypersurface selects the spatial normal direction $\nu$.

\begin{remark}
The direction $\nu =\frac12(\ell-k)$ is exceptional among general normal
directions. Although the directional Euler--Lagrange equation generally
depends on the ambient curvature component
$\operatorname{Ein}(\nu,\nu)$, this dependence disappears for the spatial
normal direction $\nu$. In this case  the $\operatorname{Ein}(\nu,\nu)$ terms in $W_\ell$ and $W_k$ cancel when combining them according to $\frac12(\ell-k)$. This is consistent with the fact
that the resulting variational problem is entirely determined by the fixed
initial data set $(M^3,g,K)$.
\end{remark}
There is by now a substantial existence theory for these surfaces.
Their concentration behavior was studied in \cite{Alex}, while local
foliations were constructed in \cite{diaz2023local}. More recently,
large foliations at infinity by Hawking surfaces have been constructed
in asymptotically Schwarzschild initial data sets; see
\cite{Friedrich2020, penuelafol}. From the present perspective, these results provide
existence and foliation theory for directional area-constrained
Willmore surfaces in the spatial normal direction $\nu$. As we show
below, connection-compatibility also gives a direct criterion for
Hawking-energy monotonicity along such foliations.

The directional interpretation also allows us to revisit the curvature
estimate for Hawking surfaces obtained in \cite{diaz2025rigidity}.
There, under the dominant energy condition and an additional integral
hypothesis, one obtained
\[
\int_\Sigma|\vec H|^2\,d\mu\leq16\pi
\]
A refined version of that additional  hypothesis was formulated in terms of the function
\begin{equation}
\label{eq:f-tilde-Hawking-surface}
\begin{aligned}
\widetilde f
:={}&
2\frac{P}{H}
K(\nabla^h\log|H|,\nu)
+\frac12(\operatorname{tr}_gK)^2
-\frac34P^2
-
\frac{P}{H}
\left(
\nabla_\nu\operatorname{tr}_gK
-
\nabla_\nu K(\nu,\nu)
\right) \\
&
-\frac12|K|_g^2
-\frac12\|\mathring B\|_h^2
-|J|_g ,
\end{aligned}
\end{equation}
through the condition
\begin{equation}
\label{eq:f-tilde-integral-condition}
\int_\Sigma
\widetilde f-\frac{\lambda}{2}
\,d\mu
\leq0,
\end{equation}
where the factor $\lambda/2$ reflects the normalization of the
Lagrange multiplier in \eqref{eulag}.

The role of this hypothesis was to control the terms of indefinite sign arising from the hypersurface Euler--Lagrange equation. In
\cite{diaz2025rigidity}, however, no direct geometric interpretation of
this condition was identified, and it was suggested that it might impose
more positivity than is required for the curvature estimate. The following identity gives an exact geometric formulation  of \eqref{eq:f-tilde-integral-condition} and shows that it controls an additional nonnegative defect beyond the Hawking-energy term itself. 
\begin{proposition}[Exact defect identity for the refined hypersurface condition]
\label{prop:refined-hypersurface-condition}
Let $\Sigma$ be a Hawking surface satisfying $
H^2-P^2>0$.
Then
\begin{equation}
\label{eq:f-tilde-topological-defect}
\int_\Sigma
\widetilde f-\frac{\lambda}{2}
\, d\mu
=
\frac14\int_\Sigma H^2-P^2\,d\mu
-2\pi\chi(\Sigma)
+\int_\Sigma \|\nabla^h\log|H|\|_h^2
+\mu-|J|_g
\,d\mu .
\end{equation}
Consequently, $\int_\Sigma
\widetilde f-\frac{\lambda}{2}
\,d\mu
\leq0$ is equivalent to
\begin{equation}
\label{eq:f-tilde-curvature-form}
\int_\Sigma H^2-P^2 \,d\mu
+
4\int_\Sigma
\|\nabla^h\log|H|\|_h^2
+\mu-|J|_g
\, d\mu
\leq
8\pi\chi(\Sigma),
\end{equation}
where the second integral is nonnegative under the dominant energy condition.
\end{proposition}
\begin{proof}
Since $H^2-P^2>0$, the function $H$ is nowhere vanishing. Dividing
the Hawking-surface equation \eqref{eulag} by $H$ and integrating over
$\Sigma$ gives
\begin{align*}
0=\int_\Sigma
&
\frac{\lambda}{2}
+\frac{\Delta_hH}{H}
+\|\mathring B\|_h^2
+\operatorname{Ric}^{M^3}(\nu,\nu)+
\frac{P}{H}
\left(
\nabla_\nu\operatorname{tr}_gK
-
\nabla_\nu K(\nu,\nu)
\right)
\\
&-2\frac{P}{H}\operatorname{div}_h(K(\cdot,\nu))
+\frac12P^2
-\frac{2}{H}K(\nabla^hP,\nu)
\, d\mu .
\end{align*}
Note that 
\begin{align*}
2\int_\Sigma
-\frac{P}{H}\operatorname{div}_h(K(\cdot,\nu))
-\frac{1}{H}K(\nabla^hP,\nu)\,d\mu &=
2\int_\Sigma
K\left(\nabla^h\frac{P}{H},\nu\right)
-\frac{1}{H}K(\nabla^hP,\nu)\,d\mu
\\
&=
-2\int_\Sigma
\frac{P}{H}
K(\nabla^h\log|H|,\nu)\,d\mu.
\end{align*}
Moreover, $\int_\Sigma\frac{\Delta_hH}{H}\,d\mu
=
\int_\Sigma\|\nabla^h\log|H|\|_h^2\,d\mu$ and the Gauss equation gives
\[
2\operatorname{Ric}^{M^3}(\nu,\nu)
=
\operatorname{Sc}^{M^3}
-
\operatorname{Sc}^{\Sigma}
+
\frac12H^2
-
\|\mathring B\|_h^2.
\]
Substituting this and the preceding integration-by-parts identity yields
\begin{align*}
0=\int_\Sigma
&
\frac{\lambda}{2}
+\|\nabla^h\log|H|\|_h^2
+\frac12\|\mathring B\|_h^2
+\frac12
\left(
\operatorname{Sc}^{M^3}
-
\operatorname{Sc}^{\Sigma}
\right)
+
\frac{P}{H}
\left(
\nabla_\nu\operatorname{tr}_gK
-
\nabla_\nu K(\nu,\nu)
\right)
\\
&+\frac14H^2
+\frac12P^2-
2\frac{P}{H}
K(\nabla^h\log|H|,\nu)\,d\mu .
\end{align*}
Using the Einstein  constraint $
\operatorname{Sc}^{M^3}
-|K|_g^2
+(\operatorname{tr}_gK)^2
=
2\mu$
and the definition \eqref{eq:f-tilde-Hawking-surface}, the integral
can be rewritten as
\begin{equation}
\label{eq:f-tilde-exact-defect}
0 = \int_\Sigma \frac{\lambda}{2} -\widetilde f +\|\nabla^h\log|H|\|_h^2 +\mu-|J|_g +\frac14(H^2-P^2) -\frac12\operatorname{Sc}^{\Sigma}
\,d\mu .
\end{equation}
Finally, Gauss--Bonnet, $
\int_\Sigma\operatorname{Sc}^{\Sigma}\,d\mu
=
4\pi\chi(\Sigma)$,
gives \eqref{eq:f-tilde-topological-defect}. Rearranging proves
\eqref{eq:f-tilde-curvature-form}.
\end{proof}
\begin{remark}[Comparison with connection-compatibility]
\label{rem:connection-compatible-vs-f}

For a Hawking surface with $H^2-P^2>0$, the spatial normal direction
$\nu$ corresponds to $
\beta=\frac{P}{H}$.
Hence, by Lemma~\ref{lem:connection-compatible-initial-data}, $\nu$ is
connection-compatible precisely when
\begin{equation}
\label{eq:connection-compatible-Hawking-beta}
\int_\Sigma
K(\nabla^h\beta,\nu)
-
\frac{\|\nabla^h\beta\|_h^2}{1-\beta^2}
\,d\mu
\geq0.
\end{equation}

This condition is closely related to the mixed first-order term appearing
in the refined function $\widetilde f$. Indeed, expanding the gradient gives
\[
K(\nabla^h\beta,\nu)
=
\frac1H K(\nabla^hP,\nu)
-
\beta K(\nabla^h\log|H|,\nu).
\]
The integration by parts in the proof of Proposition~\ref{prop:refined-hypersurface-condition} shows that the terms involving $\operatorname{div}_h(K(\cdot,\nu))$ and $K(\nabla^hP,\nu)$ combine precisely into the term $-2\beta K(\nabla^h\log|H|,\nu)$, which enters $\widetilde f$. Thus both conditions arise from the same
gradient--normal-connection interaction, which in the Lorentzian
formulation is encoded intrinsically by the canonical normal connection.

The advantage of this formulation becomes clear from
Proposition~\ref{prop:refined-hypersurface-condition}. The refined
condition of \cite{diaz2025rigidity} is equivalent to (\ref{eq:f-tilde-curvature-form}), where under the dominant energy condition, the second integral is nonnegative. Thus the refined condition controls not only the desired curvature quantity, but also an additional nonnegative defect.

By contrast, the present Lorentzian argument only requires the favorable
sign of the canonical normal-connection term
\eqref{eq:connection-compatible-Hawking-beta}; together with the
dominant energy condition, this yields
\[
2\lambda|\Sigma| + \int_\Sigma H^2-P^2\,d\mu
\leq
8\pi\chi(\Sigma).
\]
Thus the two hypotheses are not equivalent. In particular, if $P/H$ is constant, then
\eqref{eq:connection-compatible-Hawking-beta} holds automatically with
equality, while the refined condition still imposes the additional
global restriction above. This distinction is realized concretely by
the FLRW Clifford tori of Section~\ref{clifford}, which are
connection-compatible but, in the spacelike regime under the dominant
energy condition, do not satisfy the refined condition
\eqref{eq:f-tilde-integral-condition}, see Remark \ref{torihawwill}.
\end{remark}
\begin{corollary}[Curvature estimate for Hawking surfaces]
\label{cor:Hawking-surfaces-directional-estimate}
Let $(M^3,g,K)$ be an initial data set satisfying the dominant energy
condition, and let $\Sigma\subset M^3$ be a closed, connected Hawking
surface with $
H^2-P^2>0$.
Assume that the spatial normal direction $\nu$ is
connection-compatible, equivalently,
\[
\int_\Sigma
K(\nabla^h\beta,\nu)
-
\frac{\|\nabla^h\beta\|_h^2}{1-\beta^2}
\,d\mu
\geq0,
\qquad
\beta=\frac{P}{H}.
\]
Then
\begin{equation}
\label{eq:Hawking-surface-sharp-estimate}
2\lambda|\Sigma|
+
\int_\Sigma H^2-P^2\,d\mu
\leq
8\pi\chi(\Sigma)
\leq16\pi .
\end{equation}
In particular, if $\lambda\geq0$, then $
\int_\Sigma H^2-P^2\,d\mu\leq16\pi$, and if equality holds, then $\Sigma$ is a round sphere. If, in addition, $\Sigma$ bounds a compact mean-convex region
$\Omega\subset M^3$, then the induced initial data on $\Omega$ admit
an initial-data-preserving embedding into Minkowski spacetime.
\end{corollary}
\begin{proof}
Note that $
\operatorname{Ein}
\left(
\vec H^\star,
\vec H^\star+\beta\vec H
\right)
=
(H^2-P^2)\bigl(\mu+\beta J(\nu)\bigr)$.
Since $H^2-P^2>0$, one has $|\beta|<1$, and the dominant energy
condition implies $
\mu+\beta J(\nu)\geq \mu-|J|_g\geq0$. Then the inequality follows from  Theorem~\ref{thm:connection-compatible-spacelike}. For the equality case, as in the proof of Theorem~\ref{thm:directional-Minkowski-rigidity}  one concludes that $\mathcal E_{KLY}(\Sigma)=0$. By the rigidity statement in
Theorem~\ref{liuyaurigi}, the induced initial data on $\Omega$ embed
into Minkowski spacetime.
\end{proof}
The rigidity conclusion substantially strengthens the corresponding
result for Hawking surfaces in \cite{diaz2025rigidity}. There, rigidity
required a stronger auxiliary integral
condition involving the modified quantity $\widetilde{f}_\beta$, introduced in order to force additional vanishing of $K$ along $\Sigma$. Here this
condition is replaced by connection-compatibility of the spatial normal direction. Thus the Lorentzian directional formulation yields the rigidity directly from the canonical normal-bundle condition.

The preceding estimate has an immediate consequence for the
monotonicity of the Hawking energy along families of Hawking surfaces.
This is particularly relevant for the foliations constructed in
\cite{penuelafol}, where monotonicity was obtained under an additional
integral condition on the initial data.
\begin{corollary}[Hawking-energy monotonicity along Hawking foliations]
\label{cor:Hawking-foliation-monotonicity}
Let $(M^3,g,K)$ satisfy the dominant energy condition, and let
$\{\Sigma_s\}_{s\in I}$ be a smooth family of closed, connected Hawking
surfaces satisfying $
H_s^2-P_s^2>0$. Assume that the spatial normal direction $\nu_s$ is
connection-compatible on every leaf. If the normal velocity of the
family is $
\partial_sF=\alpha_s\nu_s $
and $ \int_{\Sigma_s}\alpha_s H_s\,d\mu_s\geq0$, then
\[
\frac{d}{ds}\mathcal E_H(\Sigma_s)\geq0.
\]
In particular, the Hawking energy is nondecreasing along any
area-nondecreasing foliation by connection-compatible Hawking surfaces.
\end{corollary}
\begin{proof}
At a Hawking surface, the first variation of the Hawking energy in the
direction $\alpha\nu$ is
\[
(16\pi)^{3/2}\,
\delta_{\alpha\nu}\mathcal E_H(\Sigma)
=
\frac{1}{2|\Sigma|^{1/2}}
\left(
\int_\Sigma\alpha H\,d\mu
\right)
\left(
16\pi
-2\lambda|\Sigma|
-\int_\Sigma H^2-P^2\,d\mu
\right).
\]
By Corollary~\ref{cor:Hawking-surfaces-directional-estimate},
connection-compatibility and the dominant energy condition imply
\[
2\lambda|\Sigma|
+\int_\Sigma H^2-P^2\,d\mu
\leq8\pi\chi(\Sigma)
\leq16\pi.
\]
Hence the second factor is nonnegative, while the first is
nonnegative by assumption.
\end{proof}
\begin{remark}[Relation with the monotonicity condition of \cite{penuelafol}]
\label{rem:Hawking-foliation-old-condition}
For a spherical Hawking surface, let $\mathcal M(\Sigma) := 16\pi - 2\lambda|\Sigma| - \int_\Sigma H^2-P^2 \,d\mu$. Hawking-energy monotonicity along area-nondecreasing variations follows from $\mathcal M(\Sigma)\geq0$. In \cite{penuelafol}, this  was obtained from the auxiliary condition
$\int_\Sigma\widetilde f\,d\mu\leq0$. By
Proposition~\ref{prop:refined-hypersurface-condition}, under the dominant energy condition this implies
\[
\mathcal M(\Sigma) \geq 4\int_\Sigma \left( \|\nabla^h\log|H|\|_h^2+\mu-|J|_g \right) d\mu \geq 0.
\]
Thus, the refined condition forces the monotonicity defect to dominate an additional nonnegative quantity.

The condition used in \cite{penuelafol} and
connection-compatibility control different combinations of the terms in the Hawking-surface Euler--Lagrange equation. Both provide sufficient mechanisms for obtaining $ \mathcal M(\Sigma)\geq0$, but neither is a reformulation of the other. The distinction is genuine: the FLRW Clifford tori of
Section~\ref{clifford} provide examples for which
connection-compatibility holds while the refined hypersurface condition fails.
\end{remark}

\section{Directional examples}\label{examplsec}

\subsection{Directional Willmore spheres in spherical symmetry} \label{subsec:spherical-symmetry}
We next describe a natural class of exact directional area-constrained Willmore surfaces. Let $(M,g)$ be a spherically symmetric Lorentzian manifold. Locally, the metric can be written as a warped product 
\begin{equation}\label{eq:spherical-warped-product} 
   M=Q\times\mathbb S^2,
\qquad
g=\overline g+r^2g_{\mathbb S^2}.
\end{equation} 
where \((Q,\overline g)\) is a \(2\)-dimensional Lorentzian manifold, \(r:Q\rightarrow(0,\infty)\) is the area-radius function.  For each \(q\in Q\), let 
$ \Sigma_q := \{q\}\times\mathbb S^2 $ 
be the corresponding symmetry sphere. Its induced metric is $ h_q = r(q)^2g_{\mathbb S^2}$,  
and therefore $ \operatorname{Sc}^{\Sigma_q} = \frac{2}{r(q)^2}$.

The geometry of a symmetry sphere is invariant under the action of \(\operatorname{SO}(3)\). In particular, its null second fundamental forms are pure trace $
    \mathring\chi_\ell = \mathring\chi_k = 0$.  Moreover, \(|\vec H|^2\) is constant on each symmetry sphere, and   in the canonical mean curvature frame, the connection one-form $s_\ell=0$ vanishes. 

It follows that symmetry spheres are directionally critical in the entire constant $\beta$ family.
\begin{proposition}[Directional Willmore symmetry spheres] \label{prop:symmetry-spheres-directional-Willmore} Let $\Sigma_q$ be a symmetry sphere with $|\vec H|^2>0$. For every constant $\beta\in\mathbb R$, the sphere $\Sigma_q$ is area-constrained Willmore in the direction $ V_\beta = \vec H+\beta\vec H^\star$.  Its directional Lagrange multiplier is \begin{equation} \label{eq:lambda-beta-spherical-symmetry}  \lambda_\beta = \frac{2}{r^2} -\frac12|\vec H|^2 -2\operatorname{Ein} \big(\frac{\vec H^\star}{|\vec H|},\frac{\vec H^\star}{|\vec H|}+\beta\frac{\vec H}{|\vec H|}\big)  = \frac{4\mathcal E_H(\Sigma_q)}{r^3} -2\operatorname{Ein} \big(\frac{\vec H^\star}{|\vec H|},\frac{\vec H^\star}{|\vec H|}+\beta\frac{\vec H}{|\vec H|}\big) 
\end{equation} where  $r=r(q)$. 
\end{proposition} \begin{proof} For a symmetry sphere, all gradient, normal-connection, and trace-free shear terms in the directional equation vanish. Hence \[ W+\beta W^\star = |\vec H|^2 \big( \operatorname{Sc}^{\Sigma_q} -\frac12|\vec H|^2 -2\operatorname{Ein} \big(\frac{\vec H^\star}{|\vec H|},\frac{\vec H^\star}{|\vec H|}+\beta\frac{\vec H}{|\vec H|}\big) \big). \] Since every quantity on the right-hand side is constant on $\Sigma_q$, this is precisely the directional area-constrained Willmore equation $ W+\beta W^\star = \lambda_\beta|\vec H|^2 $.  Using  $\operatorname{Sc}^{\Sigma_q}=\frac{2}{r^2}$ gives the first expression for $\lambda_\beta$. Since  $|\Sigma_q|=4\pi r^2 $ and $|\vec H|^2$ is constant on $\Sigma_q$, its Hawking energy is $ \mathcal E_H(\Sigma_q) = \frac r2 \left( 1-\frac{r^2}{4}|\vec H|^2 \right)$.  Equivalently, $ |\vec H|^2 = \frac{4}{r^2} \left( 1-\frac{2\mathcal E_H(\Sigma_q)}{r} \right)$,  and therefore $ \frac{2}{r^2} -\frac12|\vec H|^2 = \frac{4\mathcal E_H(\Sigma_q)}{r^3}$.  Substitution gives the second expression in \eqref{eq:lambda-beta-spherical-symmetry}.
\end{proof} 
Thus spherical symmetry gives a particularly transparent distinction between directional and full area-constrained Willmore criticality. Although every constant $\beta$ determines a directional criticality equation, the corresponding multiplier  depends on the chosen normal direction. 
\begin{equation} \label{eq:lambda-beta-affine-spherical} \lambda_\beta = \lambda_0 - 2\beta\,\operatorname{Ein}(\frac{\vec H^\star}{|\vec H|},    \frac{\vec H}{|\vec H|}), \end{equation} where $\lambda_0 = \frac{2}{r^2} -\frac12|\vec H|^2 -2\operatorname{Ein}(\frac{\vec H^\star}{|\vec H|},\frac{\vec H^\star}{|\vec H|})$.  In particular, the mean curvature direction corresponds to $\beta=0$.

Comparing the coefficients of $\beta$ in the directional equation gives \begin{equation} \label{eq:Wstar-spherical-symmetry} W^\star = -2|\vec H|^2\operatorname{Ein}(\frac{\vec H^\star}{|\vec H|},    \frac{\vec H}{|\vec H|}). \end{equation} Consequently, the directional multipliers are independent of $\beta$ if and only if  $\operatorname{Ein}(\frac{\vec H^\star}{|\vec H|},    \frac{\vec H}{|\vec H|})=0$,  or equivalently, $ W^\star=0$.  Since the equation for $W$ is already satisfied, this
is exactly the condition for full area-constrained Willmore criticality.
\begin{corollary}[Full criticality of symmetry spheres] \label{cor:full-Willmore-symmetry-spheres} Let $\Sigma_q$ be a symmetry sphere with $|\vec H|^2>0$. Then $\Sigma_q$ is fully area-constrained Willmore if and only if \begin{equation} \label{eq:full-Willmore-spherical-condition} \operatorname{Ein}(\frac{\vec H^\star}{|\vec H|},    \frac{\vec H}{|\vec H|})=0. \end{equation} In this case the common Lagrange multiplier is \[ \lambda = \frac{4\mathcal E_H(\Sigma_q)}{r^3} - 2\operatorname{Ein}(\frac{\vec H^\star}{|\vec H|},\frac{\vec H^\star}{|\vec H|}). \] In particular, every such symmetry sphere contained in a vacuum region is fully area-constrained Willmore, with $ \lambda = \frac{2}{r^2} -\frac12|\vec H|^2$.  \end{corollary}

\textbf{Minkowski and Schwarzschild symmetry spheres.} In Minkowski spacetime, \[ \mathcal E_H(\Sigma_q)=0, \qquad \operatorname{Ein}=0. \] Hence every symmetry sphere with spacelike mean-curvature vector is fully area-constrained Willmore with $ \lambda=0$.  More generally, for every coordinate sphere in the exterior region of the Schwarzschild spacetime of mass $m$, $ \mathcal E_H(\Sigma_q)=m$ and  $\operatorname{Ein}=0$.  Consequently, these spheres are fully area-constrained Willmore with common multiplier \begin{equation} \label{eq:lambda-Schwarzschild} \lambda = \frac{4m}{r^3} >0. \end{equation} Thus Schwarzschild coordinate spheres provide a natural family of full area-constrained Willmore surfaces satisfying strictly the sign condition appearing in the rigidity results above.

\subsection{Directional criticality of Kerr coordinate spheres}
\label{subsec:Kerr-directional}
The Kerr geometry provides a different type of directional example.
Consider the Kerr metric in Boyer--Lindquist coordinates
$(t,r,\theta,\varphi)$, and let $
\Sigma_R:=\{t=t_0,\ r=R\}$ 
be a coordinate sphere outside the horizon ($R>r_+$) in a  time slice. The induced
metric on each time slice is diagonal,
\[
g_{M^3}
=
g_{rr}\,dr^2
+
g_{\theta\theta}\,d\theta^2
+
g_{\varphi\varphi}\,d\varphi^2,
\]
while the only nonzero mixed time--space component of the Kerr metric
is $g_{t\varphi}$. Consequently, the future unit normal to the slices
$t=\mathrm{const}$ has the form
\[
n=N^{-1}
\bigl(
\partial_t+\omega(r,\theta)\partial_\varphi
\bigr)
\]
for suitable functions $N$ and $\omega$.

Since $\partial_t$ and $\partial_\varphi$ are Killing fields, the
second fundamental form
\[
K(X,Y)=\frac12(\mathcal L_n g)(X,Y),
\qquad X,Y\in TM^3,
\]
is generated only by derivatives of the coefficient $\omega$.
Moreover, $\omega$ depends only on $(r,\theta)$ and
$\partial_\varphi$ is orthogonal to both $\partial_r$ and
$\partial_\theta$. It follows that the only possibly nonzero
components of $K$ are $
K_{r\varphi}$ and $ K_{\theta\varphi}$.

In particular, the restriction of $K$ to $\Sigma_R$ has only the
off-diagonal component $K_{\theta\varphi}$. Since the induced metric
$h$ on $\Sigma_R$ is diagonal, $
P=\operatorname{tr}_{\Sigma_R}K=0$, and therefore $|\vec H|^2=H^2$.
\begin{proposition}[Dual criticality and time-flatness of Kerr coordinate spheres]
\label{prop:Kerr-dual-directional}
Every Boyer--Lindquist coordinate sphere $\Sigma_R$ with $H\neq0$ is
Willmore critical in the dual mean-curvature direction
$\vec H^\star$. Moreover, $\Sigma_R$ is time-flat.  In particular, every normal direction generated by $\vec H+\beta\vec H^\star$ is
connection-compatible.
\end{proposition}
\begin{proof}
Since Kerr is vacuum and $P=0$, the dual Euler--Lagrange quantity reduces to
\begin{equation}
\label{eq:Wstar-Kerr}
W^\star
=
2H^2\operatorname{div}_h s_\ell
+
4H K(\nabla^hH,\nu)
-
2H^2
\langle\mathring B,\mathring K^\top\rangle_h .
\end{equation}

On $\Sigma_R$, the unit normal within the Boyer--Lindquist time slice is
proportional to $\partial_r$. Since the only components of $K$ involving
$\partial_r$ are of the form $K_{r\varphi}$, the normal-connection
one-form $
s_\ell=K(\,\cdot\,,\nu)$ has only a $d\varphi$ component. Its coefficient is independent of $\varphi$ by axisymmetry. Since the induced metric $h$ on $\Sigma_R$ is
diagonal and also independent of $\varphi$,
\[
\operatorname{div}_h s_\ell
=
\frac{1}{\sqrt{\det h}}
\partial_\varphi
\left(
\sqrt{\det h}\,s_\ell^\varphi
\right)
=0.
\]
Since $P=0$, by Lemma \ref{lem:connection-compatible-initial-data} the slice-adapted null frame agrees with the canonical
mean-curvature frame, and hence $
s_{\ell_{\vec H}}=s_\ell$. Therefore $\operatorname{div}_h s_{\ell_{\vec H}}=0$, so $\Sigma_R$ is
time-flat.

Moreover, $H$ is independent of $\varphi$, so $
\nabla^h H
=
h^{\theta\theta}(\partial_\theta H)\,\partial_\theta$.
Since $K_{\theta r}=0$, it follows that $
K(\nabla^hH,\nu)=0$.

Finally, the second fundamental form $B$ of $\Sigma_R$ in the
Boyer--Lindquist slice is diagonal in the coordinates
$(\theta,\varphi)$. Indeed, $\nu$ is proportional to $\partial_r$ and
the induced metric $h$ is diagonal, so the mixed component
$B_{\theta\varphi}$ vanishes. On the other hand, the tangential
restriction of $K$ to $\Sigma_R$ has only the mixed component
$K_{\theta\varphi}$. Since $P=0$, one has $
\mathring K^\top=K^\top$,
and therefore $
\langle\mathring B,\mathring K^\top\rangle_h=0$.

All three terms in \eqref{eq:Wstar-Kerr} vanish, and hence $
W^\star=0$. 
\end{proof}
Since $W^\star=0$, the directional equation for $
V_\beta=\vec H+\beta\vec H^\star $
reduces, for every finite $\beta$, to $
W=\lambda H^2$. Thus  area-constrained Willmore in direction $V_\beta$ is equivalent
to the constancy of $W/H^2$ on $\Sigma_R$.
For rotating Kerr, $a\neq0$, a direct computation shows that $W/H^2$ depends nontrivially on $\theta$. Evaluating the resulting expression at the pole and the
equator gives distinct values for every $a\neq0$. Hence no constant $\lambda$ can satisfy the
equation, and $\Sigma_R$ is not area-constrained Willmore in any finite-$\beta$ direction, while it remains unconstrained Willmore in the timelike direction $\vec{H}^\star$.
\begin{corollary}[Hawking-energy preservation and monotonicity in Kerr]
\label{cor:Kerr-Hawking-energy}
Let $\Sigma_R$ be a Boyer--Lindquist coordinate sphere with $H\neq0$.
Then
\[
\delta_{\alpha\vec H^\star}\mathcal E_H(\Sigma_R)=0
\qquad
\text{for every }\alpha\in C^\infty(\Sigma_R).
\]
Moreover, for every uniformly area-expanding spacelike direction $
\xi_\beta
=
\frac{\vec H+\beta\vec H^\star}{|\vec H|^2}$, $|\beta|<1$,
one has
\[
\delta_{\xi_\beta}\mathcal E_H(\Sigma_R)\geq0.
\]
\end{corollary}
\begin{proof}
The first statement follows from $W^\star=0$ and
$\langle\vec H,\vec H^\star\rangle=0$. The second follows from the
time-flatness of $\Sigma_R$, which makes every direction
$\vec H+\beta\vec H^\star$ connection-compatible, together with
Proposition~\ref{prop:Hawking-monotonicity-connection-compatible}.
\end{proof}
In the Schwarzschild limit, the rotational part of $K$ disappears and the
coordinate spheres recover the fully area-constrained Willmore criticality
of the spherically symmetric setting. Thus Kerr rotation provides a simple
example in which full criticality is lost while criticality in the dual
mean curvature direction survives.

The Kerr horizon provides a complementary null directional Willmore example.  A cross-section
of the event horizon is a stationary marginally outer trapped surface (MOTS), satisfying $
\theta_\ell=0$, $\mathring\chi_\ell=0$. Since Kerr is vacuum, $\operatorname{Ein}=0$, and therefore $
W_\ell=0$.
Hence the horizon cross-section is directionally Willmore in the null direction $\ell$. Moreover,
$\vec H=-\frac{\theta_k}{2}\ell$ and $
\vec H^\star=-\frac{\theta_k}{2}\ell$, so the mean-curvature and dual mean-curvature directions coalesce into this same null direction.

\subsection{Clifford tori in FLRW manifolds} \label{clifford} We now give an explicit family of nonspherical directional area-constrained Willmore surfaces. Let \[ (\mathcal N^{3+1},\mathbf g) = (\mathbb R\times\mathbb S^3,-dt^2+a(t)^2g_{\mathbb S^3}) \] be a FLRW Lorentzian manifold, where $a(t)>0$. Fix a time $t_0$ and set $ \mathcal H_0:=\frac{\dot a(t_0)}{a(t_0)}$.  With our sign convention, the second fundamental form of the slice $ M_{t_0}:=\{t=t_0\}$  is 
\[ K=\mathcal H_0 g^{M_{t_0}}. \] 
Let $a_0=a(t_0)$ and identify $M_{t_0}$ with the round sphere $\mathbb S^3_{a_0}$. For $0<\alpha<\pi/2$, consider the Clifford torus \[ T_\alpha = \left\{ (z_1,z_2)\in\mathbb S^3_{a_0}\subset\mathbb C^2: |z_1|=a_0\cos\alpha,\quad |z_2|=a_0\sin\alpha \right\}. \] Equivalently, $T_\alpha$ is a product torus with radii $ r_1=a_0\cos\alpha$ and $ r_2=a_0\sin\alpha$.  Its principal curvatures in the round slice are constant and given by 
\begin{equation}\label{principalcurv}
    \kappa_1=\frac{r_2}{a_0r_1}=\frac{\tan\alpha}{a_0},
\qquad
\kappa_2=-\frac{r_1}{a_0r_2}=-\frac{\cot\alpha}{a_0}.
\end{equation}
Consequently,
\[ H = \kappa_1+\kappa_2 = \frac1{a_0}(\tan\alpha-\cot\alpha)\quad  \text{and} \quad \|\mathring B\|_h^2 = \frac12(\kappa_1-\kappa_2)^2. \]
Both quantities are constant on $T_\alpha$. Moreover, the induced metric on $T_\alpha$ is flat, so $ \operatorname{Sc}^{T_\alpha}=0$. 
Since $ K=\mathcal H_0g^{M_{t_0}}$, 
we have 
\[ P=\operatorname{tr}_{T_\alpha}K=2\mathcal H_0, \qquad K(\cdot,\nu)|_{TT_\alpha}=0, \qquad \mathring K^\top=0, \qquad J=0. \]
In particular, $H$ and $P$ are constant on each torus. Spatial isotropy of the FLRW geometry also implies that $ \mu=\operatorname{Ein}(n,n)$ and $\operatorname{Ein}(\nu,\nu) $ 
are constant on the slice.

Using these identities in the hypersurface expressions for $W$ and $W^\star$, gives \begin{equation} \label{eq:W-FLRW-Clifford}  W = -\frac12(H^2-P^2)^2 -(H^2+P^2)\|\mathring B\|_h^2 -2H^2\mu -2P^2\operatorname{Ein}(\nu,\nu),
\end{equation} and 
\begin{equation} \label{eq:Wstar-FLRW-Clifford} W^\star = 2HP \left( \|\mathring B\|_h^2 +\mu +\operatorname{Ein}(\nu,\nu) \right). \end{equation} In particular, both $W$ and $W^\star$ are constant on $T_\alpha$. This immediately yields an entire family of directional criticality equations.   
\begin{proposition}[Directional Willmore Clifford tori] \label{prop:directional-FLRW-Clifford} Let $T_\alpha\subset M_{t_0}\cong\mathbb S^3_{a_0}$ be a Clifford torus as above. If $H^2-P^2\neq0$, then, for every constant $\beta\in\mathbb R$, the torus
$T_\alpha$ is area-constrained Willmore in the direction $ V_\beta=\vec H+\beta\vec H^\star$.  The corresponding directional Lagrange multiplier is $\lambda_\beta = \frac{W+\beta W^\star}{H^2-P^2}$,  or equivalently, \begin{equation} \label{eq:lambda-beta-FLRW-Clifford-expanded} \begin{aligned} \lambda_\beta(H^2-P^2) ={}& -\frac12(H^2-P^2)^2 - \bigl(H^2+P^2-2\beta HP\bigr) \|\mathring B\|_h^2\\&
- 2H(H-\beta P)\mu - 2P(P-\beta H)\operatorname{Ein}(\nu,\nu). \end{aligned} 
\end{equation} 
If $H^2-P^2=0$, the normalized $\beta$-description degenerates.
Nevertheless, $T_\alpha$ is area-constrained Willmore in every constant
normal direction transverse to $\vec H$.
\end{proposition}
\begin{proof}
For $H^2-P^2\neq0$, equations
\eqref{eq:W-FLRW-Clifford} and
\eqref{eq:Wstar-FLRW-Clifford} show that $W$ and $W^\star$ are constant on
$T_\alpha$. Hence $
W+\beta W^\star
=
\lambda_\beta(H^2-P^2) $ with constant $\lambda_\beta$, which is precisely the normalized
directional Euler--Lagrange equation. Expanding $W+\beta W^\star$ gives
\eqref{eq:lambda-beta-FLRW-Clifford-expanded}.

If $H^2-P^2=0$, the same FLRW identities show that $W_\ell$ and $W_k$
are constant on $T_\alpha$. Thus, for every constant normal direction
$V=a\ell+bk$ satisfying $\langle V,\vec H\rangle\neq0$, the directional
equation holds with the constant multiplier
\[
\lambda_V
=
-\frac{aW_\ell+bW_k}{\langle V,\vec H\rangle}.
\]
Since $\vec H$ is null, the condition
$\langle V,\vec H\rangle\neq0$ is equivalent to $V$ being transverse to
the null line generated by $\vec H$.
\end{proof}
For the non-null members, the directional multiplier depends affinely on
the prescribed direction:
\begin{equation}
\label{eq:lambda-beta-affine-Clifford}
\lambda_\beta
=
\lambda_0+\beta\frac{W^\star}{H^2-P^2}.
\end{equation}

Consequently, a non-null Clifford torus is fully area-constrained
Willmore if and only if $W^\star=0$. In particular, the standard
minimal Clifford torus satisfies this condition.
\begin{corollary}[Full criticality of the minimal Clifford torus]
\label{cor:minimal-Clifford-full-criticality}
Assume that $\mathcal H_0\neq0$. Then the minimal Clifford torus
$T_{\pi/4}$ is fully area-constrained Willmore. Its mean-curvature vector
is timelike, $
|\vec H|^2=-P^2=-4\mathcal H_0^2<0$, and its common Lagrange multiplier is
\begin{equation}
\label{eq:minimal-Clifford-common-lambda}
\lambda
=
\frac12P^2
+
\|\mathring B\|_h^2
+
2\operatorname{Ein}(\nu,\nu).
\end{equation}
If, in addition, the ambient FLRW manifold satisfies the dominant energy
condition, then $T_{\pi/4}$ is the unique fully area-constrained Willmore
member of the Clifford family.
\end{corollary}
\begin{proof}
For $\alpha=\pi/4$ one has $H=0$. Hence
\eqref{eq:Wstar-FLRW-Clifford} gives $
W^\star=0$. Thus the directional multiplier is independent of the normal direction,
and the torus is fully area-constrained Willmore. Since
$H^2-P^2=-P^2$, equation \eqref{eq:W-FLRW-Clifford} gives $W=-P^2\left(\frac12P^2+\|\mathring B\|_h^2+2\operatorname{Ein}(\nu,\nu)\right)$, which yields \eqref{eq:minimal-Clifford-common-lambda}.

Suppose now that the dominant energy condition holds and
$\alpha\neq\pi/4$. Then $H\neq0$, while
$P=2\mathcal H_0\neq0$. Spatial isotropy and the dominant energy
condition give $
\mu+\operatorname{Ein}(\nu,\nu)\geq0$.
Since $\|\mathring B\|_h^2>0$, equation
\eqref{eq:Wstar-FLRW-Clifford} implies $
W^\star\neq0$. Since full area-constrained Willmore criticality necessarily implies
$W^\star=0$, no nonminimal member of the family can be fully critical.
\end{proof}

The spacelike members of the family exhibit the opposite sign from that
required in the rigidity results above.

\begin{corollary}[Negative directional multipliers]
\label{cor:negative-directional-multiplier-Clifford}
Assume that the ambient FLRW manifold satisfies the dominant energy
condition, and let $T_\alpha$ satisfy $
H^2-P^2>0$. Then for every $|\beta|\leq1$, $
\lambda_\beta<0 $. 
\end{corollary}
\begin{proof}
Equation \eqref{eq:lambda-beta-FLRW-Clifford-expanded} can be written as
\[
\begin{aligned}
\lambda_\beta(H^2-P^2)
={}&
-\frac12(H^2-P^2)^2
-\bigl(H^2+P^2-2\beta HP\bigr)\|\mathring B\|_h^2
\\
&-2\operatorname{Ein}
\bigl(
\vec H^\star,\vec H^\star+\beta\vec H
\bigr).
\end{aligned}
\]
For $|\beta|\leq1$, $
H^2+P^2-2\beta HP
=
(H-\beta P)^2+(1-\beta^2)P^2>0$.
Moreover, since $\vec H$ is spacelike,
$\vec H^\star$ and $\vec H^\star+\beta\vec H$ are causal with the same
time orientation for $|\beta|\leq1$. The dominant energy condition
therefore gives $
\operatorname{Ein}
\bigl(
\vec H^\star,\vec H^\star+\beta\vec H
\bigr)\geq0$.
Hence $
\lambda_\beta(H^2-P^2)<0$.
Since $H^2-P^2>0$, it follows that $\lambda_\beta<0$.
\end{proof}

The existence of nonspherical critical points for the area-constrained
Willmore functional has been extensively studied in the Riemannian
setting, most notably by Ikoma--Malchiodi--Mondino through the perturbative
construction of small tori in compact $3$-manifolds
\cite{ikoma2016embedded,ikoma2017embedded}. The FLRW family $T_\alpha$
provides an explicit Lorentzian directional analogue of this nonspherical
critical-point phenomenon.

The preceding corollary uses the dominant energy condition. In the
thin-torus regime, negativity in the mean-curvature direction and the
corresponding instability can be obtained without any energy condition.
Indeed, as $\alpha\to0$ or $\alpha\to\pi/2$, $
|H|\longrightarrow\infty$.
Hence, for sufficiently thin tori,
\[
H^2-P^2>0,
\qquad
\int_{T_\alpha}|\vec H|^2\,d\mu>16\pi.
\]
The following proposition records the stronger instability property of
these examples.

\begin{proposition}[Thin FLRW Clifford tori]
\label{prop:instability_tori}
Let $T_\alpha\subset M_{t_0}\cong\mathbb S^3_{a_0}$ be the Clifford tori
considered above. Then, for $\alpha>0$ sufficiently small, $T_\alpha$ is
STCMC and area-constrained Willmore in every constant direction $
V_\beta=\vec H+\beta\vec H^\star$,
$\beta\in\mathbb R$. Moreover, $\lambda_\beta<0$ for every $|\beta|\leq1$.
Finally, $T_\alpha$ is both constant-mode unstable and variationally
unstable as an STCMC surface, in the sense of
\cite[Definition 3.2]{diaz2026curvature}.
\end{proposition}
\begin{proof}
Since $H$ and $P$ are constant on $T_\alpha$, $
|\vec H|^2=H^2-P^2$
is constant, and hence $T_\alpha$ is STCMC. Moreover,
$|H|\to\infty$ as $\alpha\to0$, whereas $P=2\mathcal H_0$ is fixed on the
slice. Thus $
H^2-P^2>0 $
for all sufficiently small $\alpha$. In particular, Proposition
\ref{prop:directional-FLRW-Clifford} shows that $T_\alpha$ is
area-constrained Willmore in every constant direction $V_\beta$. From \eqref{principalcurv} one has
\[
\|\mathring B\|_h^2
=
\frac12(\kappa_1-\kappa_2)^2
=
\frac12(\kappa_1+\kappa_2)^2
-2\kappa_1\kappa_2
=
\frac12H^2+\frac{2}{a_0^2}.
\]
Substituting this into
\eqref{eq:lambda-beta-FLRW-Clifford-expanded} gives, uniformly for
$|\beta|\leq1$,
\[
\lambda_\beta(H^2-P^2)
=
-H^4+O(H^3)
\qquad\text{as }\alpha\to0.
\]
Since $
H^2-P^2=H^2+O(1)>0$,
it follows that $
\lambda_\beta<0 $
for every $|\beta|\leq1$, provided $\alpha$ is sufficiently small.

It remains to prove the STCMC instability. In the mean curvature
direction $\beta=0$,
\[
W=\lambda_0|\vec H|^2=-H^4+O(H^2)<0.
\]
Taking the constant mode $u\equiv1$ gives
\[
\delta^2_{\vec H}|T_\alpha|
=\int_{T_\alpha} W\,d\mu=
W\,|T_\alpha|<0.
\]
Thus $T_\alpha$ is constant-mode unstable. For variational instability, choose the zero-average test function
\[
u(\phi_1,\phi_2)=\cos\phi_1.
\]
 Since $ h=r_1^2\,d\phi_1^2+r_2^2\,d\phi_2^2$, where $
r_1=a_0\cos\alpha$ and $r_2= a_0\sin\alpha$ one has
\[
\int_{T_\alpha}|\nabla^h u|^2\,d\mu
= \int_{T_\alpha}\frac{\sin^2\phi_1}{r_1^2}d\mu =
\frac1{r_1^2}\int_{T_\alpha}u^2\,d\mu.
\]
Hence
\[
\delta^2_{u\vec H}|T_\alpha|=\int_{T_\alpha}\left(2|\vec H|^2|\nabla^h u|^2+W u^2\right)d\mu =\left(\frac{2|\vec H|^2}{r_1^2}+W\right)\int_{T_\alpha}u^2\,d\mu
\]
As $\alpha\to0$, $r_1\to a_0$, $
\frac{2|\vec H|^2}{r_1^2}=O(H^2)$, and
 $W=-H^4+O(H^2)$.
Therefore $
\delta^2_{u\vec H}|T_\alpha|<0 $
for all sufficiently small $\alpha$. Thus $T_\alpha$ is variationally
unstable.
\end{proof}
\begin{remark}\label{torihawwill}
The Clifford tori are also examples of the hypersurface-restricted
Willmore surfaces, or Hawking surfaces, discussed in
Section~\ref{hawsur}. For the spacelike members, 
$H^2-P^2>0$, one has $
K(\cdot,\nu)|_{TT_\alpha}=0$ and $
dH=dP=0$. Hence Lemma~\ref{lem:connection-compatible-initial-data} gives $
s_{\ell_{\vec H}}=0$,
so these tori are time-flat. In particular, every normal direction
generated by $\vec H+\beta\vec H^\star$ is connection-compatible.  Nevertheless, under the dominant energy
condition these surfaces do not satisfy the refined condition $\int_\Sigma \widetilde f-\frac{\lambda}{2} \,d\mu \leq0   $. Since $\chi(T_\alpha)=0$, $H$ is constant, and $J=0$,
equation~\eqref{eq:f-tilde-integral-condition} reduces this condition to
\[
\int_{T_\alpha}(H^2-P^2)\,d\mu
+
4\int_{T_\alpha}\mu\,d\mu
\leq0.
\]
Under the dominant energy condition both terms are nonnegative, while
the first is strictly positive.  Thus the refined condition fails even though the torus is time-flat.
\end{remark}
\begin{remark}
As $\alpha$ varies from the minimal Clifford torus towards the thin-torus
regime, the quantity $
\int_{T_\alpha}|\vec H|^2\,d\mu $
varies continuously from negative values to $+\infty$. Hence there exists
at least one member of the family satisfying
\[
\int_{T_\alpha}|\vec H|^2\,d\mu=16\pi.
\]
Such a torus necessarily has $|\vec H|^2>0$. If the dominant energy
condition holds, then the preceding results give $
\lambda_\beta<0$ for every $|\beta|\leq1$.
Since $T_\alpha$ has genus one, this shows that the sign condition
$\lambda\geq0$ in the spacelike rigidity theorems is essential for
excluding higher-genus directional critical surfaces.
\end{remark}

\appendix

\section{First variation of the Willmore functional}
\label{app:invariant-willmore}

We provide the derivations of the first variation of the mean curvature vector and the resulting coordinate-invariant Willmore equation, explicitly tracking the indefinite signature of the Lorentzian ambient space \((M,g)\).

We begin with the variation of the mean curvature vector, which serves as the fundamental building block for the Willmore energy variation.

\begin{lemma}[First variation of the mean curvature vector]\label{lem:variation-H-app}
Let \(\Sigma\) be a spacelike surface immersed in a \(4\)-dimensional Lorentzian manifold \(M\) with induced metric \(h\). Let  \(V\in\Gamma(N\Sigma)\) be a normal variation field, then
\begin{equation}\label{eq:var_H_app}
    \delta_V \vec{H} = L(V) :=-\Delta^\perp V - \tilde{\chi}(V) + \operatorname{tr}_h \big( \operatorname{Rm}^{M}(\cdot, V)\cdot \big)^\perp ,
\end{equation}
where \(\Delta^\perp\)   is the Laplacian with respect to the normal connection, \(\operatorname{Rm}^M\) is the ambient Riemann curvature tensor, and \(\tilde{\chi}(V) = \sum_{i,j} \langle V, \chi(e_i, e_j) \rangle \chi(e_i, e_j)\) is the  Simons operator.
\end{lemma}
\begin{proof}
Let \(\varphi_s:\Sigma\to M\) be a smooth variation of the immersion. We denote
by \(\partial_s\) the variational vector field along the family
\(\varphi_s(\Sigma)\), so that \(\partial_s|_{s=0}=V\). Fix a point
\(p\in\Sigma\), and choose a local frame \(\{e_1,e_2\}\) tangent to
\(\Sigma\), orthonormal and geodesic at \(p\), so that
\(h_{ij}(p)=\delta_{ij}\) and \((\nabla^\Sigma_{e_i}e_j)_p=0\). We extend this
frame along the variation so that \([\partial_s,e_i]=0\). All computations
below are evaluated at \(s=0\). Differentiating \(\vec{H} = h^{ij}\chi(e_i, e_j)\) with respect to the variation parameter \(s\) at \(s=0\) yields
\begin{equation}\label{eq:var_H_split}
    \delta_V \vec{H} = (\partial_s h^{ij}) \chi(e_i, e_j) + h^{ij} \nabla_{\partial_s}^\perp \chi(e_i, e_j).
\end{equation}
Since \(h^{ik}h_{kj} = \delta^i_j\), differentiating gives \(\partial_s h^{ij} = -\partial_s h_{ij}\) at \(p\). We compute the variation of the spatial metric components
\[ \partial_s h_{ij} = \langle \nabla^{M}_{\partial_s} e_i, e_j \rangle + \langle e_i, \nabla^{M}_{\partial_s} e_j \rangle = 2\langle \nabla^{M}_{e_i} V, e_j \rangle = -2\langle V, (\nabla^{M}_{e_i} e_j)^\perp \rangle. \]
By definition, \((\nabla^{M}_{e_i} e_j)^\perp = -\chi(e_i, e_j)\), so \(\partial_s h_{ij} = 2 \langle V, \chi(e_i, e_j) \rangle\). Inverting the sign for the inverse metric yields \(\partial_s h^{ij} = -2 \langle V, \chi(e_i, e_j) \rangle\). Thus, the first term of \eqref{eq:var_H_split} evaluates to
\begin{equation}\label{eq:step1}
    (\partial_s h^{ij}) \chi(e_i, e_j) = -2 \sum_{i,j=1}^2 \langle V, \chi(e_i, e_j) \rangle \chi(e_i, e_j) = -2\tilde{\chi}(V).
\end{equation}
By definition, \(\chi(e_i, e_j) = -(\nabla^M_{e_i} e_j)^\perp\). We evaluate its normal covariant derivative
$ \nabla_{\partial_s}^\perp \chi(e_i, e_j) = -\big( \nabla^{M}_{\partial_s} \nabla^{M}_{e_i} e_j \big)^\perp$. 
Commuting the ambient derivatives with the standard convention \(\operatorname{Rm}^{M}(X, Y)Z = \nabla^{M}_X \nabla^{M}_Y Z - \nabla^{M}_Y \nabla^{M}_X Z\), we have
\[ -\nabla^{M}_{\partial_s} \nabla^{M}_{e_i} e_j = -\nabla^{M}_{e_i} \nabla^{M}_{\partial_s} e_j - \operatorname{Rm}^{M}(s, e_i) e_j = -\nabla^{M}_{e_i} \nabla^{M}_{e_j} V - \operatorname{Rm}^{M}(V, e_i) e_j. \]
We decompose the ambient derivative of \(V\) using the shape operator. By definition, \(\langle A_V(X), Y \rangle = \langle V, \chi(X, Y) \rangle = -\langle V, \nabla^M_X Y \rangle = \langle \nabla^M_X V, Y \rangle\). Thus, \(\nabla^{M}_{e_j} V = A_V(e_j) + \nabla^\perp_{e_j} V\). Taking the normal projection gives
\[ -\big( \nabla^{M}_{e_i} \nabla^{M}_{e_j} V \big)^\perp = -\big( \nabla^{M}_{e_i} A_V(e_j) \big)^\perp - \big( \nabla^{M}_{e_i} \nabla^\perp_{e_j} V \big)^\perp. \]
Evaluating at \(p\) where \((\nabla_{e_i} e_j)_p = 0\), the second term forms the negative normal Hessian of \(V\). Taking the trace \(h^{ij}\) produces
\begin{equation}\label{eq:step2}
    h^{ij} \nabla_{\partial_s}^\perp \chi(e_i, e_j) = -\Delta^\perp V + \sum_{i=1}^2 \chi(e_i, A_V(e_i)) - \operatorname{tr}_h \big( \operatorname{Rm}^{M}(V, \cdot)\cdot \big)^\perp.
\end{equation}
The middle trace term is precisely the Simons operator: \(\sum_i \chi(e_i, A_V(e_i)) = \tilde{\chi}(V)\). Furthermore, by the antisymmetry of the Riemann tensor, \(-\operatorname{Rm}^{M}(V, e_i)e_i = \operatorname{Rm}^{M}(e_i, V)e_i\). 
Combining \eqref{eq:step1} and \eqref{eq:step2} into \eqref{eq:var_H_split}, we obtain the result.
\end{proof}
With this identity established, we may now compute the variation of the area-constrained Willmore functional.
\begin{proof}[Proof of Proposition \ref{lem:frame-invariant-willmore}]
Let \(V\in\Gamma(N\Sigma)\) be a compactly supported normal variation
field. Using that $\delta_V(d\mu)
=
\langle \vec H,V\rangle\,d\mu$ and the preceding lemma we have 
\begin{align*}
\delta_V\mathcal W=\delta_V\int_\Sigma |\vec H|^2\,d\mu
&=
\int_\Sigma
\left(
2\langle L(V),\vec H\rangle
+
|\vec H|^2\langle\vec H,V\rangle
\right)d\mu \\
&=
\int_\Sigma
\left(
2\langle V,L(\vec H)\rangle
+
|\vec H|^2\langle\vec H,V\rangle
\right)d\mu,
\end{align*}
where we used the formal self-adjointness of \(L\) with respect to the
natural \(L^2\)-pairing on \(N\Sigma\). Using the curvature symmetry $
\operatorname{Rm}^{M}(\cdot,\vec H)
=
-\operatorname{Rm}^{M}(\vec H,\cdot)$,
we have $
L(\vec H)
=
-\Delta^\perp\vec H
-\widetilde\chi(\vec H)
-
\operatorname{tr}_h
\bigl(\operatorname{Rm}^{M}(\vec H,\cdot)\cdot\bigr)^\perp$.
Therefore $$
\delta_V\mathcal W
=
-2\int_\Sigma
\langle \mathfrak W,V\rangle\,d\mu,$$ where $\mathfrak W=\Delta^\perp\vec H+ \operatorname{tr}_h \bigl(\operatorname{Rm}^{M}(\vec H,\cdot)\cdot\bigr)^\perp + \widetilde\chi(\vec H) -\frac12|\vec H|^2\vec H$. This proves \eqref{eq:Willmore-first-variation}.

On the other hand, $
\delta_V|\Sigma|
=
\int_\Sigma
\langle\vec H,V\rangle\,d\mu$. Hence,  full area-constrained
Willmore criticality is equivalent to the existence of a constant
$\lambda\in\mathbb R$ such that $
\delta_V\mathcal W-\lambda\,\delta_V|\Sigma|=0$
for every compactly supported normal variation $V$. Thus
\[
-2\int_\Sigma
\left\langle
\mathfrak W+\frac{\lambda}{2}\vec H,
V
\right\rangle d\mu
=0
\]
for every $V$, and therefore $
\mathfrak W+\frac{\lambda}{2}\vec H=0$.
\end{proof}

\section{Spacetime quasi-local rigidity}\label{appendix}

For the Lorentzian rigidity arguments in Section~\ref{sectionrigi}, we also use
the following positivity and rigidity results for the
Kijowski-Liu-Yau quasi-local energy (\ref{liuyaumass}).
\begin{theorem}[Liu-Yau {\cite[Theorem 1]{liu2003positivity,liu2006positivity}}]
\label{liuyaurigi}
Let $(\Omega,g,K)$ be a compact initial data set satisfying
the dominant energy condition. Suppose $\partial\Omega$ has finitely many components $\Sigma_i$,
each mean-convex, with positive Gaussian curvature and spacelike mean curvature vector.
Then
\[
    \mathcal{E}_{KLY}(\Sigma_\alpha)
=
\frac{1}{8\pi}
\int_{\Sigma_\alpha}
\left(
H_0-|\vec H|
\right)\,d\mu \ge 0.
\]
Moreover, if equality holds for some component,
then $\partial\Omega$ is connected and
$\Omega$ is isometric to a spacelike hypersurface in Minkowski spacetime.  Specifically, $\Omega$ can be isometrically embedded in $\mathbb{R}^{3,1}$ as a spacelike graph $(x, f(x))$ over a spatial domain $\Omega_0 \subset \mathbb{R}^3$, where $f$ is a smooth function on $\Omega_0$ that vanishes on $\partial\Omega_0$.
\end{theorem}

The rigidity of the Kijowski-Liu-Yau energy in Minkowski spacetime is even stronger, this was first observed by Ó Murchadha and Szabados in \cite{murchadha2004comment} and was later fully characterized by Miao, Shi, and Tam in the following result. 
\begin{theorem}[{\cite[Theorem 4.1]{miao2010geometric}}]\label{minkowskirigi}
    Let $\Sigma$ be a closed, connected, smooth, spacelike $2$-surface in Minkowski spacetime $\mathbb{R}^{3,1}$. 
    Suppose $\Sigma$ spans a compact spacelike hypersurface in $\mathbb{R}^{3,1}$.
    If $\Sigma$ has positive Gaussian curvature and a spacelike mean curvature vector, then $
    \mathcal{E}_{KLY} (\Sigma) \geq 0$.
    Moreover, $$\mathcal{E}_{KLY} (\Sigma) = 0$$ if and only if $\Sigma$ lies on a hyperplane in $\mathbb{R}^{3,1}$.
\end{theorem}

\vspace{0.8 cm}

\paragraph*{\emph{Acknowledgements.}}The author would like to thank Jan Metzger for the  helpful discussions about this work.  The author gratefully acknowledges the support of the Austrian Science Fund (FWF) through the project "Geometric Analysis of Biwave Maps" (DOI: 10.55776/P34853). 

\thanks{OpenAI's ChatGPT (GPT-5.6) and Google Gemini 3.1 Pro were used for editorial assistance and as discussion tools for checking arguments and exploring examples. The author reviewed and edited all generated material and takes full responsibility for the content of the manuscript.}

\vspace{0.5 cm}

\vspace{0.15 cm}


\bibliographystyle{amsplain}
\bibliography{Lit_new}

@article {willflat,
    AUTHOR = {Lamm, Tobias and Metzger, Jan and Schulze, Felix},
     TITLE = {Foliations of asymptotically flat manifolds by surfaces of
              {W}illmore type},
   JOURNAL = {Mathematische Annalen},
    VOLUME = {350},
      YEAR = {2011},
    NUMBER = {1},
     PAGES = {1--78},
      ISSN = {0025-5831},
   MRCLASS = {53C12 (53C24)},
  MRNUMBER = {2785762},
MRREVIEWER = {Jesse Ratzkin},
       DOI = {10.1007/s00208-010-0550-2},
       URL = {https://doi.org/10.1007/s00208-010-0550-2},
}

@phdthesis{Friedrich2020,
  author      = {Friedrich, Alexander},
  title       = {Minimizers of generalized Willmore energies and applications in general relativity},
  type        = {doctoral thesis},
  pages       = {100},
  school      = {Universit{\"a}t Potsdam},
  doi       = {10.25932/publishup-48142},
  year        = {2020},
}

@article{Alex,
title = {Concentration of small Hawking type surfaces},
journal = {Differential Geometry and its Applications},
volume = {85},
pages = {101927},
year = {2022},
issn = {0926-2245},
doi = {https://doi.org/10.1016/j.difgeo.2022.101927},
url = {https://www.sciencedirect.com/science/article/pii/S0926224522000808},
author={Friedrich, Alexander},
}

@article {Living,
    AUTHOR = {Szabados,   László Benő},
     TITLE = {Quasi-Local Energy-Momentum and Angular Momentum in GR},
   JOURNAL = {Living Rev. Relativity},
    VOLUME = {7},
      YEAR = {2004},
    NUMBER = {4},
}

@article{Hayward,
  title = {Quasilocal gravitational energy},
  author = {Hayward, Sean A.},
  journal = {Phys. Rev. D},
  volume = {49},
  issue = {2},
  pages = {831--839},
  numpages = {0},
  year = {1994},
  month = {Jan},
  publisher = {American Physical Society},
  doi = {10.1103/PhysRevD.49.831},
  url = {https://link.aps.org/doi/10.1103/PhysRevD.49.831}
}

@article {Hawma,
    AUTHOR = {Hawking, Stephen W.},
     TITLE = {Gravitational radiation in an expanding universe},
   JOURNAL = {J. Mathematical Phys.},
  FJOURNAL = {Journal of Mathematical Physics},
    VOLUME = {9},
      YEAR = {1968},
    NUMBER = {4},
     PAGES = {598--604},
      ISSN = {0022-2488},
   MRCLASS = {83C30},
  MRNUMBER = {3960907},
       DOI = {10.1063/1.1664615},
       URL = {https://doi.org/10.1063/1.1664615},
}

@article{diaz2023local,
  title={Local foliations by critical surfaces of the Hawking energy and small sphere limit},
  author={Peñuela Díaz, Alejandro},
  journal={Classical and Quantum Gravity},
  volume={40},
  number={3},
  pages={035002},
  year={2023},
  publisher={IOP Publishing}
}

@article{liu2003positivity,
  title={Positivity of quasilocal mass},
  author={Liu, Chiu-Chu Melissa and Yau, Shing-Tung},
  journal={Physical review letters},
  volume={90},
  number={23},
  pages={231102},
  year={2003},
  publisher={APS}
}

@article{liu2006positivity,
  title={Positivity of quasi-local mass II},
  author={Liu, Chiu-Chu Melissa and Yau, Shing-Tung},
  journal={Journal of the American Mathematical Society},
  volume={19},
  number={1},
  pages={181--204},
  year={2006}
}

@article{maximo2012hawking,
author = {Máximo, Davi and Nunes, Ivaldo},
year = {2012},
month = {06},
pages = {409–-433},
title = {Hawking mass and local rigidity of minimal two-spheres in
three-manifolds},
volume = {21},
journal = {Communications in Analysis and Geometry},
doi = {10.4310/CAG.2013.v21.n2.a6}
}

@article{barros2017hawking,
  title={Hawking mass and local rigidity of minimal surfaces in three-manifolds},
  author={Barros, A and Batista, R and Cruz, T},
  journal={Communications in Analysis and Geometry},
  volume={25},
  number={1},
  pages={1--23},
  year={2017},
  publisher={International Press of Boston}
}

@article{baltazar2023local,
author = {Baltazar, H. and Barros, Abdênago and Batista, Rondinelle},
year = {2023},
month = {08},
pages = {},
title = {A local rigidity theorem for minimal two-spheres in charged time-symmetric initial data set},
volume = {113},
journal = {Letters in Mathematical Physics},
doi = {10.1007/s11005-023-01713-8}
}

@article{sousa2023charged,
  title={Charged Hawking mass and local rigidity of three-manifolds},
  author={Sousa, Paulo A and Lima, Alexandre B},
  journal={The Journal of Geometric Analysis},
  volume={33},
  number={1},
  pages={11},
  year={2023},
  publisher={Springer}
}

@article{bray2015time,
  title={Time flat surfaces and the monotonicity of the spacetime Hawking mass},
  author={Bray, Hubert L and Jauregui, Jeffrey L},
  journal={Communications in Mathematical Physics},
  volume={335},
  pages={285--307},
  year={2015},
  publisher={Springer}
}

@inproceedings{bray2016time,
  title={Time flat surfaces and the monotonicity of the spacetime Hawking mass II},
  author={Bray, Hubert L and Jauregui, Jeffrey L. and Mars, Marc},
  booktitle={Annales Henri Poincar{\'e}},
  volume={17},
  number={6},
  pages={1457--1475},
  year={2016},
  organization={Springer}
}

@article{bray2007generalized,
  title={Generalized inverse mean curvature flows in spacetime},
  author={Bray, Hubert and Hayward, Sean and Mars, Marc and Simon, Walter},
  journal={Communications in mathematical physics},
  volume={272},
  pages={119--138},
  year={2007},
  publisher={Springer}
}

@article{huisken2001inverse,
  title={The inverse mean curvature flow and the Riemannian Penrose inequality},
  author={Huisken, Gerhard and Ilmanen, Tom},
  journal={Journal of Differential Geometry},
  volume={59},
  number={3},
  pages={353--437},
  year={2001},
  publisher={Lehigh University}
}

@article{miao2010geometric,
  title={On geometric problems related to Brown-York and Liu-Yau quasilocal mass},
  author={Miao, Pengzi and Shi, Yuguang and Tam, Luen-Fai},
  journal={Communications in mathematical physics},
  volume={298},
  number={2},
  pages={437--459},
  year={2010},
  publisher={Springer}
}

@article{murchadha2004comment,
  title={Comment on “Positivity of quasilocal mass”},
  author={Ó Murchadha,  Niall and Szabados, László Benő and Tod, Paul},
  journal={Physical review letters},
  volume={92},
  number={25},
  pages={259001},
  year={2004},
  publisher={APS}
}

@article{lee2025modified,
  title={Modified Hawking mass and rigidity of three-manifolds with boundary},
  author={Lee, Jihyeon and Lee, Sanghun},
  journal={arXiv preprint arXiv:2505.08301},
  year={2025}
}

@article{mars2012stability,
  title={Stability of MOTS in totally geodesic null horizons},
  author={Mars, Marc},
  journal={Classical and Quantum Gravity},
  volume={29},
  number={14},
  pages={145019},
  year={2012},
  publisher={IOP Publishing}
}

@article{frauendiener2001penrose,
  title={On the Penrose inequality},
  author={Frauendiener, J{\"o}rg},
  journal={Physical review letters},
  volume={87},
  number={10},
  pages={101101},
  year={2001},
  publisher={APS}
}

@article{diaz2025rigidity,
  title={Rigidity and positivity of Hawking quasi-local energy on area-constrained critical surfaces},
  author={Peñuela Díaz, Alejandro },
  journal={arXiv preprint arXiv:2507.16588},
  year={2025}
}

@book{choquet2009general,
  title={General relativity and the Einstein equations},
  author={Choquet-Bruhat, Yvonne},
  year={2009},
  publisher={Oxford university press}
}

@article{diaz2026curvature,
  title={Curvature inequalities and rigidity for constant mean curvature and spacetime constant mean curvature surfaces},
  author={Peñuela Díaz, Alejandro},
  journal={arXiv preprint arXiv:2603.16707},
  year={2026}
}

@article{riviere2008analysis,
  title={Analysis aspects of Willmore surfaces},
  author={Riviere, Tristan},
  journal={Inventiones mathematicae},
  volume={174},
  number={1},
  pages={1--45},
  year={2008},
  publisher={Springer}
}

@article{marques2014min,
  title={Min-max theory and the Willmore conjecture},
  author={Marques, Fernando C and Neves, Andr{\'e}},
  journal={Annals of mathematics},
  volume={2},
  pages={683--782},
  year={2014},
  publisher={JSTOR}
}

@article{mondino2010some,
  title={Some results about the existence of critical points for the Willmore functional},
  author={Mondino, Andrea},
  journal={Mathematische Zeitschrift},
  volume={266},
  number={3},
  pages={583--622},
  year={2010},
  publisher={Springer}
}

@article{chen1974some,
  title={Some conformal invariants of submanifolds and their applications},
  author={Chen, Bang-yen},
  journal={Boll. Un. Mat. Ital},
  volume={10},
  number={4},
  pages={380--385},
  year={1974}
}

@article{lan2026analysis,
  title={The Analysis of Willmore Surfaces and Its Generalizations in Higher Dimensions},
  author={Lan, Tian and Martino, Dorian and Rivi{\`e}re, Tristan},
  journal={Journal of Mathematical Study},
  volume={59},
  number={1},
  pages={80--188},
  year={2026}
}

@article{simon1993existence,
  title={Existence of surfaces minimizing the Willmore functional},
  author={Simon, Leon},
  journal={Communications in Analysis and Geometry},
  volume={1},
  number={2},
  pages={281--326},
  year={1993},
  publisher={International Press of Boston}
}

@article{bauer2003existence,
  title={Existence of minimizing Willmore surfaces of prescribed genus},
  author={Bauer, Matthias and Kuwert, Ernst},
  journal={International Mathematics Research Notices},
  volume={2003},
  number={10},
  pages={553--576},
  year={2003},
  publisher={OUP}
}

@inproceedings{lamm2013minimizers,
  title={Minimizers of the Willmore functional with a small area constraint},
  author={Lamm, Tobias and Metzger, Jan},
  booktitle={Annales de l'Institut Henri Poincar{\'e} C, Analyse non lin{\'e}aire},
  volume={30},
  number={3},
  pages={497--518},
  year={2013},
  organization={Elsevier}
}

@article{lamm2010small,
  title={Small surfaces of Willmore type in Riemannian manifolds},
  author={Lamm, Tobias and Metzger, Jan},
  journal={International mathematics research notices},
  volume={2010},
  number={19},
  pages={3786--3813},
  year={2010},
  publisher={OUP}
}

@article{ma2008spacelike,
  title={Spacelike Willmore surfaces in 4-dimensional Lorentzian space forms},
  author={Ma, Xiang and Wang, Peng},
  journal={Science in China Series A: Mathematics},
  volume={51},
  number={9},
  pages={1561--1576},
  year={2008},
  publisher={Springer}
}

@article{allas1996conformal,
  title={Conformal geometry of surfaces in Lorentzian space forms},
  author={Alias Linares, Luis Jose and Palmer, Bennett},
  journal={Geometriae Dedicata},
  volume={60},
  number={3},
  pages={301--315},
  year={1996},
  publisher={Springer}
}

@article{ikoma2016embedded,
  title={Embedded area-constrained Willmore tori of small area in Riemannian three-manifolds I: minimization},
  author={Ikoma, Norihisa and Malchiodi, Andrea and Mondino, Andrea},
  journal={Proceedings of the London Mathematical Society},
  volume={115},
  number={3},
  pages={502--544},
  year={2017},
  publisher={Wiley Online Library}
}

@article{ikoma2017embedded,
  title={Embedded area-constrained Willmore tori of small area in Riemannian three-manifolds, II: Morse Theory},
  author={Ikoma, Norihisa and Malchiodi, Andrea and Mondino, Andrea},
  journal={American Journal of Mathematics},
  volume={139},
  number={5},
  pages={1315--1378},
  year={2017},
  publisher={Johns Hopkins University Press}
}

@article{penuelafol,
  author  = {Peñuela Díaz, Alejandro}, 
  title   = {Foliations by critical surfaces of the Hawking energy in asymptotically flat initial data sets},
  journal = {Communications in Analysis and Geometry},
  volume  = {34}, 
  number  = {2}, 
  pages   = {505--570},
  year    = {2026}, 
  doi     = {10.4310/CAG.260531221657} 
}

@article{geroch1973energy,
  title={Energy extraction},
  author={Geroch, Robert},
  journal={Annals of the New York Academy of Sciences},
  volume={224},
  number={1},
  pages={108--117},
  year={1973},
  publisher={Blackwell Publishing Ltd Oxford, UK}
}
\end{document}